\documentclass[11pt]{article}

\usepackage{amssymb, amsmath, amsthm}
\usepackage[misc]{ifsym}
\usepackage{mathtools, algorithm, algpseudocode, paralist}
\usepackage{epsfig} 
\usepackage{epstopdf} 
\usepackage{geometry} 
\usepackage{graphicx, color, xcolor, booktabs}
\usepackage{subcaption}
\usepackage{fancyhdr, footmisc}  
\usepackage{nicefrac}
\usepackage{titlesec}
\usepackage{bbding,pifont}
\usepackage{enumerate}
\numberwithin{equation}{section}    
\usepackage{latexsym}

\newtheorem{definition}{Definition}[section]
\newtheorem{remark}{Remark}[section]
\newtheorem{theorem}{Theorem}[section]
\newtheorem{lemma}{Lemma}[section]

\newtheorem{prop}{Proposition}[section]

\newcommand{\abs}[1]{\lvert#1\rvert}

\usepackage[bookmarksnumbered, colorlinks, plainpages]{hyperref}
\hypersetup{urlcolor=blue, citecolor=red}

\title{\Large A second order positivity-preserving scheme for the nonlocal Cahn-Hilliard system with variable mobility and logarithmic Flory-Huggins potential}  

\author{
	Yuanyi Sheng \thanks{School of Mathematical Sciences, Beijing Normal University, Beijing 100875, PR China  (yuanyisheng@mail.bnu.edu.cn)}
	\and
	Zhengru Zhang\thanks{Laboratory of Mathematics and Complex Systems, Ministry of Education and School of Mathematical Sciences, Beijing Normal University, Beijing 100875, PR China (zrzhang@bnu.edu.cn, Corresponding author)}
}    
\date{}

\begin{document}

\maketitle
\begin{abstract}
	A second order finite difference scheme is analyzed for the nonlocal Cahn-Hilliard system with variable mobility and the singular Flory-Huggins logarithmic potential. The temporal discretization employs a modified Crank-Nicolson formulation, while the mobility is treated explicitly to guarantee ellipticity and reduce computational cost. To strictly enforce the positivity-preserving property at the discrete level, a nonlinear regularization term is added into the numerical scheme. Rigorous analysis shows that the scheme is uniquely solvable and stable under a modified numerical energy. The well-known analytical challenge associated with the variable mobility has to be handled carefully. To overcome this, a convergence framework is constructed based on two non-standard techniques: (1) a higher-order asymptotic expansion (extending up to the third order in time and fourth order in space) to retain a sufficiently order of accuracy; (2) a two-stage error strategy, wherein a rough error estimate first guarantees the discrete boundedness of the mobility and nonlinear terms, followed by a refined error analysis that yields the optimal convergence rate. Numerical experiments validate the theoretical results and demonstrate the robustness of the proposed scheme.

	\bigskip
	
	
	\bigskip
	
\end{abstract}

\section{Introduction}
This work is concerned with the rigorous theoretical analysis of a second order numerical scheme for the nonlocal Cahn-Hilliard (NCH) system with variable mobility and singular Flory-Huggins potential. Let $\Omega\subset\mathbb{R}^d$ ($d=2,3$) be a smooth bounded spatial domain. The phase field variable $\phi(\mathbf{x},t)$, representing the normalized concentration difference of a binary mixture, is assumed to restrict within the physical domain $\mathcal{D} := (-1,1)$ at a point-wise level. The total free energy functional $E(\phi)$ comprises a local contribution $E_1(\phi)$ and a long-range interaction energetic term $E_2(\phi)$, defined respectively as
\begin{equation}
	\begin{aligned}
		&E(\phi)=E_1(\phi)+E_2(\phi),\label{energy}\\
		&E_1(\phi)=\int_\Omega (1+\phi)\ln(1+\phi)+(1-\phi)\ln(1-\phi)-\frac{\theta_{0}}{2}\phi^2+\frac{\varepsilon^2}{2}\abs{\nabla \phi}^{2}\,\mathrm{d}\mathbf{x},\\
		&E_2(\phi)=\frac{\sigma}{2}\int_\Omega \abs{(-\Delta)^{-\frac{1}{2} }(\phi-\bar{\phi})}^2\,\mathrm{d}\mathbf{x},\qquad\overline{\phi}=\frac{1}{\abs{\Omega}}\int_{\Omega}\phi\,\mathrm{d}\mathbf{x},
	\end{aligned}
\end{equation}
where $\theta_{0}$ and $\varepsilon$ denote positive physical parameters, and $\bar{\phi} \in \mathcal{D}$ is the spatial average \cite{Cahn1996TheCE,copetti1992numerical,DoiSoft}. The parameter $\sigma\ge 0$ scales the intensity of nonlocal interactions. In the localized limit $\sigma = 0$, the functional \eqref{energy} reduces to the classical local Cahn-Hilliard (CH) definition. For $\sigma>0$, the long range constraint, formulated via the inverse fractional Laplacian, yields the Ohta-Kawasaki free energy, which serves as a standard mathematical model for diblock copolymer self-assembly and microscopic phase separation dynamics \cite{bates1990block,Nishiura1995SomeMA,ohta1995dynamics,1986Equilibrium}.

 The evolutionary equation of the NCH system is governed by $H^{-1}(\Omega)$ gradient flow of the variational energy \eqref{energy}:
   \begin{equation}\label{pdeeq}
 	\begin{aligned}
 		\frac{\partial \phi}{\partial t} &=\nabla\cdot(\mathcal{M}(\phi)\nabla \mu),\\
 		\mu &=\ln (1+\phi)-\ln (1-\phi)-\theta_{0}\phi-{\varepsilon}^{2}\Delta\phi+\sigma (-\Delta)^{-1}(\phi-\bar{\phi}),
 	\end{aligned}
 \end{equation}
 where $\mu$ is the chemical potential and $\mathcal{M}(\phi)>0$ denotes the variable mobility. A direct variational derivation confirms that the system complies with the energy dissipation law,  $\frac{\mathrm{d}}{\mathrm{d}t}E(\phi) = -\int_\Omega \mathcal{M}(\phi)|\nabla \mu|^2 \mathrm{d}\mathbf{x} \le 0$. This dissipative structure constitutes the mathematical foundation for both the continuous analysis and the construction of positivity-preserving discrete approximations. A distinguishing feature of the system \eqref{pdeeq} is the concentration-dependent nature of the mobility $\mathcal{M}(\phi)$. Although a constant mobility is commonly assumed in existing numerical literature, it often fails to capture the realistic transport dynamics where the local concentration migration velocity scales with the thermodynamic force gradient \cite{bird1987dynamics,doi1988theory}. This feature was intrinsically recognized in the foundational phase field derivations \cite{cahn1971spinodal}. A physically consistent and mathematically sound archetype is given by $\mathcal{M}(\phi)=1-\lambda\phi^{2}$ with $0\le\lambda\le 1$ \cite{cahn1994overview}. In the degenerate limit $\lambda = 1$, the interfacial diffusion becomes dominant as bulk diffusion vanishes upon phase separation \cite{elliott1996cahn}. Numerical evidence indicates that degenerate or near-degenerate mobilities substantially improves numerical resolution and physical fidelity for multiphase flow simulations within the framework of fixed grid computations \cite{zhang2010cahn}.
 
From an analytical perspective, the coexistence of a variable mobility $\mathcal{M}(\phi)$ and a singular logarithmic potential introduces severe mathematical challenges that drastically distinguish this model from its constant mobility counterparts. Numerical methods utilizing polynomial regularizations of the Flory-Huggins potential or restricted to constant mobility \cite{chen19b,du2018stabilized,guan2017convergence} fail to simulate complex coarsening dynamics accurately. Recently, we proposed a first order, positivity-preserving scheme \cite{sheng2026positivity} for the NCH system \eqref{pdeeq}. While this work provided a robust framework for handling singular logarithmic potentials and accommodated a broad class of mobility functions, the first-order temporal accuracy remains limit in efficiency and resolution in complex, long time coarsening processes. This work is therefore devoted to the construction and analysis of a second order accurate scheme while retaining the essential properties of the model. The temporal discretization employs a modified Crank-Nicolson (CN) approximation. To enforce the strict ellipticity of the operator and ensure unique solvability, a second order explicit extrapolation is deployed for the mobility, which simultaneously reduces the computational cost. Instead of a direct CN approximation, we propose a modified form for the logarithmic term, $\frac{G(1\pm\phi^{n+1})-G(1\pm\phi^{n})}{\phi^{n+1}-\phi^{n}}$, which, when made inner product with the discrete time derivative, precisely yields the corresponding nonlinear energy difference. This thereby ensures the smooth progress of our energy stability analysis. Additionally, a nonlinear logarithmic artificial regularization term is introduced. It is revealed that the singular nature of the logarithmic function prevents the numerical solution of the phase variable reaching the limit singular values, so that the positivity preserving property could be theoretically proved. Finally, the nonlocal term is discretized through a second order temporal averaging, so that it does not appear in the form of modified energy in the theoretical analysis of energy stability.

The convergence analysis of the proposed second order scheme represents the core theoretical contribution of this work, primarily due to the intense nonlinear coupling between the variable mobility and the singular potential. The classical convergence analysis established for constant mobility systems \cite{chen2022modified,chen19b} are no longer applicable. To address this challenge, we establish a comprehensive convergence framework based on a combination of higher-order consistency expansions with the rough and refined error estimates. Specifically, we first introduce supplementary fields and perform a higher-order consistency analysis to ensure that error estimates maintain a sufficiently order. Next, by using rough error estimates and inverse inequalities, we estimate the discrete $\ell^{\infty}$ norm of the numerical solution at the next time step. This crucial step successfully establishes a priori $\ell^\infty$ bound for the numerical solution, thereby theoretically providing the necessary regularity both for the variable mobility and the singular logarithmic functions. Finally, through a refined error analysis, we obtain the desired convergence rate, confirming the optimal performance of our proposed scheme. It is particularly noteworthy that our convergence result applies to a broad and physically relevant class of mobility functions, demonstrating significant generality and practical importance.

 The rest of the article is organized as follows. Section 2 reviews the finite difference spatial discretization and presents the detailed formulation of our second order numerical scheme. The unique solvability, positivity preserving analysis and energy stability estimate are proved in Section 3. Section 4 is dedicated to the convergence analysis of the scheme, utilizing a methodology that combines higher-order consistency analysis with rough and refined error estimates. Finally, Section 5 presents various numerical results to validate the accuracy and performance of the proposed scheme, and Section 6 concludes this paper.

\section{The second order accurate numerical scheme}
\subsection{The finite difference spatial discretization}
We employ a standard centered finite difference framework for spatial discretization. Without loss of generality, let the computational domain be defined as $\Omega=(0,L)^3$ with $L>0$, partitioned by a uniform grid spacing $h = L/N$ for $N \in \mathbb{N}$ (the extension to non-uniform meshes is straightforward). We define the index sets $W:=\{\, p_{i+1/2}\, |\, i\in\mathbb{Z}\,\}$ and $V:=\{\, p_i=(i-1/2)\cdot h\, |\, i\in\mathbb{Z}\,\}$. The three-dimensional periodic discrete spaces for cell-centered and face-centered grid functions are defined respectively as
\begin{align*}
	&\mathcal{V}_{\mathrm{per}}:=\{\,u:V\times V\times V\to\mathbb{R}\, |\, u_{i,j,k}=u_{i+\ell N,j+m N,k+n N},\,\forall i,j,k,\ell,m,n\in \mathbb{Z}\,\},\\
	&\mathcal{W}_{\mathrm{per}}^{x}:=\{\,u:W\times V\times V\to\mathbb{R}\, |\, u_{i+\frac{1}{2},j,k}=u_{i+\frac{1}{2}+\ell N,j+m N,k+n N},\,\forall i,j,k,\ell,m,n\in \mathbb{Z}\,\},
\end{align*}
where $u_{i,j,k}=u(p_i,p_j,p_k)$. The directional face-centered spaces $\mathcal{W}_{\mathrm{per}}^{y}$ and $\mathcal{W}_{\mathrm{per}}^{z}$ are constructed analogously. The center-to-face and face-to-center averaging and difference operators, denoted by $A_{x},\, D_{x}:\mathcal{V}_{\mathrm{per}}\to\mathcal{W}_{\mathrm{per}}^{x}$, $a_{x},\, d_{x}:\mathcal{W}_{\mathrm{per}}^{x}\to\mathcal{V}_{\mathrm{per}}$, are formulated as
\begin{align*}
	&A_{x}u_{i+\frac{1}{2},j,k} := 
	\frac{1}{2}(u_{i+1,j,k}+u_{i,j,k}),\  
	&D_{x}u_{i+\frac{1}{2},j,k} := 
	\frac{1}{h}(u_{i+1,j,k}-u_{i,j,k}),\\
	&a_{x}u_{i,j,k} := 
	\frac{1}{2}(u_{i+\frac{1}{2},j,k}+u_{i-\frac{1}{2},j,k}),\  
	&d_{x}u_{i,j,k} := 
	\frac{1}{h}(u_{i+\frac{1}{2},j,k}-u_{i-\frac{1}{2},j,k}),
\end{align*}
with parallel definitions extended to the $y$- and $z$-directions.
By introducing $\vec{\mathcal{W}}_{\mathrm{per}} := \mathcal{W}_{\mathrm{per}}^{x}\times\mathcal{W}_{\mathrm{per}}^{y}\times\mathcal{W}_{\mathrm{per}}^{z}$, the discrete gradient $\nabla_{h}$ and the discrete divergence $\nabla_{h}\cdot$ are represented point-wise by
\begin{align*}
	\nabla_{h}:\mathcal{V}_{\mathrm{per}}\to\vec{\mathcal{W}}_{\mathrm{per}},\,\qquad
	&\nabla_{h}u_{i,j,k}=(D_{x}u_{i+1/2,j,k},D_{y}u_{i,j+1/2,k},D_{z}u_{i,j,k+1/2}),\\
	\nabla_{h}\cdot:\vec{\mathcal{W}}_{\mathrm{per}}\to\mathcal{V}_{\mathrm{per}},\qquad
	&\nabla_{h}\cdot\vec{w}_{i,j,k}=d_{x}w_{i,j,k}^{x}+d_{y}w_{i,j,k}^{y}+d_{z}w_{i,j,k}^{z},\quad
	\vec{w}=(w^x,w^y,w^z)\in\vec{\mathcal{W}}_{\mathrm{per}}.
\end{align*}
The standard discrete Laplacian operator $\Delta_{h}:\mathcal{V}_{\mathrm{per}}\to\mathcal{V}_{\mathrm{per}}$ is subsequently generated by the composition $\Delta_{h}u := \nabla_{h}\cdot(\nabla_{h}u)$.

To naturally accommodate the variable mobility coefficient, let $g$ be a periodic scalar function defined at the face-centered grid points. For any vector $\vec{w}\in\vec{\mathcal{W}}_{\mathrm{per}}$, the product $g\vec{w}\in\vec{\mathcal{W}}_{\mathrm{per}}$ is defined via point-wise multiplication, and the variable coefficient discrete divergence operator $\nabla_{h}\cdot(g\nabla_{h}):\mathcal{V}_{\mathrm{per}}\to\mathcal{V}_{\mathrm{per}}$ is explicitly specified as
\begin{equation*}
\nabla_{h}\cdot(g\nabla_{h}u)_{i,j,k}=d_x(g\cdot D_xu)_{i,j,k}+d_y(g\cdot D_yu)_{i,j,k}+d_z(g\cdot D_zu)_{i,j,k}.
\end{equation*}

The discrete inner products over the cell-centered and face-centered spaces are defined as
\begin{align*}
	&\left\langle{u,\xi}\right\rangle_{\Omega} :=  h^3\sum\limits_{i,j,k=1}^{N}u_{i,j,k}\,\xi_{i,j,k},\quad u,\,\xi\in\mathcal{V}_{\mathrm{per}},\qquad
	[u, \xi]_{\mathrm{x}} := \left\langle{a_x(u\xi),1}\right\rangle_{\Omega},\quad u,\, \xi\in\mathcal{W}_{\mathrm{per}}^{x},\\
	&[\vec{w}_1, \vec{w}_2]_{\Omega} := [w_1^x, w_2^x]_{\mathrm{x}}+[w_1^y, w_2^y]_{\mathrm{y}}+[w_1^z, w_2^z]_{\mathrm{z}},\qquad\ \vec{w}_{k}=(w_{k}^x,w_{k}^y,w_{k}^z)\in\vec{\mathcal{W}}_{\mathrm{per}},\quad k=1,2.
\end{align*}
$[u, \xi]_{\mathrm{y}}$ and $[u, \xi]_{\mathrm{z}}$ in $\mathcal{W}_{\mathrm{per}}^{y}$ and $\mathcal{W}_{\mathrm{per}}^{z}$ can be analogously formulated. Subsequently, the induced discrete $\ell^p$ and discrete Sobolev norms for cell-centered grid functions $u\in\mathcal{V}_{\mathrm{per}}$ are specified by
\begin{align*}
	&\|u\|_2^2 := \left\langle{u,u}\right\rangle_{\Omega},\quad
	\|\nabla_{h}u\|_2^2 =[D_x u,D_x u]_{\mathrm{x}}+[D_y u,D_y u]_{\mathrm{y}}+[D_z u,D_z u]_{\mathrm{z}},\quad
	\|u\|_{\infty} := \max_{1\le i,j,k\le N}\abs{u_{i,j,k}},\\
	&\|u\|_p^p := \left\langle{\abs{u}^p,1}\right\rangle_{\Omega},\quad
	\|\nabla_{h} u\|_{p}^{p} = [\abs{D_x u}^p,1]_{\mathrm{x}}+[\abs{D_y u}^p,1]_{\mathrm{y}}+[\abs{D_z u}^p,1]_{\mathrm{z}},\quad\ 1\le p<\infty,
\end{align*}
alongside the higher-order spatial norms
\begin{equation*}
\| u\|_{H_h^1}^2 := \| u\|_2^2+\|\nabla_{h} u\|_2^2,\qquad 
\| u\|_{H_h^2}^2 := \| u\|_2^2+\|\nabla_{h} u\|_2^2+\|\Delta_{h} u\|_2^2,\qquad \forall u\in\mathcal{V}_{\mathrm{per}}.
\end{equation*}

The foundational summation-by-parts properties are recalled below, which ensure the conservation and dissipation attributes at the discrete level.
\begin{lemma}[\cite{wise2009energy}]\label{sum by part}
	Let $g$ denote a periodic scalar function evaluated at the face-centered points. For any functions $v,\, u\in\mathcal{V}_{\mathrm{per}}$ and $\vec{w}\in\vec{\mathcal{W}}_{\mathrm{per}}$, the following summation-by-parts identities hold:
	\begin{equation*}
	\left\langle{v,\nabla_{h}\cdot\vec{w}\ }\right\rangle_{\Omega}=-\left[\nabla_{h}v,\vec{w}\ \right]_{\Omega},\quad \left\langle{v,\nabla_{h}\cdot(g\nabla_{h} u)}\right\rangle_{\Omega}=-[\nabla_{h}v,g\nabla_{h} u]_{\Omega}. 
	\end{equation*}
\end{lemma}

To facilitate the subsequent convergence analysis, a discrete analogue of the Sobolev space $ H_{\mathrm{per}}^{-1}(\Omega)$ must be precisely constructed. We introduce the mean-zero grid function space
\begin{equation*}
\mathring{\mathcal{V}}_{\mathrm{per}} := \Big\{\,\nu\in\mathcal{V}_{\mathrm{per}}\, \Big| \, \bar{\nu}
=\frac{1}{\abs{\Omega}}\langle{\nu,1}\rangle_{\Omega}=0\,\Big\}.
\end{equation*}
For any $\phi\in\mathcal{V}_{\mathrm{per}}$, it holds that $\phi-\bar{\phi}\in\mathring{\mathcal{V}}_{\mathrm{per}}$. 
Given a strictly positive, periodic face-centered weight function $\mathcal{M}$ (which will represent the explicitly extrapolated variable mobility), we define the discrete elliptic operator $\mathcal{L}_{\mathcal{M}}(\psi) := -\nabla_{h}\cdot(\mathcal{M}\nabla_{h}\psi)$. For any $\phi \in \mathring{\mathcal{V}}_{\mathrm{per}}$, there exists a unique solution $\psi\in\mathring{\mathcal{V}}_{\mathrm{per}}$ to the equation $\mathcal{L}_{\mathcal{M}}(\psi) = \phi$. Consequently, a discrete $\mathcal{L}_{\mathcal{M}}^{-1}$ bilinear form is defined on $\mathring{\mathcal{V}}_{\mathrm{per}} \times \mathring{\mathcal{V}}_{\mathrm{per}}$ as
\begin{equation*}
	\langle{\phi_{1},\phi_{2}}\rangle_{\mathcal{L}_{\mathcal{M}}^{-1}} := \left[\mathcal{M}\nabla_{h}\psi_{1},\nabla_{h}\psi_{2}\right]_{\Omega},\quad\phi_{1},\,\phi_{2}\in\mathring{\mathcal{V}}_{\mathrm{per}},
\end{equation*}
where $\mathcal{L}_{\mathcal{M}}(\psi_{i})=\phi_{i}$ for $i=1,2$.

\begin{lemma}[\cite{wang2011energy}]
	The bilinear mapping $\langle{\cdot,\cdot}\rangle_{\mathcal{L}_{\mathcal{M}}^{-1}}$ defines a valid inner product on $\mathring{\mathcal{V}}_{\mathrm{per}}$ and satisfies
	\begin{equation*}
\langle{\phi_{1},\phi_{2}}\rangle_{\mathcal{L}_{\mathcal{M}}^{-1}}=\left\langle{\phi_{1},\mathcal{L}_{\mathcal{M}}^{-1}(\phi_{2})}\right\rangle_{\Omega}=\left\langle{\mathcal{L}_{\mathcal{M}}^{-1}(\phi_{1}),\phi_{2}}\right\rangle_{\Omega},\quad\phi_{1},\,\phi_{2}\in\mathring{\mathcal{V}}_{\mathrm{per}}.
	\end{equation*}
	Accordingly, the norm induced by this inner product is given by
	\begin{equation*}
	\|\phi\|_{\mathcal{L}_{\mathcal{M}}^{-1}} := \sqrt{\langle{\phi,\phi}\rangle_{\mathcal{L}_{\mathcal{M}}^{-1}}},\quad \forall\phi\in\mathring{\mathcal{V}}_{\mathrm{per}}.
	\end{equation*}
\end{lemma}
\begin{remark}
In the constant coefficient limit where $\mathcal{M}\equiv 1$, this framework simplifies to the standard discrete $H^{-1}_{\mathrm{per}}$ inner product and norm, denoted as  $\langle{\cdot,\cdot}\rangle_{\mathcal{L}_{\mathcal{M}}^{-1}}=:\langle{\cdot,\cdot}\rangle_{-1,h}$, $\|\cdot\|_{\mathcal{L}_{\mathcal{M}}^{-1}}=:\|\cdot\|_{-1,h}$,
	\begin{equation*}
	\langle{u,v}\rangle_{-1,h} := \langle{u,(-\Delta_{h})^{-1}v}\rangle_{\Omega}=\langle{(-\Delta_{h})^{-1}u,v}\rangle_{\Omega},\qquad\|u\|_{-1,h} := \sqrt{\langle{u,u}\rangle_{-1,h}}.
	\end{equation*}
\end{remark}

The following discrete Sobolev inequalities for grid functions have been derived in previous works, which play significant roles in our proof of convergence analysis.
\begin{lemma}[\cite{chen2022preconditioned,wang2011energy}]\label{norm4}
	For any periodic grid function $u$ over 3D cell-centered mesh points, the associated discrete Sobolev inequalities are satisfied:
	\begin{equation*}
		\|u\|_{4}\le C\|u\|_{H_{h}^{1}},\qquad\|\nabla_{h}u\|_{4}\le C\|\Delta_{h}u\|_{2}.
	\end{equation*}
	Positive constant $C$ only depends on the domain $\Omega$.
\end{lemma}

\subsection{A second order numerical scheme}
The second order accurate (in time) numerical scheme is constructed as follows: given $\phi^{n},\,\phi^{n-1}\in\mathcal{V}_{\mathrm{per}}$, find $\phi^{n+1},\,\mu^{n+\frac{1}{2}}\in\mathcal{V}_{\mathrm{per}}$ such that
\begin{equation}\label{2order}
	\begin{aligned}
		\frac{\phi^{n+1}-\phi^{n}}{\tau}
		&=\nabla_{h}\cdot(\breve{\mathcal{M}}^{n+\frac{1}{2}}\nabla_{h}\mu^{n+\frac{1}{2}}),\\
		\mu^{n+\frac{1}{2}}
		&=\frac{G(1+\phi^{n+1})-G(1+\phi^{n})}{\phi^{n+1}-\phi^{n}}
		  +\frac{G(1-\phi^{n+1})-G(1-\phi^{n})}{\phi^{n+1}-\phi^{n}}\\
		&\quad-\theta_{0}\breve{\phi}^{n+\frac{1}{2}}
		   -\varepsilon^2\Delta_{h}\hat{\phi}^{n+\frac{1}{2}}
		   +\sigma(-\Delta_{h})^{-1}\left(\phi^{n+\frac{1}{2}}-\overline{\phi^{n+\frac{1}{2}}}\right)\\
		&\quad+\tau \, \Big(\ln(1+\phi^{n+1})-\ln(1+\phi^{n})-\ln(1-\phi^{n+1})+\ln(1-\phi^{n})\Big),
	\end{aligned}
\end{equation}
where
\begin{equation*}
G(x)=x\ln x,\;
\breve{\phi}^{n+\frac{1}{2}}=\frac{3}{2}\phi^{n}-\frac{1}{2}\phi^{n-1},\;
\hat{\phi}^{n+\frac{1}{2}}=\frac{3}{4}\phi^{n+1}+\frac{1}{4}\phi^{n-1},\;
\phi^{n+\frac{1}{2}}=\frac{1}{2}\phi^{n+1}+\frac{1}{2}\phi^{n}.
\end{equation*}
The discrete mobility is evaluated at the face-centered points
\begin{equation}\label{mob regu}
		\breve{\mathcal{M}}^{n+\frac{1}{2}}=\Big((A_{h}(\frac{3}{2}\mathcal{M}^{n}-\frac{1}{2}\mathcal{M}^{n-1}))^{2}+\tau^{6}\Big)^{\frac{1}{2}}.
\end{equation}
This definition guarantees both second order accuracy in time and a point-wise positive lower bound, $\breve{\mathcal{M}}^{n+1/2}\ge\tau^3$, at all face-centered points. An operator and the corresponding norm associated with the discrete mobility are defined below.
\begin{definition}
	For any $\phi\in\mathring{\mathcal{V}}_{\mathrm{per}}$, there exists a unique $\psi\in\mathring{\mathcal{V}}_{\mathrm{per}}$ that solves
	\begin{equation*}
		\mathcal{L}_{\breve{\mathcal{M}}^{n+1/2}}(\psi) := -\nabla_{h}\cdot(\breve{\mathcal{M}}^{n+1/2}\nabla_{h}\psi)=\phi.
	\end{equation*}
	In turn, the corresponding norm could be introduced:
	\begin{equation*}
		\|\phi\|_{\mathcal{L}_{\breve{\mathcal{M}}^{n+1/2}}^{-1}}=\sqrt{\left\langle{\phi,\mathcal{L}_{\breve{\mathcal{M}}^{n+1/2}}^{-1}(\phi)}\right\rangle_{\Omega}}.
	\end{equation*}
\end{definition}

\begin{lemma}[\cite{chen19b}]\label{lemma mob 2nd}
	Assume that $\mathcal{M}(x)\ge\mathcal{M}_{0}>0$ for all $x\in[-1,1]$. Let $u_{1},\, u_{2}\in\mathcal{V}_{\mathrm{per}}$ be two periodic grid functions satisfying $u_{1}-u_{2}\in\mathring{\mathcal{V}}_{\mathrm{per}}$ and $\|u_{1}\|_{\infty},\,\|u_{2}\|_{\infty}\le M_{h}$, where $M_{h}$ may depend upon $h$. Then, we have the following estimate
	\[
	\|\mathcal{L}_{\breve{\mathcal{M}}^{n+1/2}}^{-1}(u_{1}-u_{2})\|_{\infty}\le C.
	\]
	Positive constant $C$ depends only on $M_{h}$, $\mathcal{M}_{0}$, $h$ and the domain $\Omega$. 
\end{lemma}
\begin{remark}
	Since the proposed numerical scheme \eqref{2order} is a three-level method, an initialization at the ``ghost" level $t^{-1}$ is required. For numerical convenience, we adopt the extrapolation $\phi^{-1}=\phi^{0}$ directly. Such a choice simplifies the energy stability analysis, and the second order temporal accuracy is still guaranteed.
\end{remark}

\section{Positivity-preserving analysis and energy stability estimate}
Given the underlying $H^{-1}$ gradient flow structure, the proposed numerical scheme inherently conserves mass, provided the solution exists. To ensure the well-posedness of the scheme, it is crucial to guarantee that the numerical solution remains within the physically admissible range at each grid point, specifically $-1<\phi_{i,j,k}^{n+1}<1$. This pointwise boundedness property will be rigorously verified in the following theorem. As a preparatory step, we now introduce several smooth auxiliary functions associated with the logarithmic potential $G(x)=x\ln x$, which will be instrumental in the subsequent analysis. For a fixed $a>0$ and for all $x>0$, define
\begin{equation*}
		H_a^1(x):=\frac{G(x)-G(a)}{x-a},\quad
		H_a^0(x):=\int_{a}^{x}H_a^1(t)\,\mathrm{d}t,\quad
		H_a^2(x):=(H_a^1)'(x)=(H_a^0)''(x).
\end{equation*}
Through straightforward computation, above functions can be shown to satisfy the following estimates.
\begin{lemma}\label{G prop}
	Suppose that $a>0$ is fixed.
	\begin{enumerate}[$1.$]
		\item $H_a^1(x)$ is an increasing function of $x$, and $H_a^1(x)\le H_a^1(a)=\ln a+1$, $\forall\ 0<x\le a$;
		\item $H_a^0(x)$ is a convex function of $x$ in the domain $[0,\infty)$;
		\item $H_a^2(x)=\frac{G'(x)(x-a)-(G(x)-G(a))}{(x-a)^2}=\frac{1}{2\tilde{\zeta}}$, for some $\tilde{\zeta}$ between $a$ and $x$, and $H_a^2(x)\ge 0$, $\forall\ x>0$.
	\end{enumerate}
\end{lemma}

The numerical scheme satisfies a positivity-preserving property, as formulated below.
\begin{theorem}\label{positive theo}
	Assume that the mobility function satisfies $\mathcal{M}(x)\geq\mathcal{M}_{0}>0$  for all $x\in[-1,1]$. Given $\phi^{k}\in\mathcal{V}_{\mathrm{per}}$ with $\|\phi^{k}\|_{\infty}< 1,\, k=n,\, n-1$ , and $\abs{\overline{\phi^{n}}}=\abs{\overline{\phi^{n-1}}}<1$. Then, there exists a unique solution $\phi^{n+1}\in\mathcal{V}_{\mathrm{per}}$ to \eqref{2order}, such that $\phi^{n+1}-\overline{\phi^{n}}\in\mathring{\mathcal{V}}_{\mathrm{per}}$ and $\|\phi^{n+1}\|_{\infty}<1$.
\end{theorem}
\begin{proof}
	If a valid numerical solution exists, it must satisfy the mass conservation identity, namely $\phi^{n+1}-\overline{\phi^{n}}\in\mathring{\mathcal{V}}_{\mathrm{per}}$. For the sake of simplicity, set $\gamma := \overline{\phi^{n}}$. The numerical solution of \eqref{2order} is a minimizer of the following strictly convex discrete energy functional
	\begin{equation}\label{convex functional_2nd}
		\begin{split}
			\Gamma^{n}(\phi)
			& := \frac{1}{2\tau}\|\phi-\phi^{n}\|_{\mathcal{L}_{\breve{\mathcal{M}}^{n+1/2}}^{-1}}^2
			+\left\langle{H_{1+\phi^{n}}^{0}(1+\phi),1}\right\rangle_{\Omega}
			-\left\langle{H_{1-\phi^{n}}^{0}(1-\phi),1}\right\rangle_{\Omega}
			+\frac{3{\varepsilon}^2}{8}\|\nabla_{h}\phi\|_2^2\\
			&\hspace{1cm}
			+\tau\big(\left\langle{1+\phi,\ln (1+\phi)}\right\rangle_{\Omega}
			+\left\langle{1-\phi,\ln (1-\phi)}\right\rangle_{\Omega}\big)
			+\frac{\sigma}{4}\|\phi-\bar{\phi}\|_{-1,h}^2
			+\left\langle{\phi,S^{n}}\right\rangle_{\Omega},
		\end{split}
	\end{equation}
	over the admissible set
	\begin{equation*}
	\mathcal{D}_{h} := \left\{
	\phi\in\mathcal{V}_{\mathrm{per}}|\, \|\phi\|_{\infty}\le 1,
	\, \left\langle{\phi-\gamma,1}\right\rangle_{\Omega}=0
	\right\}\subset\mathbb{R}^{N^3},
	\end{equation*}
	where $S^{n}$ absorbs historical time-level data:
	\begin{equation}\label{fn}
		S^{n}:=-\theta_{0}\breve{\phi}^{n+\frac{1}{2}}
		-\frac{\varepsilon^{2}}{4}\Delta_{h}\phi^{n-1}
		+\tau\left(-\ln (1+\phi^{n})+\ln (1-\phi^{n})\right)
		+\frac{\sigma}{2}(-\Delta_{h})^{-1}(\phi^{n}
		-\overline{\phi^{n}}).
	\end{equation}
	By introducing the zero-mean variable $\varphi=\phi-\gamma$, the functional \eqref{convex functional_2nd} is transformed into an equivalent form $\Xi^{n}(\varphi) := \Gamma^{n}(\varphi+\gamma)$, yields
	\begin{align*}
		\Xi^{n}(\varphi)
			&= \frac{1}{2\tau}\|\varphi+\gamma-\phi^{n}\|_{\mathcal{L}_{\breve{\mathcal{M}}^{n+1/2}}^{-1}}^2
			+\left\langle{H_{1+\phi^{n}}^{0}(1+\varphi+\gamma),1}\right\rangle_{\Omega}
			-\left\langle{H_{1-\phi^{n}}^{0}(1-\varphi-\gamma),1}\right\rangle_{\Omega}\\
			&\quad+\tau\big(
			\left\langle{1+\varphi+\gamma,\ln (1+\varphi+\gamma)}\right\rangle_{\Omega}
			+\left\langle{1-\varphi-\gamma,\ln (1-\varphi-\gamma)}\right\rangle_{\Omega}\big)\\
			&\quad+\frac{3{\varepsilon}^2}{8}\|\nabla_{h}\varphi\|_2^2
			+\frac{\sigma}{4}\|\varphi\|_{-1,h}^2
			+\left\langle{\varphi+\gamma,S^{n}}\right\rangle_{\Omega},   	
	\end{align*}
	defined on the corresponding space
	\[
	\mathring{\mathcal{D}}_{h} := \left\{\varphi\in\mathring{\mathcal{V}}_{\mathrm{per}}\,|\,
	                           -1-\gamma\le\varphi\le 1-\gamma\right\}\subset\mathbb{R}^{N^3}.
	\]
	The minimizers of $\Xi^{n}$ on $\mathring{\mathcal{D}}_{h}$ and $\Gamma^{n}$ on $\mathcal{D}_h$ are naturally in a one-to-one correspondence via $\phi=\varphi+\gamma$.
	
	To rigorously establish the existence of a minimizer, we introduce a regularized compact subset for a small parameter $\vartheta\in (0,1/2)$:
	\[
	\mathring{\mathcal{D}}_{h,\vartheta} 
	:= \left\{\varphi\in\mathring{\mathcal{V}}_{\mathrm{per}}\,|\, \vartheta-1-\gamma
	\le\varphi\le 1-\gamma-\vartheta\right\}\subset\mathbb{R}^{N^3}.
	\] 
	The finite dimensional boundedness, compactness, and convexity of $\mathring{\mathcal{D}}_{h,\vartheta}$ ensure that $\Xi^{n}$ possesses a global minimizer within this regularized domain. The key observation in the positivity-preserving analysis is that, for sufficiently small $\vartheta$, no minimizer lies on the boundary of $\mathring{\mathcal{D}}_{h,\vartheta}$. Here, the boundary is characterized by
	\begin{equation*}
		\partial\mathring{\mathcal{D}}_{h,\vartheta} 
		:= \left\{\psi\in\mathring{\mathcal{D}}_{h,\vartheta}\,|\,
		\|\psi + \gamma\|_{\infty} = 1 - \vartheta\right\}.
	\end{equation*}
	
	Suppose, for contradiction, that the minimization is achieved at a boundary profile $\varphi^{\star}\in\partial\mathring{\mathcal{D}}_{h,\vartheta}$. Then there exists a specific spatial grid location $\vec{a}=(a_1,a_2,a_3)$ such that $\abs{{\varphi}_{\vec{a}}^{\star}+\gamma}=1-\vartheta$. Without loss of generality, we assume ${\varphi}_{\vec{a}}^{\star}+\gamma=\vartheta-1$, implying that ${\varphi}^{\star}$ reaches its global minimum at $\vec{a}$. 
	Let $\vec{b}=(b_1,b_2,b_3)$ be another grid point at which ${\varphi}^{\star}$ attains its maximum. The zero-mean property $\overline{\varphi^{\star}}=0$ requires that ${\varphi}_{\vec{b}}^{\star}\ge 0$, which yields the following bound:
	\begin{equation}\label{alpha1beta0_2nd}
		1-\vartheta\ge{\varphi}_{\vec{b}}^{\star}+\gamma\ge\gamma.
	\end{equation}

   Select a specific direction $\psi\in\mathring{\mathcal{V}}_{\mathrm{per}}$ as
	\[
	\psi_{i,j,k}=\delta_{i,a_1}\delta_{j,a_2}\delta_{k,a_3}-\delta_{i,b_{1}}\delta_{j,b_2}\delta_{k,b_3},
	\]
	$\delta_{i,j}$ denotes the standard Kronecker delta function. Computing the directional derivative of $\Xi^{n}$ at $\varphi^{\star}$ along $\psi$ yields
	\begin{equation}\label{eseq1_2nd}
		\begin{split}
		&\frac{1}{h^3}d_s\Xi^{n}(\varphi^{\star}+s\psi)|_{s=0}\\
		=&\,\frac{1}{\tau}\big(
		\mathcal{L}_{\breve{\mathcal{M}}^{n+1/2}}^{-1}(\phi^{\star}-\phi^{n})_{\vec{a}}
		-\mathcal{L}_{\breve{\mathcal{M}}^{n+1/2}}^{-1}(\phi^{\star}-\phi^{n})_{\vec{b}}\big)
		+(S^{n}_{\vec{a}}-S^{n}_{\vec{b}})
		-\frac{3\varepsilon^{2}}{4}(\Delta_{h}\phi^{\star}_{\vec{a}}
		-\Delta_{h}\phi^{\star}_{\vec{b}})\\
		&+H_{1+\phi^{n}}^{1}(1+\phi^{\star}_{\vec{a}})
		-H_{1-\phi^{n}}^{1}(1-\phi^{\star}_{\vec{a}})
		-H_{1+\phi^{n}}^{1}(1+\phi^{\star}_{\vec{b}})
		+H_{1-\phi^{n}}^{1}(1-\phi^{\star}_{\vec{b}})\\
		&+\tau\Big(\ln(1+\phi^{\star}_{\vec{a}})
		-\ln(1-\phi^{\star}_{\vec{a}})
		-\ln(1+\phi^{\star}_{\vec{b}})
		+\ln(1-\phi^{\star}_{\vec{b}})\Big)\\
		&+\frac{\sigma}{2}(-\Delta_{h})^{-1}(\phi^{\star}-\gamma)_{\vec{a}}
		-\frac{\sigma}{2}(-\Delta_{h})^{-1}(\phi^{\star}-\gamma)_{\vec{b}}.
		\end{split}
	\end{equation}
	In \eqref{eseq1_2nd}, a rewritten form $\phi^{\star}:=\varphi^{\star}+\gamma$ has been used.
	By the fact that $\phi_{\vec{a}}^{\star}=-1+\vartheta$ and \eqref{alpha1beta0_2nd}, the logarithmic differences are bounded by
	\begin{equation}\label{eq1_2nd}
		\ln(1+\phi^{\star}_{\vec{a}})-\ln(1-\phi^{\star}_{\vec{a}})
		-\ln(1+\phi^{\star}_{\vec{b}})+\ln(1-\phi^{\star}_{\vec{b}})
		\le\ln\frac{\vartheta}{2-\vartheta}-\ln\frac{1+\gamma}{1-\gamma}.
	\end{equation}
	By Lemma \ref{G prop}, the derivative $H_{a}^1(x)$ is strictly monotonically increasing with respect to $x$ for any fixed parameter $a>0$. This guarantees the following array of inequalities:
	\begin{equation}
		\begin{aligned}
			&H_{1+\phi^{n}}^{1}(1+\phi^{\star}_{\vec{a}})
			=H_{1+\phi^{n}}^{1}(\vartheta)
			\le H_{1+\phi^{n}}^{1}(1+\phi^{n})
			=\ln (1+\phi^{n})+1,\\
			&H_{1-\phi^{n}}^{1}(1-\phi^{\star}_{\vec{a}})
			=H_{1-\phi^{n}}^{1}(2-\vartheta)\ge H_{1-\phi^{n}}^{1}(1),\\
			&H_{1+\phi^{n}}^{1}(1+\phi^{\star}_{\vec{b}})
			\ge	H_{1+\phi^{n}}^{1}(1+\gamma),\qquad
			H_{1-\phi^{n}}^{1}(1-\phi^{\star}_{\vec{b}})
			\le	H_{1-\phi^{n}}^{1}(1-\gamma).
		\end{aligned}
	\end{equation}
	Furthermore, the extreme value properties of $\phi^{\star}$ at $\vec{a}$ and $\vec{b}$ imply that the discrete Laplacian satisfies
	\begin{equation}\label{1.3}
		\Delta_{h}\phi^{\star}_{\vec{a}}\ge0,\qquad
		\Delta_{h}\phi^{\star}_{\vec{b}}\le 0.
	\end{equation}
	Owing to the uniform bound $\|\phi^{n}\|_{\infty}<1$ and the finiteness of grid, it is reasonable to assume that there exists a constant $\vartheta_{0}>0$, depending on $t_n$, such that
	\begin{equation*}
		-1+\vartheta_{0}\le\phi^{n}_{i,j,k}\le 1-\vartheta_{0},\quad
		\forall\, 1\le i, j, k\le N.
	\end{equation*}
	 Referring back to \eqref{fn}, we then derive
	\begin{equation*}
		\begin{split}
		\abs{S^{n}}
		&\le \theta_{0}\abs{\breve{\phi}^{n+\frac{1}{2}}}
		+\frac{\varepsilon^2}{4}\abs{\Delta_{h}\phi^{n-1}}
		+\tau\bigl(\abs{\ln (1+\phi^{n})}+\abs{\ln (1-\phi^{n})}\bigr)
		+\frac{\sigma}{2}\abs{(-\Delta_{h})^{-1}(\phi^{n}-\bar{\phi}^{n})}\\
		&\le 2\theta_{0}+\frac{3\varepsilon^2}{h^2}
		+2\tau\abs{\ln \vartheta_{0}}+\frac{\sigma}{2}\tilde{C_2}.
		\end{split}           
	\end{equation*}
	Therefore, it follows pointwise that
	\begin{equation}
		-4\theta_{0}-6\frac{\varepsilon^2}{h^2}-4\tau \abs{\ln \vartheta_0}-\tilde{C_2}\sigma
		\le S^{n}_{\vec{a}}-S^{n}_{\vec{b}}
		\le 4\theta_{0}+6\frac{\varepsilon^2}{h^2}+4\tau \abs{\ln \vartheta_0}+\tilde{C_2}\sigma.
	\end{equation}	
	For the last four terms appearing in \eqref{eseq1_2nd}, an application of Lemma \ref{lemma mob 2nd} reveals that
	\begin{align}
		-2\tilde{C_1}&\le	\mathcal{L}_{\breve{\mathcal{M}}^{n+1/2}}^{-1}(\phi^{\star}-\phi^{n})_{\vec{a}}-\mathcal{L}_{\breve{\mathcal{M}}^{n+1/2}}^{-1}(\phi^{\star}-\phi^{n})_{\vec{b}}\le 2\tilde{C_1},\label{eq4_2nd}\\
		-\tilde{C_{2}}\sigma&\le\frac{\sigma}{2}(-\Delta_{h})^{-1}(\phi^{\star}-\gamma)_{\vec{a}}-\frac{\sigma}{2}(-\Delta_{h})^{-1}(\phi^{\star}-\gamma)_{\vec{b}}\le \tilde{C_{2}}\sigma.\label{eq5_2nd}
	\end{align}
	Consequently, a substitution of \eqref{eq1_2nd} through \eqref{eq5_2nd} into \eqref{eseq1_2nd} yields the following bound on the directional derivative:
	\begin{equation*}
		\begin{aligned}
		&\frac{1}{h^3}d_s\Xi^{n}(\varphi^{\star}+s\psi)|_{s=0}
		\le \tau(\ln \frac{\vartheta}{2-\vartheta}-\ln \frac{1+\gamma}{1-\gamma})+J_0,\\
		&J_0:= \ln(1+\phi^{n})+1
		-H_{1-\phi^{n}}^{1}(1)-H_{1+\phi^{n}}^{1}(1+\gamma)
		+H_{1-\phi^{n}}^{1}(1-\gamma)+4\theta_{0}+2\tilde{C_2}\sigma\\
		&\qquad\,+2\tilde{C_1}\tau^{-1}+6\varepsilon^{2}h^{-2}
		+4\tau \,\abs{\ln \vartheta_{0}}.
	   \end{aligned}
	\end{equation*}
	It is worth noting that although $J_0$ exhibits singular behavior as $\tau,\, h\to 0$, it behaves as a finite constant for any fixed $\tau$ and $h$. This property guarantees that one can select an exceptionally small $\vartheta\in(0,1/2)$ such that
	\begin{equation*}
		d_s\Xi^{n}(\varphi^{\star}+s\psi)|_{s=0}<0.
	\end{equation*}
	This contradicts the earlier assumption that $\Xi^{n}$ attains a minimum at $\varphi^{\star}$, as the directional derivative is negative in a direction pointing into the interior of $\mathring{\mathcal{D}}_{h,\vartheta}$.
	
	An analogous analysis eliminates the possibility of the global minimum occurring at a boundary point characterized by $ \varphi_{\vec{a}}^{\star}+\gamma=1-\vartheta$. Consequently, the global minimizer of $\Xi^{n}$ over $\mathring{\mathcal{D}}_{h,\vartheta}$ is strictly contained within the interior region $(\mathring{\mathcal{D}}_{h,\vartheta})^{o}\subset(\mathring{\mathcal{D}}_{h})^{o}$. By mapping this interior optimizer back via $\phi=\varphi+\gamma$, we confirm the existence of a solution $\phi \in \mathcal{D}_h$ for \eqref{2order}. 
	
	Uniqueness is derived directly from the strict convexity of the discrete functional $\Gamma^{n}$ on the admissible set $\mathcal{D}_{h}$. This completes the proof.
	
\end{proof}

We now proceed to the energy stability analysis of the proposed scheme \eqref{2order}. The discrete energy $\mathcal{E}_{h}(\phi):\mathcal{V}_{\mathrm{per}}\to\mathbb{R}$ is defined as
\begin{equation*}
	\mathcal{E}_{h}(\phi)=\left\langle{1+\phi,\ln(1+\phi)}\right\rangle_{\Omega}+\left\langle{1-\phi,\ln(1-\phi)}\right\rangle_{\Omega}-\frac{\theta_{0}}{2}\|\phi\|_2^2+\frac{{\varepsilon}^2}{2}\|\nabla_{h}\phi\|_2^2+\frac{\sigma}{2}\|\phi-\bar{\phi}\|_{-1,h}^2.
\end{equation*}
\begin{theorem}\label{dis ener theo}
	The second order numerical scheme \eqref{2order} satisfies a discrete energy dissipation law:
	\begin{equation}\label{modi energy 2nd}
		\tilde{\mathcal{E}}_{h}(\phi^{n+1},\phi^{n})\le \tilde{\mathcal{E}}_{h}(\phi^{n},\phi^{n-1}),
	\end{equation}
	where the modified energy functional is defined as
	\begin{equation*}
		\tilde{\mathcal{E}}_{h}(\phi,\psi):=\mathcal{E}_{h}(\phi)+\frac{\theta_{0}}{4}\|\phi-\psi\|_2^2
		                      +\frac{\varepsilon^2}{8}\|\nabla_{h}(\phi-\psi)\|_2^2.
	\end{equation*}
	As a direct consequence,
	\begin{equation*}
		 \mathcal{E}_{h}(\phi^{m})
		\le\tilde{\mathcal{E}}_{h}(\phi^{m},\phi^{m-1})
		\le\cdots
		\le\tilde{\mathcal{E}}_{h}(\phi^{1},\phi^{0})
		\le\tilde{\mathcal{E}}_{h}(\phi^{0},\phi^{-1})
		=\mathcal{E}_{h}(\phi^{0})\le C_{3},\quad \forall m\in \mathbb{N}.
	\end{equation*}
	 $C_{3}>0$ is a constant independent of $h$.
\end{theorem}
\begin{proof}
	Taking an inner product with \eqref{2order} by $\mu^{n+\frac{1}{2}}$ yields
	\begin{equation}\label{energy decay theory 2nd}
		  \left\langle{\phi^{n+1}-\phi^{n},\mu^{n+\frac{1}{2}}}\right\rangle_{\Omega}
			=-\tau\left[\breve{\mathcal{M}}^{n+\frac{1}{2}}\nabla_{h}\mu^{n+\frac{1}{2}},
			\nabla_{h}\mu^{n+\frac{1}{2}}\right]_{\Omega}
			\le -\tau^4\,\|\nabla_{h}\mu^{n+\frac{1}{2}}\|_2^2\le 0.
	\end{equation}
	On the other hand, the inner product of the discrete temporal derivative with the modified CN part precisely leads to the corresponding nonlinear energy difference,
	\begin{equation}\label{CN energy es 2nd}
			\left\langle{\phi^{n+1}-\phi^{n},
				\frac{G(1\pm\phi^{n+1})-G(1\pm\phi^{n})}{\phi^{n+1}-\phi^{n}}
			}\right\rangle_{\Omega}
			=\left\langle{1\pm\phi^{n+1},\ln (1\pm\phi^{n+1})}\right\rangle_{\Omega}
			-\left\langle{1\pm\phi^{n},\ln (1\pm\phi^{n})}\right\rangle_{\Omega}.
	\end{equation}
    For the expansive and surface diffusion terms, the following estimates could be derived:
    	\begin{align}
    		&\quad\left\langle{\phi^{n+1}-\phi^{n},
    			-\breve{\phi}^{n+\frac{1}{2}}}\right\rangle_{\Omega}\nonumber\\
    		&=\left\langle{\phi^{n+1}-\phi^{n},-\phi^{n}}\right\rangle_{\Omega}
    		-\frac{1}{2}\left\langle{\phi^{n+1}-\phi^{n},\phi^{n}-\phi^{n-1}}\right\rangle_{\Omega}\\
    		&\ge-\frac{1}{2}\left(\|\phi^{n+1}\|_2^2-\|\phi^{n}\|_2^2\right)
    		+\frac{1}{4}\left(\|\phi^{n+1}-\phi^{n}\|_2^2-\|\phi^{n}-\phi^{n-1}\|_2^2\right),\nonumber\\
    		&\quad\left\langle{\phi^{n+1}-\phi^{n},
    			-\Delta_{h}\hat{\phi}^{n+\frac{1}{2}}}\right\rangle_{\Omega}\nonumber\\
    		&=\left\langle{\nabla_{h}(\phi^{n+1}-\phi^{n}),
    			\frac{1}{2}\nabla_{h}\left(\phi^{n+1}+\phi^{n}\right)
    			+\frac{1}{4}\nabla_{h}\left(\phi^{n+1}-2\phi^{n}+\phi^{n-1}\right)
    		}\right\rangle_{\Omega}\\
    		&\ge\frac{1}{2}\left(\|\nabla_{h}\phi^{n+1}\|_2^2-\|\nabla_{h}\phi^{n}\|_2^2\right)
    		+\frac{1}{8}\left(\|\nabla_{h}(\phi^{n+1}-\phi^{n})\|_2^2-\|\nabla_{h}(\phi^{n}-\phi^{n-1})\|_2^2\right).\nonumber
    	\end{align}
    Analysis of the nonlinear artificial regularization term utilizes the convexity property of the logarithmic function,
    \begin{align}
    		\left\langle{\phi^{n+1}-\phi^{n},
    			\ln (1+\phi^{n+1})-\ln (1+\phi^{n})}\right\rangle_{\Omega}\ge 0,\\
    		\left\langle{\phi^{n+1}-\phi^{n},
    			-\ln (1-\phi^{n+1})+\ln (1-\phi^{n})}\right\rangle_{\Omega}\ge 0.
    \end{align}
    Noting that $\phi^{n+1}-\overline{\phi^{n+1}}\in\mathring{\mathcal{V}}_{\mathrm{per}}$ and that $\overline{\phi^{n+1}}=\overline{\phi^{n}}$ due to the mass conservation, we obtain
    \begin{equation*}
    	\begin{split}
    		&\left\langle{\phi^{n+1}-\phi^{n},(-\Delta_{h})^{-1}(\phi^{n+1}-\overline{\phi^{n+1}})}\right\rangle_{\Omega}\\
    		=&\left\langle{(\phi^{n+1}-\overline{\phi^{n+1}})-(\phi^{n}-\overline{\phi^{n}}),(-\Delta_{h})^{-1}(\phi^{n+1}-\overline{\phi^{n+1}})}\right\rangle_{\Omega}\\
    		=&\frac{1}{2}\left(\|\phi^{n+1}-\overline{\phi^{n+1}}\|_{-1,h}^{2}-\|\phi^{n}-\overline{\phi^{n}}\|_{-1,h}^{2}+\|(\phi^{n+1}-\overline{\phi^{n+1}})-(\phi^{n}-\overline{\phi^{n}})\|_{-1,h}^{2}\right).
    	\end{split}
    \end{equation*}
    An analogous argument for the $\phi^{n}$ gives,
        \begin{equation*}
    	\begin{split}
    		&\left\langle{\phi^{n+1}-\phi^{n},(-\Delta_{h})^{-1}(\phi^{n}-\overline{\phi^{n}})}\right\rangle_{\Omega}\\
    		=&\frac{1}{2}\left(\|\phi^{n+1}-\overline{\phi^{n+1}}\|_{-1,h}^{2}-\|\phi^{n}-\overline{\phi^{n}}\|_{-1,h}^{2}-\|(\phi^{n+1}-\overline{\phi^{n+1}})-(\phi^{n}-\overline{\phi^{n}})\|_{-1,h}^{2}\right).
    	\end{split}
    \end{equation*}
    We therefore obtain the following equality for the nonlocal term:
     \begin{equation}\label{nonlocal energy es 2nd}
    		\left\langle{\phi^{n+1}-\phi^{n},\sigma(-\Delta_{h})^{-1}\left(\phi^{n+\frac{1}{2}}-\overline{\phi^{n+\frac{1}{2}}}\right)}\right\rangle_{\Omega}
    		=\frac{\sigma}{2}\|\phi^{n+1}-\overline{\phi^{n+1}}\|_{-1,h}^{2}
    		-\frac{\sigma}{2}\|\phi^{n}-\overline{\phi^{n}}\|_{-1,h}^{2}.
    \end{equation}
    Finally, a substitution of \eqref{CN energy es 2nd} - \eqref{nonlocal energy es 2nd} into \eqref{energy decay theory 2nd} yields \eqref{modi energy 2nd}, so that the modified energy stability is proved.
\end{proof}

\section{Optimal rate convergence analysis}
Now we proceed into the rigorous convergence analysis of the proposed second order scheme. Let $\Phi$ be the exact solution of the NCH system \eqref{pdeeq}. With sufficiently regular initial data, it is reasonable to assume that $\Phi$ has regularity of class $\mathcal{R}$:
\begin{equation*}
	\Phi\in\mathcal{R} := H^4(0,T;C_{\mathrm{per}}(\Omega))\cap H^3(0,T;C_{\mathrm{per}}^{2}(\Omega))\cap L^{\infty}(0,T;C_{\mathrm{per}}^{8}(\Omega)).
\end{equation*}
Moreover, the following separation property is valid for the exact solution:
\begin{equation}\label{sep for exact}
	1+\Phi\ge\epsilon_{0},\quad 1-\Phi\ge\epsilon_{0}, \quad \text{for some}\ \epsilon_{0}>0 \; \text{at a point-wise level}.
\end{equation}
It is worth noting that directly restricting or interpolating the continuous exact solution onto the spatial grid points fails to preserve mass conservation, whereas projecting it into the Fourier spectral space perfectly preserves this nature. Driven by this essential structural consistency, we introduce $\Phi_{N}(\cdot,t) := \mathcal{P}_{N}\Phi(\cdot,t)$, which represents the standard Fourier projection of the exact solution $\Phi$ into the finite dimensional trigonometric polynomial subspace $\mathfrak{B}^{K}$ with $N=2K+1$. The standard projection estimate holds: if $\Phi\in L^{\infty}(0,T;H^{l}_{\mathrm{per}}(\Omega))$, $l\in\mathbb{N}$, then for any $0\le m\le l$,
\begin{equation}\label{standard estimate}
	\|\Phi_{N}-\Phi\|_{L^{\infty}(0,T;H^{m})}\le Ch^{l-m}\|\Phi\|_{L^{\infty}(0,T;H^{l})}.
\end{equation}
Although the Fourier projection operator does not intrinsically inherit the positivity of $1\pm\Phi_{N}$, an appropriate refinement of the spatial grid size $h$ allows us to enforce the separation property:
\begin{equation}
	1+\Phi_{N}\ge\frac{3}{4}\epsilon_{0},\quad
	1-\Phi_{N}\ge\frac{3}{4}\epsilon_{0}.
\end{equation}

Let $\Phi_{N}^{n}=\Phi_{N}(\cdot,t_n)$ with $t_n=n\tau$, and define $\phi_{N}^{n} := \mathcal{P}_{h}\Phi_{N}(\cdot,t_{n})$ as the corresponding collocated grid point values, such that $(\phi^{n}_{N})_{i,j,k}=\Phi_{N}(x_{i},y_{j},z_{k},t=t_{n})$. Mass conservation is preserved at the continuous level:
\begin{equation}\label{mc}
	\int_{\Omega}\Phi_{N}(\cdot,t_n)\,\mathrm{d}\mathbf{x}=\int_{\Omega}\Phi(\cdot,t_n)\,\mathrm{d}\mathbf{x}=\int_{\Omega}\Phi(\cdot,t_{n-1})\,\mathrm{d}\mathbf{x}=\int_{\Omega}\Phi_{N}(\cdot,t_{n-1})\,\mathrm{d}\mathbf{x},\quad\forall n\in\mathbb{N_{+}}.
\end{equation}
On the other hand, the numerical solution \eqref{2order} is also mass conservative at the discrete level $\overline{\phi^{n+1}}=\overline{\phi^{n}}$.
By virtue of the fact that $\Phi_{N}$ belongs to $\mathfrak{B}^{K}$, it always holds that
\[
\int_{\Omega}\Phi_{N}(\cdot,t_{n})\,\mathrm{d}\mathbf{x}=h^{3}\sum\limits_{i,j,k}\Phi_{N}(x_{i},y_{j},z_{k},t_{n})=h^{3}\sum\limits_{i,j,k}(\phi_{N}^{n})_{i,j,k}.
\] 
Therefore, the mass conservation property $\overline{\phi_{N}^{n}}=\overline{\phi_{N}^{n-1}}$ holds at the discrete level. By specifying the initial data as $\phi^{0}=\phi^{0}_{N}=\mathcal{P}_{h}\Phi_{N}(\cdot,t=0)$, we can define the numerical error grid function
\[
e^{n} := \phi_{N}^{n}-\phi^{n},\quad \forall  n\in \mathbb{N}.
\]
It follows that $\overline{e^{n}}=0$, so the discrete $\|\cdot\|_{-1,h}$ norm is well defined for $e^{n}$.

We are now in a position to state the main convergence theorem.

\begin{theorem}\label{error analysis 2order}
	Assume that the mobility satisfies $0<\mathcal{M}_{0}\le\mathcal{M}(x)\le\mathcal{M}_{1}<\infty$, $ \abs{\mathcal{M}'(x)}\le M$, $\forall x\in[-1,1]$, for some positive constants $\mathcal{M}_{0},\,\mathcal{M}_{1}$ and $M$. Given initial data $\Phi(\cdot,t=0)\in\mathcal{V}_{\mathrm{per}}^{6}(\Omega)$, and suppose the exact solution of the NCH equation \eqref{pdeeq} belongs to the regularity class $\mathcal{R}$. If $\tau$ and $h$ are sufficiently small and satisfy the linear refinement requirement $C_{1}h\le\tau\le C_{2}h$, then for all positive integers $n$ with $t_{n}=n\tau\le T$, we have
	\begin{equation}\label{err es 2nd}
		\|\nabla_{h} e^{n}\|_{2}+\Big(\frac{\mathcal{M}_{0}\varepsilon^2}{24}\tau\sum_{m=1}^{n}\|\nabla_{h}\Delta_{h}e^{m}\|_{2}^{2} \Big)^{1/2}\le C(\tau^{2}+h^{2}).
	\end{equation}   
	In \eqref{err es 2nd}, constant $C>0$ is independent of $n,\,\tau$ and $h$.
\end{theorem}

\subsection{Higher-order consistency analysis}
By standard consistency, the projected solution $\Phi_{N}$ satisfies the discrete scheme \eqref{2order} with second order accuracy in both time and space. However, this classical leading local truncation error is far from sufficient to recover an $\ell^{\infty}$ bound for the numerical solution under a variable mobility. To remedy this, we carry out a higher-order consistency analysis. The core motivation is to retain a sufficiently order of accuracy so that inverse inequalities can be applied, thereby yielding the desired $\ell^{\infty}$ separation property. This property is crucial for controlling the nonlinear error terms and establishing a uniform positive lower bound for the variable mobility.

The statement of higher-order consistency result is given below, with a detailed proof provided in Appendix A.

\begin{prop}\label{prop higher order}
	Let $\Phi$ be the exact solution of the NCH system \eqref{pdeeq}, and let $\Phi_{N}$ denote its Fourier projection. Then there exist supplementary fields $\Phi_{\tau}$ and $\Phi_{h}$, which are continuous functions depending only on $\Phi$, such that the corrected profile
	\begin{equation*}
		\hat{\Phi}:=\Phi_{N}+\tau^2\mathcal{P}_{N}\Phi_{\tau}+h^2\mathcal{P}_{N}\Phi_{h},
	\end{equation*}
	satisfies the numerical scheme with a higher-order consistency:
	\begin{align}
		\frac{\hat{\Phi}^{n+1}-\hat{\Phi}^{n}}{\tau}
		&=\nabla_{h}\cdot\left[\left(\frac{3}{2}\breve{\mathcal{M}}(\hat{\Phi}^{n})-\frac{1}{2}\breve{\mathcal{M}}(\hat{\Phi}^{n-1})\right)\nabla_{h}\hat{\mu}^{n+\frac{1}{2}}\right]+\omega^{n+\frac{1}{2}},
		\label{hat Phi 1}\\
		\hat{\mu}^{n+\frac{1}{2}}
		&=\frac{G(1+\hat{\Phi}^{n+1})-G(1+\hat{\Phi}^{n})}{\hat{\Phi}^{n+1}-\hat{\Phi}^{n}}
		+\frac{G(1-\hat{\Phi}^{n+1})-G(1-\hat{\Phi}^{n})}{\hat{\Phi}^{n+1}-\hat{\Phi}^{n}}
		\label{hat Phi 2}\\
		&\quad-\theta_{0}\left(\frac{3}{2}\hat{\Phi}^{n}-\frac{1}{2}\hat{\Phi}^{n-1}\right)
		-\varepsilon^2\Delta_{h}\left(\frac{3}{4}\hat{\Phi}^{n+1}+\frac{1}{4}\hat{\Phi}^{n-1}\right)
		+\sigma(-\Delta_{h})^{-1}\left(\hat{\Phi}^{n+\frac{1}{2}}
		-\overline{\hat{\Phi}^{n+\frac{1}{2}}}\right)
		\nonumber\\
		&\quad+\tau \,
		\Big(\ln(1+\hat{\Phi}^{n+1})-\ln(1+\hat{\Phi}^{n})
		-\ln(1-\hat{\Phi}^{n+1})+\ln(1-\hat{\Phi}^{n})\Big),
		\nonumber
	\end{align}
	where $\breve{\mathcal{M}}(x)=A_h \mathcal{M}(x)$ and the local truncation error satisfies
	\begin{equation}
		\|\omega^{n+\frac{1}{2}}\|_{-1,h}\le C(\tau^3+h^4), 
	\end{equation}
	 with zero mean, $\overline{\omega^{n+\frac{1}{2}}}=0$, at the discrete level for any $n\in\mathbb{N}$.
\end{prop}
\begin{remark}\label{mobility separation}
	The construction process reveals that the supplementary function $\Phi_{\tau}$ and $\Phi_{h}$ depend solely on the exact solution $\Phi$, and are consequently bounded. For sufficiently small temporal step size $\tau$ and spatial step size $h$, the modified function $\hat{\Phi}$ satisfies the separation property:
	\begin{equation}\label{Phi sep}
		\epsilon_{0}^{\star}-1\le\hat{\Phi}\le 1-\epsilon_{0}^{\star},\qquad
		\epsilon_{0}^{\star}>0,
	\end{equation}
	along with boundedness in the discrete $W_{h}^{1,\infty}$ norm:
	\begin{equation}\label{the regularity assumption}
		\|\hat{\Phi}^{k}\|_{\infty}\le C^{\star},\quad
		\|\nabla_{h}\hat{\Phi}^{k}\|_{\infty}\le C^{\star},\quad
		\|\hat{\Phi}^{k+1}-\hat{\Phi}^{k}\|_{\infty}\le C^{\star}\tau,\quad\forall\, k\ge 0.
	\end{equation} 
	Given that the mobility $\mathcal{M}(x)$ satisfies $0<\mathcal{M}_{0}\le\mathcal{M}(x)\le\mathcal{M}_{1}<\infty$, $\abs{\mathcal{M}'(x)}\le M$ for $x\in[-1,1]$, we have $\breve{\mathcal{M}}(\hat{\Phi}^{n+\frac{1}{2}})\ge\mathcal{M}_{0}>0$. When approximating the mobility at temporal nodes $t^{n}$ and $t^{n-1}$ using the explicit extrapolation with coefficients $\frac{3}{2}$ and $-\frac{1}{2}$ respectively, we obtain:
	\begin{equation}\label{extra mob Phi positive}
		\frac{3}{2}\breve{\mathcal{M}}(\hat{\Phi}^{n})
		-\frac{1}{2}\breve{\mathcal{M}}(\hat{\Phi}^{n-1})
		=\breve{\mathcal{M}}(\hat{\Phi}^{n+\frac{1}{2}})+O(\tau^2)
		\ge \frac{\mathcal{M}_{0}}{2}>0.
	\end{equation}
	This approximation remains strictly positive and serves as a numerical approximation at $t^{n+\frac{1}{2}}$. Consequently, in the consistency analysis, no regularization of the form \eqref{mob regu} is required.
\end{remark}

\subsection{Rough error estimate}
 To achieve a higher-order convergence rate, instead of evaluating the standard error function $e^{n}$ directly, we estimate the error between the numerical solution and the constructed profile in Proposition \ref{prop higher order}, and define the perturbed error function as
 \begin{equation*}
 	\tilde{e}^{n} := \mathcal{P}_{h}\hat{\Phi}^{n}-\phi^{n},\quad\forall n\in\mathbb{N}.
 \end{equation*}
 Noting that $\hat{\Phi}^{n}\in\mathfrak{B}^{K}$ and utilizing the mass conservation property \eqref{mass conser hat Phi 2nd}, one readily observes that the error function preserves the discrete zero-mean property, i.e., $\overline{\tilde{e}^{n}}=0$. 
 
Subtracting the fully discrete scheme \eqref{2order} from \eqref{hat Phi 1} - \eqref{hat Phi 2} yields the following error evolutionary equation
 \begin{equation}\label{new error fun 2nd}
 	\frac{\tilde{e}^{n+1}-\tilde{e}^{n}}{\tau}
 	=\nabla_{h}\cdot\left(\check{\tilde{\mathcal{M}}}^{n+\frac{1}{2}}\nabla_{h}\hat{\mu}^{n+\frac{1}{2}}
 	+\breve{\mathcal{M}}^{n+\frac{1}{2}}\nabla_{h}\tilde{\mu}^{n+\frac{1}{2}}\right)
 	+\omega^{n+\frac{1}{2}},
 \end{equation}
 with the following expansions
\begin{align}
	\check{\tilde{\mathcal{M}}}^{n+\frac{1}{2}}
	&=\left(\frac{3}{2}\breve{\mathcal{M}}(\hat{\Phi}^{n})
	-\frac{1}{2}\breve{\mathcal{M}}(\hat{\Phi}^{n-1})\right)
	-\breve{\mathcal{M}}^{n+\frac{1}{2}},
	\label{check tilde mob}\\
	\breve{\mathcal{M}}^{n+\frac{1}{2}}
	&=\frac{3}{2}\breve{\mathcal{M}}(\phi^{n})-\frac{1}{2}\breve{\mathcal{M}}(\phi^{n-1}),
	\label{check mob}\\
		\hat{\mu}^{n+\frac{1}{2}}
	&=\frac{G(1+\hat{\Phi}^{n+1})-G(1+\hat{\Phi}^{n})}{\hat{\Phi}^{n+1}-\hat{\Phi}^{n}}
	+\frac{G(1-\hat{\Phi}^{n+1})-G(1-\hat{\Phi}^{n})}{\hat{\Phi}^{n+1}-\hat{\Phi}^{n}}
	\nonumber\\
	&\quad-\theta_{0}\left(\frac{3}{2}\hat{\Phi}^{n}-\frac{1}{2}\hat{\Phi}^{n-1}\right)
	-\varepsilon^2\Delta_{h}\left(\frac{3}{4}\hat{\Phi}^{n+1}+\frac{1}{4}\hat{\Phi}^{n-1}\right)
	+\sigma(-\Delta_{h})^{-1}\left(\hat{\Phi}^{n+\frac{1}{2}}
	-\overline{\hat{\Phi}^{n+\frac{1}{2}}}\right)
	\nonumber\\
	&\quad+\tau \,
	\Big(\ln(1+\hat{\Phi}^{n+1})-\ln(1+\hat{\Phi}^{n})
	-\ln(1-\hat{\Phi}^{n+1})+\ln(1-\hat{\Phi}^{n})\Big),
	\nonumber\\
	\tilde{\mu}^{n+\frac{1}{2}}
	&=\hat{\mu}^{n+\frac{1}{2}}-\mu^{n+\frac{1}{2}} \nonumber\\
	&=H_{1+\hat{\Phi}^n}^1(1+\hat{\Phi}^{n+1})-H_{1+\phi^n}^1(1+\phi^{n+1})
	-H_{1-\hat{\Phi}^n}^1(1-\hat{\Phi}^{n+1})+H_{1-\phi^n}^1(1-\phi^{n+1})
	\nonumber\\
	&\quad -\theta_{0}\left(\frac{3}{2}\tilde{e}^{n}-\frac{1}{2}\tilde{e}^{n-1}\right)
	-\varepsilon^2\Delta_{h}\left(\frac{3}{4}\tilde{e}^{n+1}+\frac{1}{4}\tilde{e}^{n-1}\right)
	+\sigma(-\Delta_{h})^{-1}\tilde{e}^{n+\frac{1}{2}}
	\label{tilde mu}\\
	&\quad+\tau\left(\ln(1+\hat{\Phi}^{n+1})-\ln(1+\hat{\Phi}^{n})
	             -\ln(1-\hat{\Phi}^{n+1})+\ln(1-\hat{\Phi}^{n})
	             \right.\nonumber\\
	&\hspace{2cm}\left.
	-\ln(1+\phi^{n+1})+\ln(1+\phi^{n})+\ln(1-\phi^{n+1})-\ln(1-\phi^{n})\right),\nonumber
\end{align}
where
\begin{equation*}
	\|\omega^{n+\frac{1}{2}}\|_2\le C(\tau^3+h^4).
\end{equation*}

 Since $\hat{\mu}^{n+\frac{1}{2}}$ only depends on the constructed profile $\hat{\Phi}$, it is reasonable to assume that
 \begin{equation}\label{mu hat bound 2nd}
 	\|\hat{\mu}^{n+\frac{1}{2}}\|_{W_h^{1,\infty}}\le C^{\star}.
 \end{equation}
 Furthermore, we impose the following assumption at the previous two time steps:
 \begin{equation}\label{priori assum 2nd}
 	\|\nabla_{h}\tilde{e}^{k}\|_{2}\le\tau^{\frac{11}{4}}+h^{\frac{15}{4}},\quad k=n,\, n-1.
 \end{equation}
 Applying the Poincar{\'e} inequality ($ \|\tilde{e}^{k}\|_{2}\le C_{p}\|\nabla_{h}\tilde{e}^{k}\|_{2}$), we immediately obtain
 \begin{equation}
 	\|\tilde{e}^{k}\|_{H_{h}^1}\le C(\tau^{\frac{11}{4}}+h^{\frac{15}{4}}),\quad k=n,\, n-1.
 \end{equation}
 This a priori assumption will be recovered by the convergence analysis at the next time step. As a direct consequence, the inverse inequality yields an $\ell^{\infty}$ bound for the numerical error function
 \begin{equation}\label{tilde phik infty es}
 	\|\tilde{e}^{k}\|_{\infty}\le \frac{C\|\tilde{e}^{k}\|_{H_h^1}}{h^{\frac{1}{2}}}
 	\le C(\tau^{\frac{9}{4}}+h^{\frac{13}{4}})
 	\le \frac{\epsilon_{0}^{\star}}{2}\le 1,
 \end{equation}
 with $k=n,\, n-1$, and the linear refinement constraint $C_{1}h\le\tau\le C_{2}h$ has been used. In combination with the regularity assumption~\eqref{the regularity assumption}, this immediately implies
  \begin{equation}\label{phi infty es 2nd}
 	\|\phi^{k}\|_{\infty}
 	\le \|\hat{\Phi}^{k}\|_{\infty}+\|\tilde{e}^{k}\|_{\infty}
 	\le C^{\star}+1=:\tilde{C_{3}},\quad k=n,\, n-1.
 \end{equation}
 Moreover, incorporating the separation property of $\hat{\Phi}$ in~\eqref{Phi sep} leads to the analogous separation bound for the numerical solution at $k=n,\, n-1$:
 \begin{equation}\label{sep for phi n}
 	\frac{\epsilon_{0}^{\star}}{2}\le 1+\phi^{k}=1+\hat{\Phi}^{k}-\tilde{e}^k\le 2,\quad
 	\frac{\epsilon_{0}^{\star}}{2}\le 1-\phi^{k}=1-\hat{\Phi}^{k}+\tilde{e}^k\le 2.
 \end{equation}

 \begin{remark}\label{check mob positive}
 	Based on the separation property of $\hat{\Phi}$ in~\eqref{Phi sep} and the boundedness of the first derivative of mobility over $[-1,1]$, together with the estimate~\eqref{tilde phik infty es}, we have for $k = n,\, n-1$
 	\begin{equation*}
 		\breve{\mathcal{M}}(\phi^{k})
 		=\breve{\mathcal{M}}(\hat{\Phi}^{k})+\breve{\mathcal{M}}'(\hat{\Phi}^{k})(-\tilde{e}^{k})+O((-\tilde{e}^{k})^2)
 		=\breve{\mathcal{M}}(\hat{\Phi}^{k})
 		+O(\tau^{\frac{9}{4}}+h^{\frac{13}{4}}).
 	\end{equation*}
 	Therefore, when $\tau$ and $h$ are sufficiently small, estimate~\eqref{extra mob Phi positive} implies
 	\begin{equation*}
 		\begin{split}
 			\frac{3}{2}\breve{\mathcal{M}}(\phi^{n})-\frac{1}{2}\breve{\mathcal{M}}(\phi^{n-1})
 			&=\frac{3}{2}\breve{\mathcal{M}}(\hat{\Phi}^{n})-\frac{1}{2}\breve{\mathcal{M}}(\hat{\Phi}^{n-1})+O(\tau^{\frac{9}{4}}+h^{\frac{13}{4}})\\
 			&\ge \breve{\mathcal{M}}(\hat{\Phi}^{n+\frac{1}{2}})+O(\tau^2)+O(\tau^{\frac{9}{4}}+h^{\frac{13}{4}})
 			\ge \frac{\mathcal{M}_{0}}{2}>0.
 		\end{split}
 	\end{equation*}
 	This analysis shows that, if the a priori assumption~\eqref{priori assum 2nd} holds, the regularization of the mobility in the scheme~\eqref{mob regu} can be omitted as well. In this case, $\breve{\mathcal{M}}^{n+\frac{1}{2}}$ may be reformulated as in~\eqref{check mob}, with the corresponding modified form given in~\eqref{check tilde mob}.
 \end{remark}

Before proceeding into the rough error estimate, a preliminary bound for the nonlinear term is necessary. For the sake of brevity, we adopt the following notation:
\begin{align*}
	&\mathcal{N}_1^{n}:=H_{1+\hat{\Phi}^n}^1(1+\hat{\Phi}^{n+1})-H_{1+\phi^n}^1(1+\phi^{n+1}),\,
	&&\mathcal{N}_2^{n}:=-H_{1-\hat{\Phi}^n}^1(1-\hat{\Phi}^{n+1})+H_{1-\phi^n}^1(1-\phi^{n+1}),\\
	&\mathcal{N}_3^{n}:=\ln(1+\hat{\Phi}^{n+1})-\ln(1+\phi^{n+1}),\,
	&&\mathcal{N}_4^{n}:=-\ln(1+\hat{\Phi}^{n})+\ln(1+\phi^{n}),\\
	&\mathcal{N}_5^{n}:=-\ln(1-\hat{\Phi}^{n+1})+\ln(1-\phi^{n+1}),\,
	&&\mathcal{N}_6^{n}:=\ln(1-\hat{\Phi}^{n})-\ln(1-\phi^{n}).
\end{align*}
For $\mathcal{N}_1^{n}$, the decomposition is valid
\begin{equation}\label{N1n split}
	\begin{aligned}
		\mathcal{N}_1^{n}&= \mathcal{N}_{11}^{n}+\mathcal{N}_{12}^{n},\\
		\mathcal{N}_{11}^{n}&:=
		H_{1+\phi^n}^1(1+\hat{\Phi}^{n+1})-H_{1+\phi^n}^1(1+\phi^{n+1}),\\
		\mathcal{N}_{12}^{n}&:=
		H_{1+\hat{\Phi}^{n+1}}^1(1+\hat{\Phi}^{n})-H_{1+\hat{\Phi}^{n+1}}^1(1+\phi^{n}).
	\end{aligned}
\end{equation}
And similar decomposition can be performed to $\mathcal{N}_2^{n}$. The preliminary estimate is stated in Lemma~\ref{rough bound for nolinear}; detailed proof will be provided in Appendix B.
\begin{lemma}\label{rough bound for nolinear}
	Suppose that the modified function $\hat{\Phi}$ satisfies the separation property~\eqref{Phi sep} and the regularity assumption~\eqref{the regularity assumption}, and that the numerical solutions at the two previous time levels, $\phi^{n}$ and $\phi^{n-1}$, also possess the separation property~\eqref{sep for phi n}. Then the nonlinear inner-product terms admit estimates
		\begin{align}
		&\langle\tilde{e}^{n+1},\mathcal{N}_1^{n}\rangle
		+\tau\,
		\langle\tilde{e}^{n+1},\mathcal{N}_3^{n}+\mathcal{N}_4^{n}\rangle
		\ge \tilde{C_{5}}\|\tilde{e}^{n+1}\|_{2}^{2}
		-\tilde{C_{4}}\|\tilde{e}^{n}\|_{2}^{2},\label{5.3.1}\\
		&\langle\tilde{e}^{n+1},\mathcal{N}_2^{n}\rangle
		+\tau\,
		\langle\tilde{e}^{n+1},	\mathcal{N}_5^{n}+\mathcal{N}_6^{n}\rangle
		\ge \tilde{C_{5}}\|\tilde{e}^{n+1}\|_{2}^{2}
		-\tilde{C_{4}}\|\tilde{e}^{n}\|_{2}^{2},   \label{4.1.2}  
	\end{align}
	where $\tilde{C}_{4}$ and $\tilde{C}_{5}$ are positive constants depending only on $\epsilon_{0}^{\star}$ and $C^{\star}$, but independent of the time step $\tau$ and the spatial mesh size~$h$.
\end{lemma}
 
Lemma \ref{rough error lemma} states the rough error estimate result.
 \begin{lemma}\label{rough error lemma}
 	Suppose that the regularity requirement assumption \eqref{mu hat bound 2nd} and the a priori bound \eqref{priori assum 2nd} hold. If the time step $\tau$ and the spatial mesh size $h$ are sufficiently small and satisfy the linear refinement constraint $C_{1}h\le\tau\le C_{2}h$, then the following rough error estimate is valid for the numerical error evolution system \eqref{new error fun 2nd}:
 	\begin{equation}\label{rough error es 2nd}	\|\tilde{e}^{n+1}\|_{2}+\|\nabla_{h}\tilde{e}^{n+1}\|_{2}\le C(\tau^{\frac{9}{4}}+h^{\frac{13}{4}}),
 	\end{equation}
 	where $C>0$ is a constant independent of $\tau$ and $h$.
 \end{lemma}
 \begin{proof}
 	A discrete inner product with \eqref{new error fun 2nd} by $\tilde{\mu}^{n+\frac{1}{2}}$ results in
 	\begin{equation}\label{inner mu 2nd}
 		\begin{split}
 			&\frac{1}{\tau}\langle\tilde{e}^{n+1},\tilde{\mu}^{n+\frac{1}{2}}\rangle
 			+\left[\breve{\mathcal{M}}^{n+\frac{1}{2}}\nabla_{h}\tilde{\mu}^{n+\frac{1}{2}},\nabla_{h}\tilde{\mu}^{n+\frac{1}{2}}\right]\\
 			&\qquad
 			=\frac{1}{\tau}\langle\tilde{e}^{n},\tilde{\mu}^{n+\frac{1}{2}}\rangle
 			-\left[\check{\tilde{\mathcal{M}}}^{n+\frac{1}{2}}\nabla_{h}\hat{\mu}^{n+\frac{1}{2}},
 			\nabla_{h}\tilde{\mu}^{n+\frac{1}{2}}\right]
 			+\langle{\omega^{n+\frac{1}{2}},\tilde{\mu}^{n+\frac{1}{2}}}\rangle.
 		\end{split}
 	\end{equation}
 	Since the mobility $\breve{\mathcal{M}}^{n+\frac{1}{2}}$ is bounded below, as discussed in Remark \ref{check mob positive}, it is clear that 
 	\begin{equation*}
 		\left[\breve{\mathcal{M}}^{n+\frac{1}{2}}\nabla_{h}\tilde{\mu}^{n+\frac{1}{2}},\nabla_{h}\tilde{\mu}^{n+\frac{1}{2}}\right]
 		\ge\frac{\mathcal{M}_{0}}{2}\|\nabla_{h}\tilde{\mu}^{n+\frac{1}{2}}\|_{2}^{2}.
 	\end{equation*}
 	Leveraging the mean-zero property of the truncation error $\tau$, we obtain the following estimate
 	\begin{equation*}
 		\langle\omega^{n+\frac{1}{2}},\tilde{\mu}^{n+\frac{1}{2}}\rangle
 		\le\|\omega^{n+\frac{1}{2}}\|_{-1,h}\cdot\|\nabla_{h}\tilde{\mu}^{n+\frac{1}{2}}\|_{2}
 		\le\frac{2}{\mathcal{M}_{0}}\|\omega^{n+\frac{1}{2}}\|_{-1,h}^{2}
 		+\frac{\mathcal{M}_{0}}{8}\|\nabla_{h}\tilde{\mu}^{n+\frac{1}{2}}\|_{2}^{2}.
 	\end{equation*}
 	An analogous method provides
 		\begin{equation*}
 		\langle\tilde{e}^{n},\tilde{\mu}^{n+\frac{1}{2}}\rangle
 		\le\|\tilde{e}^{n}\|_{-1,h}\cdot\|\nabla_{h}\tilde{\mu}^{n+\frac{1}{2}}\|_{2}
 		\le\frac{2}{\mathcal{M}_{0}\tau}\|\tilde{e}^{n}\|_{-1,h}^{2}
 		+\frac{\mathcal{M}_{0}\tau}{8}\|\nabla_{h}\tilde{\mu}^{n+\frac{1}{2}}\|_{2}^{2}.
 	\end{equation*}
 	For the second term in the second line of equation~\eqref{inner mu 2nd}, the regularity assumption~\eqref{mu hat bound 2nd} gives
 	\begin{equation*}
 		\begin{split}
 			-\left[\check{\tilde{\mathcal{M}}}^{n+\frac{1}{2}}\nabla_{h}\hat{\mu}^{n+\frac{1}{2}},
 			\nabla_{h}\tilde{\mu}^{n+\frac{1}{2}}\right]
 			&\le C^{\star}\|\check{\tilde{\mathcal{M}}}^{n+\frac{1}{2}}\|_{2}
 			\cdot\|\nabla_{h}\tilde{\mu}^{n+\frac{1}{2}}\|_{2}\\
 			&\le \frac{2(C^{\star})^2}{\mathcal{M}_{0}}\|\check{\tilde{\mathcal{M}}}^{n+\frac{1}{2}}\|_{2}^2+\frac{\mathcal{M}_{0}}{8}\|\nabla_{h}\tilde{\mu}^{n+\frac{1}{2}}\|_{2}^{2}.
 		\end{split}
 	\end{equation*}
 	Substituting the above inequality into~\eqref{inner mu 2nd} yields
 	\begin{equation}\label{half es in rough}
 		\begin{split}
 			&\quad\langle\tilde{e}^{n+1},\tilde{\mu}^{n+\frac{1}{2}}\rangle
 			+\frac{\mathcal{M}_{0}\tau}{8}\|\nabla_{h}\tilde{\mu}^{n+\frac{1}{2}}\|_{2}^{2}\\
 			&\le \frac{2}{\mathcal{M}_{0}\tau}\|\tilde{e}^{n}\|_{-1,h}^2
 			+\frac{2(C^{\star})^2\tau}{\mathcal{M}_{0}}\|\check{\tilde{\mathcal{M}}}^{n+\frac{1}{2}}\|_{2}^2
 			+\frac{2\tau}{\mathcal{M}_{0}}\|\omega^{n+\frac{1}{2}}\|_{-1,h}^{2}.
 		\end{split}
 	\end{equation}

 On the other hand, for the expansive and surface diffusion terms arising in the expansion of $\langle\tilde{e}^{n+1},\tilde{\mu}^{n+\frac{1}{2}}\rangle$, the following estimates hold:
   \begin{equation}\label{2rough1}
   	\begin{split}
   		\theta_{0}\langle\tilde{e}^{n +1},\frac{3}{2}\tilde{e}^{n}-\frac{1}{2}\tilde{e}^{n-1}\rangle
   		&\le\frac{\theta_{0}^{2}\varepsilon^{-2}}{2}\|3\tilde{e}^{n}-\tilde{e}^{n -1}\|_{-1,h}^{2}+\frac{\varepsilon^2}{8}\|\nabla_{h}\tilde{e}^{n +1}\|_{2}^{2}\\
   		&\le\frac{\theta_{0}^{2}\varepsilon^{-2}}{2}
   		\left(12\|\tilde{e}^{n}\|_{-1,h}^2+4\|\tilde{e}^{n-1}\|_{-1,h}^2\right)
   		+\frac{\varepsilon^2}{8}\|\nabla_{h}\tilde{e}^{n+1}\|_{2}^{2},
   	\end{split}
   \end{equation}
   \begin{equation}
   	\begin{split}
   		-\langle\tilde{e}^{n+1},\Delta_{h}(\frac{3}{4}\tilde{e}^{n+1}+\frac{1}{4}\tilde{e}^{n-1})\rangle
   		&=\frac{3}{4}\|\nabla_{h}\tilde{e}^{n+1}\|_{2}^{2}
   		+\frac{1}{4}\langle\nabla_{h}\tilde{e}^{n+1},\nabla_{h}\tilde{e}^{n-1}\rangle\\
   		&\ge\frac{5}{8}\|\nabla_{h}\tilde{e}^{n+1}\|_{2}^{2}
   		-\frac{1}{8}\|\nabla_{h}\tilde{e}^{n-1}\|_{2}^{2}.
   	\end{split}
   \end{equation}
   A natural estimate of the nonlocal term follows directly from the fundamental inequality,
 \begin{equation}\label{2rough3}
 	\sigma\langle\tilde{e}^{n+1},(-\Delta_{h})^{-1}\tilde{e}^{n+\frac{1}{2}}\rangle
 	\ge\frac{\sigma}{4}\|\tilde{e}^{n+1}\|_{-1,h}^{2}
 	-\frac{\sigma}{4}\|\tilde{e}^{n}\|_{-1,h}^{2}.
 \end{equation}
 Substituting equations~\eqref{2rough1} - \eqref{2rough3} into~\eqref{half es in rough}, and combining them with the preliminary bound estimate for the nonlinear term given in Lemma~\ref{rough bound for nolinear}, we obtain
 \begin{equation}\label{above all 2nd}
 	\begin{split}
 		& 2\tilde{C_{5}}\|\tilde{e}^{n+1}\|_{2}^{2}
 		+\frac{\varepsilon^{2}}{2}\|\nabla_{h}\tilde{e}^{n+1}\|_{2}^{2}\\
 		\le & \left(\frac{2}{\mathcal{M}_{0}\tau}+\frac{\sigma}{4}
 		+6\theta_{0}^{2}\varepsilon^{-2}\right)\|\tilde{e}^{n}\|_{-1,h}^2
 		+2\theta_{0}^{2}\varepsilon^{-2}\|\tilde{e}^{n-1}\|_{-1,h}^2
 		+2\tilde{C_{4}}\|\tilde{e}^{n}\|_{2}^{2}\\
 		& +\frac{\varepsilon^2}{8}\|\nabla_{h}\tilde{e}^{n-1}\|_{2}^{2}
 		+\frac{2(C^{\star})^2\tau}{\mathcal{M}_{0}}\|\check{\tilde{\mathcal{M}}}^{n+\frac{1}{2}}\|_{2}^2
 		+\frac{2\tau}{\mathcal{M}_{0}}\|\omega^{n+\frac{1}{2}}\|_{-1,h}^{2}.
 	\end{split}
 \end{equation}
 
 When the temporal and spatial step sizes are sufficiently small and satisfy the linear refinement $C_{1}h\le\tau\le C_{2}h$, a combined application of the a priori assumption~\eqref{priori assum 2nd}, the discrete Poincar{\'e} inequality, and the norm relation $\|\cdot\|_{-1,h}\le C\|\cdot\|_{2}$ to the terms on the right-hand side of inequality~\eqref{above all 2nd} yields
 \begin{equation}\label{2rough es1}
 	\begin{aligned}
 		&\left(\frac{2}{\mathcal{M}_{0}\tau}+\frac{\sigma}{4}
 		+6\theta_{0}^{2}\varepsilon^{-2}\right)\|\tilde{e}^{n}\|_{-1,h}^2
 		\le \frac{CC_{p}}{\mathcal{M}_{0}\tau}\|\nabla_{h}\tilde{e}^{n}\|_{2}^{2}
 		\le C(\tau^{\frac{9}{2}}+h^{\frac{13}{2}}),\\
 		&2\theta_{0}^{2}\varepsilon^{-2}\|\tilde{e}^{n-1}\|_{-1,h}^2
 		\le 2CC_{p}\theta_{0}^{2}\varepsilon^{-2}\|\nabla_{h}\tilde{e}^{n-1}\|_{2}^2
 		\le C(\tau^{\frac{11}{2}}+h^{\frac{15}{2}}),\\
 		&2\tilde{C_{4}}\|\tilde{e}^{n}\|_{2}^{2}
 		\le 2C_{p}\tilde{C_{4}}\|\nabla_{h}\tilde{e}^{n}\|_{2}^{2}
 		\le C(\tau^{\frac{11}{2}}+h^{\frac{15}{2}}),\\
 		&\frac{2\tau}{\mathcal{M}_{0}}\|\omega^{n+\frac{1}{2}}\|_{-1,h}^{2}
 		\le C(\tau^{7}+\tau h^{8}),\qquad
 		\frac{\varepsilon^{2}}{8}\|\nabla_{h}\tilde{e}^{n-1}\|_{2}^{2}
 		\le C(\tau^{\frac{11}{2}}+h^{\frac{15}{2}}).
 	\end{aligned}
 \end{equation}
 Here, $C_{p}$ denotes the constant in the Poincar{\'e} inequality. Applying the mean value theorem to the mobility error function, we have
 \begin{equation*}
 	\begin{split}
 		\|\check{\tilde{\mathcal{M}}}^{n+\frac{1}{2}}\|_{2}^2
 		&\le 3\|\breve{\mathcal{M}}(\hat{\Phi}^{n})-\breve{\mathcal{M}}(\phi^{n})\|_{2}^{2}
 		+\|\breve{\mathcal{M}}(\hat{\Phi}^{n-1})-\breve{\mathcal{M}}(\phi^{n-1})\|_{2}^{2}	\\
 		&\le 3M^2\|\tilde{e}^{n}\|_{2}^{2}+M^2\|\tilde{e}^{n-1}\|_{2}^{2} 
 		\le C(\tau^{\frac{11}{2}}+h^{\frac{15}{2}}).
 	\end{split}
 \end{equation*}
 As a result,
 \begin{equation}\label{2rough es2}
 	\frac{2(C^{\star})^2\tau}{\mathcal{M}_{0}}\|\check{\tilde{\mathcal{M}}}^{n+\frac{1}{2}}\|_{2}^2
 	\le C(\tau^{\frac{13}{2}}+h^{\frac{17}{2}}).
 \end{equation}	
 Substituting \eqref{2rough es1} and \eqref{2rough es2} into \eqref{above all 2nd} yields
 \begin{equation*}
 	\|\tilde{e}^{n+1}\|_{2}^{2}+\|\nabla_{h}\tilde{e}^{n+1}\|_{2}^{2}
 	\le C(\tau^{\frac{9}{2}}+h^{\frac{13}{2}}),
 \end{equation*}
 which implies
 \begin{equation}\label{tilde phi es 2nd}
 	\|\tilde{e}^{n+1}\|_{2}+\|\nabla_{h}\tilde{e}^{n+1}\|_{2}
 	\le C(\tau^{\frac{9}{4}}+h^{\frac{13}{4}}).
 \end{equation}
 The positive constant $C$ in \eqref{tilde phi es 2nd} is independent of the time step size $\tau$ and the spatial mesh size $h$. Therefore, proof of the rough error estimate is complete.
 \end{proof}
 
 The rough error estimate \eqref{rough error es 2nd} reveals that, at the $(n+1)$-th time level, the numerical error function satisfies the following discrete norm estimates:
 \begin{align}
 	&\|\tilde{e}^{n+1}\|_{\infty}
 	\le\frac{C\|\tilde{e}^{n+1}\|_{H_{h}^{1}}}{h^{\frac{1}{2}}}
 	\le \hat{C}(\tau^{\frac{7}{4}}+h^{\frac{11}{4}})
 	\le\frac{\epsilon_{0}^{\star}}{2},\label{err infty es 2nd}\\
 	&\|\tilde{e}^{n+1}\|_{4}
 	\le C\|\tilde{e}^{n+1}\|_{H_{h}^{1}}
 	\le C(\tau^{\frac{9}{4}}+h^{\frac{13}{4}}),\label{es for norm4 2nd}
 \end{align}
 where the 3D inverse inequality and the discrete Sobolev inequality given in Lemma~\ref{norm4} have been used.
 Its combination with the error estimate \eqref{err infty es 2nd} and the separation property of the corrected function~\eqref{Phi sep}, leads to a similar property for the numerical solution at time step $t^{n+1}$,
 \begin{equation}\label{sep for phi n1}
 	\frac{\epsilon_{0}^{\star}}{2}\le 1+\phi^{n+1}< 2,\qquad \frac{\epsilon_{0}^{\star}}{2}\le 1-\phi^{n+1}< 2.
 \end{equation}
 Moreover, by carefully utilizing \eqref{the regularity assumption} and \eqref{tilde phik infty es}, we can further establish a refined estimate involving the discrete time derivative:
 \begin{equation}\label{es for discrete temporal derivative}
 	\begin{aligned}
 		&\|\tilde{e}^{n+1}\|_{\infty}
 		\le \hat{C}(\tau^{\frac{7}{4}}+h^{\frac{11}{4}})\le\tau,\qquad
 		\|\tilde{e}^{n}\|_{\infty}
 		\le \hat{C}(\tau^{\frac{9}{4}}+h^{\frac{13}{4}})\le\tau,\\
 		&\|\tilde{e}^{n+1}-\tilde{e}^{n}\|_{\infty}
 		\le \|\tilde{e}^{n+1}\|_{\infty}+\|\tilde{e}^{n}\|_{\infty}
 		\le C(\tau^{\frac{7}{4}}+h^{\frac{11}{4}})\le\tau,\\
 		&\|\phi^{n+1}-\phi^{n}\|_{\infty}
 		\le \|\hat{\Phi}^{n+1}-\hat{\Phi}^{n}\|_{\infty}
 		+\|\tilde{e}^{n+1}-\tilde{e}^{n}\|_{\infty}
 		\le (C^{\star}+1)\tau.
 	\end{aligned}
 \end{equation}

 \subsection{A refined error estimate}
 Before embarking on the refined error analysis, preliminary estimates for the nonlinear part are likewise required. For simplicity of presentation, the complete proof is presented in Appendix C. 
 \begin{lemma}\label{nonlinear es in refine 2nd}
 	Assume that the time and spatial step sizes are sufficiently small and satisfy the linear constraint condition $C_{1}h\le\tau\le C_{2}h$. 
 	Under the following assumptions: the regularity condition~\eqref{the regularity assumption}; the $\ell^{\infty}$ a priori estimates~\eqref{phi infty es 2nd} and the separation property~\eqref{sep for phi n} for the numerical solution at the first two time steps; the separation property~\eqref{sep for phi n1} for the numerical solution at the next time step; and the discrete temporal derivative estimate~\eqref{es for discrete temporal derivative}, the nonlinear terms satisfy
 	\begin{equation}
 		\begin{aligned}
 			&\|\nabla_{h}\mathcal{N}_{q}^{n}\|_{2}^{2}
 			\le C^{(1)}\left(\|\nabla_{h}\tilde{e}^{n+1}\|_{2}^{2}+\|\nabla_{h}\tilde{e}^{n}\|_{2}^{2}\right),\quad q=1,\, 2,\\
 			&\|\nabla_{h}\mathcal{N}_{m}^{n}\|_{2}^{2}
 			\le C^{(2)}\|\nabla_{h}\tilde{e}^{n+1}\|_{2}^{2},\quad m=3,\, 5,\\
 			&\|\nabla_{h}\mathcal{N}_{l}^{n}\|_{2}^{2}
 			\le C^{(2)}\|\nabla_{h}\tilde{e}^{n}\|_{2}^{2},\quad l=4,\, 6.
 		\end{aligned}
 	\end{equation}
 	Positive constants $C^{(1)}$ and $C^{(2)}$ depend only on the separation parameter $\epsilon_{0}^{\star}$, the regularity parameter $C^{\star}$, the linear constraint constant $C_2$, as well as the Sobolev constants~$C$.
 \end{lemma}
 \begin{lemma}\label{lemma nonlocal}
 	For any $ k\ge 0$, define $\tilde{\psi}^{k}=(-\Delta_{h})^{-1}\tilde{e}^{k}$.  Then there exists a constant $C>0$, independent of $h$, such that
 	\begin{equation*}
 		\|\nabla_{h}\tilde{\psi}^{k}\|_{2}\le C \,\|\tilde{e}^{k}\|_{2}.
 	\end{equation*}
 	\begin{proof}
 		It is a direct consequence of the standard estimate $\|u\|_{-1,h}\le C \|u\|_{2}$, for any $u$ satisfies $\bar{u}=0$.
 	\end{proof}
 \end{lemma}
 
 \vspace{4mm}
 Now we perform the refined error estimate. Taking a discrete inner product with \eqref{new error fun 2nd} by $-\Delta_{h}(\frac{3}{4}\tilde{e}^{n+1}+\frac{1}{4}\tilde{e}^{n-1})$ leads to
\begin{equation}\label{refine inner}
	\begin{split}
		&-\frac{1}{\tau}\langle\tilde{e}^{n+1}-\tilde{e}^{n},\Delta_{h}(\frac{3}{4}\tilde{e}^{n+1}+\frac{1}{4}\tilde{e}^{n-1})\rangle
		-\left[\nabla_{h}\tilde{\mu}^{n+\frac{1}{2}},\breve{\mathcal{M}}^{n+\frac{1}{2}}\nabla_{h}\Delta_{h}(\frac{3}{4}\tilde{e}^{n+1}+\frac{1}{4}\tilde{e}^{n-1})\right]\\
		&\qquad=\left[\nabla_{h}\hat{\mu}^{n+\frac{1}{2}},\check{\tilde{\mathcal{M}}}^{n+\frac{1}{2}}\nabla_{h}\Delta_{h}(\frac{3}{4}\tilde{e}^{n+1}+\frac{1}{4}\tilde{e}^{n-1})\right]
		-\langle\omega^{n+\frac{1}{2}},\Delta_{h}(\frac{3}{4}\tilde{e}^{n+1}+\frac{1}{4}\tilde{e}^{n-1})\rangle,
	\end{split}
\end{equation}
where summation-by-parts formulas have been recalled. Regarding the temporal discretization part in~\eqref{refine inner}, 
\begin{equation}\label{2refine time discrete}
	\begin{split}
		&\quad-\langle\tilde{e}^{n+1}-\tilde{e}^{n},\Delta_{h}(\frac{3}{4}\tilde{e}^{n+1}+\frac{1}{4}\tilde{e}^{n-1})\rangle\\
		&=\frac{1}{2}\left[\nabla_{h}(\tilde{e}^{n+1}-\tilde{e}^{n}),\nabla_{h}(\tilde{e}^{n+1}+\tilde{e}^{n})\right]
		+\frac{1}{4}\left[\nabla_{h}(\tilde{e}^{n+1}-\tilde{e}^{n}),\nabla_{h}(\tilde{e}^{n+1}-2\tilde{e}^{n}+\tilde{e}^{n-1})\right]\\
		&\ge \frac{1}{2}\left(\|\nabla_{h}\tilde{e}^{n+1}\|_{2}^{2}
		-\|\nabla_{h}\tilde{e}^{n}\|_{2}^{2}\right)
		+\frac{1}{8}\left(\|\nabla_{h}(\tilde{e}^{n+1}-\tilde{e}^{n})\|_{2}^{2}
		-\|\nabla_{h}(\tilde{e}^{n-1}-\tilde{e}^{n})\|_{2}^{2}
		\right).
	\end{split}
\end{equation}
The bounded mobility derivative, mean value theorem and regularity condition~\eqref{mu hat bound 2nd} collectively give
\begin{equation}
	\begin{split}
		&\quad\left[\nabla_{h}\hat{\mu}^{n+\frac{1}{2}},\check{\tilde{\mathcal{M}}}^{n+\frac{1}{2}}\nabla_{h}\Delta_{h}(\frac{3}{4}\tilde{e}^{n+1}+\frac{1}{4}\tilde{e}^{n-1})\right]\\
		&\le C^{\star}\left(\frac{3}{2}M\|\tilde{e}^{n}\|_{2}
		+\frac{1}{2}M\|\tilde{e}^{n-1}\|_{2}\right)\cdot
		\|\nabla_{h}\Delta_{h}(\frac{3}{4}\tilde{e}^{n+1}+\frac{1}{4}\tilde{e}^{n-1})\|_{2}\\
		&\le \frac{3}{\mathcal{M}_{0}\varepsilon^{2}}(C^{\star})^2M^2
		\left(3\|\tilde{e}^{n}\|_{2}^{2}+\|\tilde{e}^{n-1}\|_{2}^{2}\right)
		+\frac{\mathcal{M}_{0}\varepsilon^{2}}{12} \|\nabla_{h}\Delta_{h}(\frac{3}{4}\tilde{e}^{n+1}+\frac{1}{4}\tilde{e}^{n-1})\|_{2}^{2}.
	\end{split}
\end{equation}
Estimate of the error term $\omega^{n+\frac{1}{2}}$ can be obtained by utilizing the norm relation $\|\cdot\|_{-1,h}\le C\|\cdot\|_{2}$ and Young's inequality:
\begin{equation}
	\begin{split}
		-\langle\omega^{n+\frac{1}{2}},\Delta_{h}(\frac{3}{4}\tilde{e}^{n+1}\!+\!\frac{1}{4}\tilde{e}^{n-1})\rangle
		&\le\|\omega^{n+\frac{1}{2}}\|_{-1,h}\cdot
		\|\nabla_{h}\Delta_{h}(\frac{3}{4}\tilde{e}^{n+1}\!+\!\frac{1}{4}\tilde{e}^{n-1})\|_{2}\\
		&\le \frac{3C^{2}}{\mathcal{M}_{0}\varepsilon^{2}}\|\omega^{n+\frac{1}{2}}\|_{2}^{2}
		\!+\!\frac{\mathcal{M}_{0}\varepsilon^{2}}{12}\|\nabla_{h}\Delta_{h}(\frac{3}{4}\tilde{e}^{n+1}\!+\!\frac{1}{4}\tilde{e}^{n-1})\|_{2}^{2}.
	\end{split}
\end{equation}
We now analyze the inner products involving the terms of $\tilde{\mu}^{n+\frac{1}{2}}$ in~\eqref{refine inner}. First, for the nonlocal term, Lemma~\ref{lemma nonlocal} together with the triangle inequality directly provides an upper bound for the extrapolated discrete mobility,
\begin{equation*}
\|\breve{\mathcal{M}}^{n+\frac{1}{2}}\|_{\infty}
\le \frac{3}{2}\|\breve{\mathcal{M}}(\phi^{n})\|_{\infty}
+\frac{1}{2}\|\breve{\mathcal{M}}(\phi^{n-1})\|_{\infty}
=2\mathcal{M}_{1}.
\end{equation*}
As a result,
\begin{equation}\label{2refine sigma}
	\begin{split}
		&\quad-\sigma\left[\nabla_{h}(-\Delta_{h})^{-1}\tilde{e}^{n+\frac{1}{2}},
		\breve{\mathcal{M}}^{n+\frac{1}{2}}\nabla_{h}\Delta_{h}(\frac{3}{4}\tilde{e}^{n+1}+\frac{1}{4}\tilde{e}^{n-1})\right]\\
		&\ge -\sigma\mathcal{M}_{1}
		\|\nabla_{h}\tilde{\psi}^{n+1}+\nabla_{h}\tilde{\psi}^{n}\|_{2}\cdot
		\|\nabla_{h}\Delta_{h}(\frac{3}{4}\tilde{e}^{n+1}+\frac{1}{4}\tilde{e}^{n-1})\|_{2}\\
		&\ge -\frac{3\sigma^2\mathcal{M}_{1}^{2}}{\mathcal{M}_{0}\varepsilon^{2}}
		\left(2C\|\tilde{e}^{n+1}\|_{2}^{2}+2C\|\tilde{e}^{n}\|_{2}^{2}\right)
		-\frac{\mathcal{M}_{0}\varepsilon^2}{12}\|\nabla_{h}\Delta_{h}(\frac{3}{4}\tilde{e}^{n+1}+\frac{1}{4}\tilde{e}^{n-1})\|_{2}^{2},
	\end{split}
\end{equation}
where $\tilde{\psi}^{k}=(-\Delta_{h})^{-1}\tilde{\phi^{k}},\, k=n,\, n+1$. Similarly, 
\begin{equation}\label{2refine theta}
	\begin{split}
		&\quad\theta_{0}\left[\nabla_{h}(\frac{3}{2}\tilde{e}^{n}-\frac{1}{2}\tilde{e}^{n-1}),\breve{\mathcal{M}}^{n+\frac{1}{2}}\nabla_{h}\Delta_{h}(\frac{3}{4}\tilde{e}^{n+1}+\frac{1}{4}\tilde{e}^{n-1})\right]\\
		&\ge -2\mathcal{M}_{1}\theta_{0}
		\|\nabla_{h}(\frac{3}{2}\tilde{e}^{n}-\frac{1}{2}\tilde{e}^{n-1})\|_{2}
		\|\nabla_{h}\Delta_{h}(\frac{3}{4}\tilde{e}^{n+1}\!+\!\frac{1}{4}\tilde{e}^{n-1})\|_{2}\\
		&\ge -\frac{12\mathcal{M}_{1}^{2}\theta_{0}^{2}}{\mathcal{M}_{0}\varepsilon^{2}}
		\left(3\|\nabla_{h}\tilde{e}^{n}\|_{2}^{2}+\|\nabla_{h}\tilde{e}^{n-1}\|_{2}^{2}\right)
		\!-\!\frac{\mathcal{M}_{0}\varepsilon^{2}}{12}\|\nabla_{h}\Delta_{h}(\frac{3}{4}\tilde{e}^{n+1}\!+\!\frac{1}{4}\tilde{e}^{n-1})\|_{2}^{2}.
	\end{split}
\end{equation}
The surface diffusion term is controlled by the uniform positive lower bound of the extrapolated discrete mobility (see Remark~\ref{check mob positive}), 
\begin{equation}\label{2refine epsilon}
	\begin{split}
		&\quad\varepsilon^{2}\left[\nabla_{h}\Delta_{h}(\frac{3}{4}\tilde{e}^{n+1}+\frac{1}{4}\tilde{e}^{n-1}),\breve{\mathcal{M}}^{n+\frac{1}{2}}\nabla_{h}\Delta_{h}(\frac{3}{4}\tilde{e}^{n+1}+\frac{1}{4}\tilde{e}^{n-1})\right]\\
		&\ge\frac{\mathcal{M}_{0}\varepsilon^2}{2}\|\nabla_{h}\Delta_{h}(\frac{3}{4}\tilde{e}^{n+1}+\frac{1}{4}\tilde{e}^{n-1})\|_{2}^{2}.
	\end{split}
\end{equation}
Estimates~\eqref{2refine sigma} - \eqref{2refine epsilon} complete the analysis of the linear parts in the expansion of $\tilde{\mu}^{n+\frac{1}{2}}$. For the nonlinear parts, we set $NLE^{n}:=\sum_{k=1}^{6}\mathcal{N}_{k}^{n}$ and, under the time step restriction $\tau \leq 1$,
\begin{equation*}
	\begin{split}
		&\quad-\left[\nabla_{h}\Big(\mathcal{N}_{1}^{n}+\mathcal{N}_{2}^{n}+\tau\sum_{k=3}^{6}\mathcal{N}_{k}^{n}\Big),
		\breve{\mathcal{M}}^{n+\frac{1}{2}}\nabla_{h}\Delta_{h}(\frac{3}{4}\tilde{e}^{n+1}+\frac{1}{4}\tilde{e}^{n-1})\right]\\
		&\ge-\langle\nabla_{h}NLE^{n},\breve{\mathcal{M}}^{n+\frac{1}{2}}\nabla_{h}\Delta_{h}(\frac{3}{4}\tilde{e}^{n+1}+\frac{1}{4}\tilde{e}^{n-1})\rangle\\
		&\ge -2\mathcal{M}_{1}\|\nabla_{h}NLE^{n}\|_{2}\cdot
		\|\nabla_{h}\Delta_{h}(\frac{3}{4}\tilde{e}^{n+1}+\frac{1}{4}\tilde{e}^{n-1})\|_{2}.
	\end{split}
\end{equation*}
Combining the nonlinear estimates from Lemma~\ref{nonlinear es in refine 2nd} with Young's inequality yields
\begin{equation*}
	\begin{split}
		&\quad\|\nabla_{h}NLE^{n}\|_{2}\cdot\|\nabla_{h}\Delta_{h}(\frac{3}{4}\tilde{e}^{n+1}\!+\!\frac{1}{4}\tilde{e}^{n-1})\|_{2}\\
		&\le\sum_{k=1}^{6}\left(\|\nabla_{h}\mathcal{N}_{k}^{n}\|_{2}\cdot
		\|\nabla_{h}\Delta_{h}(\frac{3}{4}\tilde{e}^{n+1}\!+\!\frac{1}{4}\tilde{e}^{n-1})\|_{2}\right)\\
		&\le\frac{36\mathcal{M}_{1}}{\mathcal{M}_{0}\varepsilon^{2}}
		\sum_{k=1}^{6}\|\nabla_{h}\mathcal{N}_{k}^{n}\|_{2}^{2}
		+\frac{6\mathcal{M}_{0}\varepsilon^{2}}{144\mathcal{M}_{1}}
		\|\nabla_{h}\Delta_{h}(\frac{3}{4}\tilde{e}^{n+1}\!+\!\frac{1}{4}\tilde{e}^{n-1})\|_{2}^{2}\\
		&=\frac{36\mathcal{M}_{1}}{\mathcal{M}_{0}\varepsilon^{2}}
		2(C^{(1)}\!+\! C^{(2)})\left(\|\nabla_{h}\tilde{e}^{n+1}\|_{2}^{2}\!+\!\|\nabla_{h}\tilde{e}^{n}\|_{2}^{2}\right)
		\!+\!\frac{\mathcal{M}_{0}\varepsilon^{2}}{24\mathcal{M}_{1}}
		\|\nabla_{h}\Delta_{h}(\frac{3}{4}\tilde{e}^{n+1}\!+\!\frac{1}{4}\tilde{e}^{n-1})\|_{2}^{2}.
	\end{split}
\end{equation*}
It therefore follows that
\begin{equation}\label{2refine nonlinear}
	\begin{split}
		&\quad-\left[\nabla_{h}\Big(\mathcal{N}_{1}^{n}+\mathcal{N}_{2}^{n}+\tau\sum_{k=3}^{6}\mathcal{N}_{k}^{n}\Big),
		\breve{\mathcal{M}}^{n+\frac{1}{2}}\nabla_{h}\Delta_{h}(\frac{3}{4}\tilde{e}^{n+1}+\frac{1}{4}\tilde{e}^{n-1})\right]\\
		&\ge-4\mathcal{M}_{1}(C^{(1)}+C^{(2)})\frac{36\mathcal{M}_{1}}{\mathcal{M}_{0}\varepsilon^{2}}\left(\|\nabla_{h}\tilde{e}^{n+1}\|_{2}^{2}+\|\nabla_{h}\tilde{e}^{n}\|_{2}^{2}\right)\\
		&\quad-\frac{\mathcal{M}_{0}\varepsilon^{2}}{12}
		\|\nabla_{h}\Delta_{h}(\frac{3}{4}\tilde{e}^{n+1}+\frac{1}{4}\tilde{e}^{n-1})\|_{2}^{2}.
	\end{split}
\end{equation}
Notice that
\begin{equation*}
	\begin{split}
		&\quad\|\nabla_{h}\Delta_{h}(\frac{3}{4}\tilde{e}^{n+1}+\frac{1}{4}\tilde{e}^{n-1})\|_{2}^{2}\\
		&\ge \frac{9}{16}\|\nabla_{h}\Delta_{h}\tilde{e}^{n+1}\|_{2}^{2}
		+\frac{1}{16}\|\nabla_{h}\Delta_{h}\tilde{e}^{n-1}\|_{2}^{2}
		-\frac{3}{16}(\|\nabla_{h}\Delta_{h}\tilde{e}^{n+1}\|_{2}^{2}
		+\|\nabla_{h}\Delta_{h}\tilde{e}^{n-1}\|_{2}^{2})\\
		&=\frac{3}{8}\|\nabla_{h}\Delta_{h}\tilde{e}^{n+1}\|_{2}^{2}
		-\frac{1}{8}\|\nabla_{h}\Delta_{h}\tilde{e}^{n-1}\|_{2}^{2}.
	\end{split}
\end{equation*}
Substituting~\eqref{2refine time discrete} - \eqref{2refine nonlinear} back into~\eqref{refine inner} and applying the discrete Sobolev and Poincar{\'e} inequalities, we finally obtain
\begin{equation*}
	\begin{split}
		&\quad\frac{1}{2\tau}\left(\|\nabla_{h}\tilde{e}^{n+1}\|_{2}^{2}-\|\nabla_{h}\tilde{e}^{n}\|_{2}^{2}\right)
		+\frac{1}{8\tau}\left(\|\nabla_{h}(\tilde{e}^{n+1}-\tilde{e}^{n})\|_{2}^{2}-\|\nabla_{h}(\tilde{e}^{n}-\tilde{e}^{n-1})\|_{2}^{2}\right)\\
		&\qquad+\frac{\mathcal{M}_{0}\varepsilon^{2}}{12}\left(\frac{3}{8}\|\nabla_{h}\Delta_{h}\tilde{e}^{n+1}\|_{2}^{2}
		-\frac{1}{8}\|\nabla_{h}\Delta_{h}\tilde{e}^{n-1}\|_{2}^{2}\right)\\
		&\le A_1\|\nabla_{h}\tilde{e}^{n+1}\|_{2}^{2}
		+A_2\|\nabla_{h}\tilde{e}^{n}\|_{2}^{2}
		+A_3\|\nabla_{h}\tilde{e}^{n-1}\|_{2}^{2}
		+\frac{3C^{2}}{\mathcal{M}_{0}\varepsilon^{2}}\|\omega^{n+\frac{1}{2}}\|_{2}^{2},
	\end{split}
\end{equation*}
where the constants take the form
\begin{equation*}
	\begin{aligned}
		A_1 &
		=\frac{6\mathcal{M}_{1}^{2}}{\mathcal{M}_{0}\varepsilon^{2}}\left[24(C^{(1)}\!+\! C^{(2)})+\sigma^2CC_{p}\right],\quad
		A_3 
		=\frac{3}{\mathcal{M}_{0}\varepsilon^{2}}\left[(C^{\star})^2M^2C_{p}+4\mathcal{M}_{1}^{2}\theta_{0}^{2}\right],\\
		A_2 &=\frac{3}{\mathcal{M}_{0}\varepsilon^{2}}
		\left[3(C^{\star})^2M^2C_{p}
		+\mathcal{M}_{1}^{2}\left(48(C^{(1)}+ C^{(2)})+12\theta_{0}^{2}+2\sigma^2C \, C_{p}\right)\right].
	\end{aligned}
\end{equation*}
For sufficiently small time step size $\tau$ and spatial mesh size $h$, the discrete Gronwall inequality yields the following higher-order convergence estimate:
\begin{equation}\label{final 2 2nd}
	\|\nabla_{h} \tilde{e}^{n+1}\|_{2}+\Big(\frac{\mathcal{M}_{0}\varepsilon^2}{24}\tau\sum_{k=1}^{n+1}\|\nabla_{h}\Delta_{h}\tilde{e}^{k}\|_{2}^{2} \Big)^{1/2}\le C(\tau^{3}+h^{4}).
\end{equation} 

Furthermore, estimate~\eqref{final 2 2nd} implies that the numerical error function satisfies
\begin{equation*}
	\|\nabla_{h}\tilde{e}^{n+1}\|_{2}
	\le C(\tau^3+h^4)
	\le\tau^{\frac{11}{4}}+h^{\frac{15}{4}}.
\end{equation*}
Consequently, for the time step $t_{n+1}$, the a priori assumption~\eqref{priori assum 2nd} remains valid, ensuring the effectiveness of the inductive argument. The convergence result stated in Theorem~\ref{error analysis 2order} then follows directly from the higher-order convergence estimate~\eqref{final 2 2nd}, the boundedness of the constructed fields $\Phi_{\tau}$ and $\Phi_{h}$, and the projection estimate~\eqref{standard estimate}. This completes the proof of Theorem~\ref{error analysis 2order}.

\section{Numerical experiments}
An efficient nonlinear Full Approximation Storage (FAS) multigrid solver is deployed to solve the proposed second order numerical scheme~\eqref{2order}. For algorithmic implementations of similar solvers, we refer the readers to \cite{chen19b}. Furthermore, while our theoretical framework is fully valid for both two- and three-dimensional situations, the subsequent numerical experiments are restricted to the two-dimensional setting to balance computational focus and clarity.

\subsection{Convergence test}
We first conduct a convergence test. The computational domain is set to $\Omega=(0,2)^2$, with physical parameters $\varepsilon=0.2,\,\theta_{0}=3.0$, and $\sigma=1\times10^{-3}$. The smooth initial condition is given by
\begin{equation}\label{conver initial}
	\phi(x,y,0)=0.25\cos 2\pi x\cos 2\pi y+0.6\cos\pi x\cos 3\pi y.
\end{equation}
Referring to the conclusions in \cite{cahn1994overview}, the variable mobility is selected as 
\begin{equation}\label{test mob}
	\mathcal{M}(\phi)=0.9 (1-\phi)(1+\phi)+0.1.
\end{equation}
Furthermore, the refinement path is set as $\tau=0.01h$, so that the second order accuracy in both time and space could be confirmed. To determine the convergence rate, the error function is computed between approximate solutions obtained by successively finer mesh sizes. Results of the numerical convergence test at the final time $T=0.1$, presented in Table~\ref{table ct}, confirm the second order accuracy in space and time.
	\begin{table}[ht]
	\centering
	\caption{Numerical convergence test with initial data \eqref{conver initial}}
	\label{table ct}
		\begin{tabular}[c]{cccccc}
			\toprule
			Grid size & $16^{2}-32^{2}$ & $32^{2}-64^{2}$ & $64^{2}-128^{2}$ & $128^{2}-256^{2}$ & $256^{2}-512^{2}$ \\
			\midrule
			$L^{2}$ error & $9.6786\mathrm{E-}05$ & $2.3800\mathrm{E-}05$ & $5.9160\mathrm{E-}06$ & $1.4914\mathrm{E-}06$ & $3.7573\mathrm{E-}07$\\
			$L^{2}$ rate &  & $2.0238$ & $2.0083$ & $1.9880$ & $1.9890$ \\
			$L^{\infty}$ error & $2.2333\mathrm{E-}04$ & $5.6895\mathrm{E-}05$ & $1.4197\mathrm{E-}05$ & $3.5822\mathrm{E-}06$ & $9.0276\mathrm{E-}07$ \\
			$L^{\infty}$ rate &  & $1.9728$ & $2.0027$ & $1.9867$ & $1.9884$ \\
			\bottomrule
		\end{tabular}
\end{table}

\subsection{Spinodal decomposition}
The logarithmic nonlinearity and the nonlocal term tend to generate numerous small structures. In this subsection, we present numerical simulations to illustrate the coarsening dynamics and to examine the positivity preservation, energy decay, and mass conservation simultaneously.

The computational domain is $\Omega=(0,2\pi)^{2}$, with parameters $\varepsilon=0.06$, $\theta_{0}=3.0$, $h=1/256$, and $\tau=1.0\times 10^{-3}$. We initiate the two-step numerical scheme by taking $\phi^{-1}=\phi^0$ with the data $\phi^0$ specified as
\begin{equation}\label{initial}
	\phi_{i,j}^{0}=\overline{\phi^{0}}+0.01\, r_{i,j},
\end{equation}
where $r_{i,j}$ denotes a uniformly distributed random variable in $[-1,1]$ with zero mean. The variable mobility is the same as in~\eqref{test mob}. The time snapshots of the
evolution for $\sigma = 5$, $\overline{\phi^{0}} =0$ are presented in Figure~\ref{randomfigure_2varmobCH_1phirandom256sigma5initial0}. 
\begin{figure}[h!]
	\centering
	\begin{minipage}{\textwidth}
		\centering
		\begin{subfigure}[b]{0.28\textwidth}
			\includegraphics[width=\linewidth]{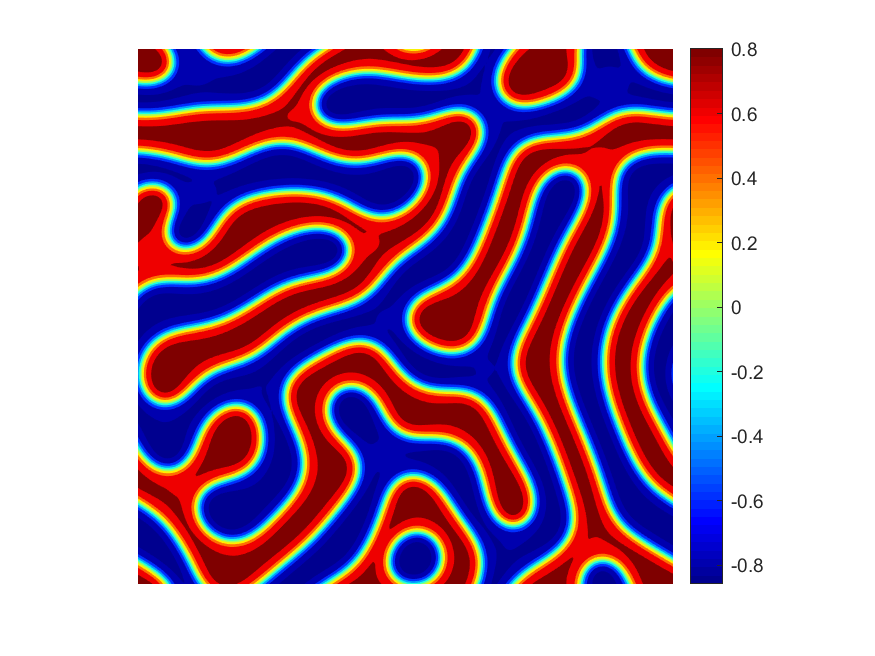}
			\caption{$t=0.25$}
		\end{subfigure}
		\hspace{-1cm}
		\begin{subfigure}[b]{0.28\textwidth}
			\includegraphics[width=\linewidth]{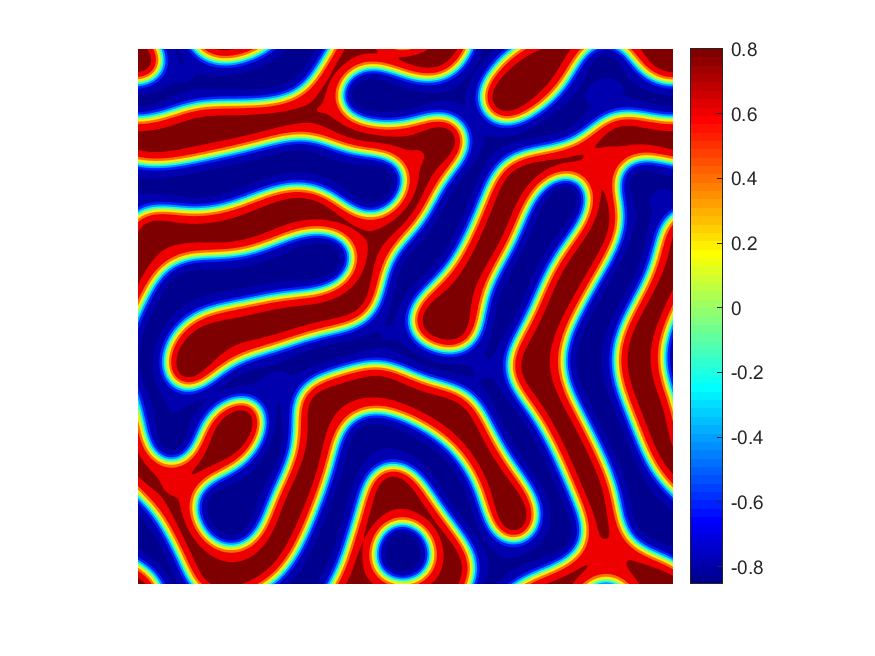}
			\caption{$t=1$}
		\end{subfigure}
		\hspace{-1cm}
		\begin{subfigure}[b]{0.28\textwidth}
			\includegraphics[width=\linewidth]{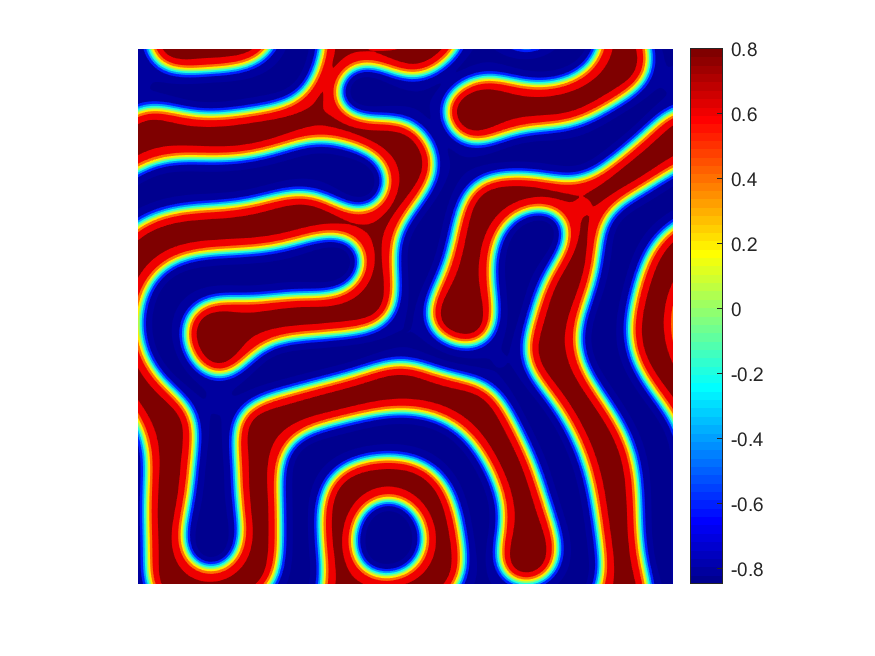}
			\caption{$t=100$}
		\end{subfigure}
		\hspace{-1cm}
		\begin{subfigure}[b]{0.28\textwidth}
			\includegraphics[width=\linewidth]{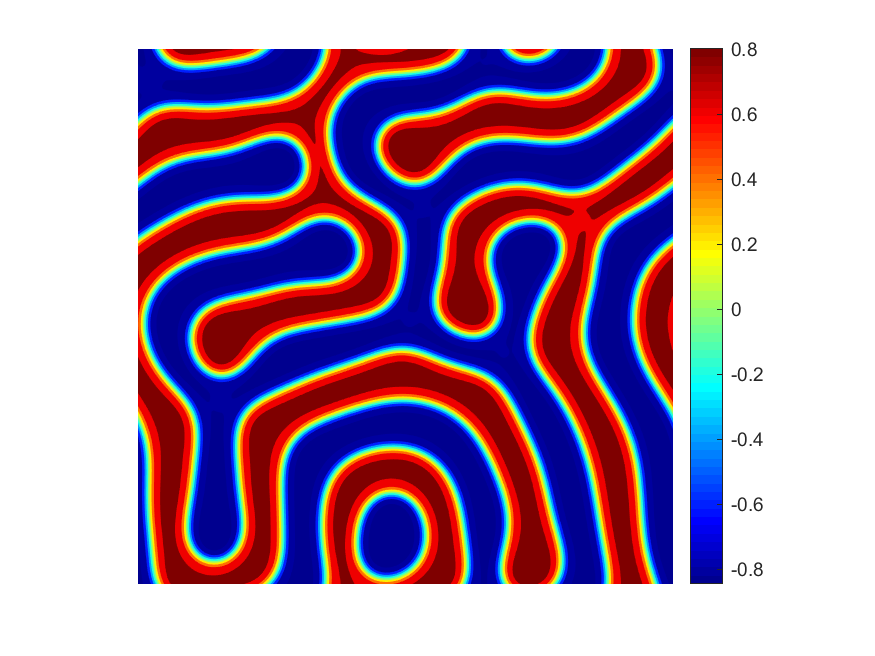}
			\caption{$t=1200$}
		\end{subfigure}
	\end{minipage}
	\caption{Evolution of the phase variable $\phi$ at selected time instants. The simulation uses the mobility and initial condition specified in~\eqref{test mob} and~\eqref{initial}, with parameters $\sigma = 5$ and $\overline{\phi^{0}} =0$.}
	\label{randomfigure_2varmobCH_1phirandom256sigma5initial0}
\end{figure}
Due to the presence of the nonlocal term, the system exhibits numerous microstructures right from the start until to the energy stabilizes.

To demonstrate the phase separation property of the NCH system, we present the minimum and maximum values of the phase variable at the specified time instants in Table~\ref{maxmin_sigma5initial0}. Although these extrema are in close proximity to the theoretical singular limit values of $-1$ and $1$, a clear separation from these limits is observed in the numerical results.

\begin{table}[h!]
	\centering
	\caption{The extrema of the phase variable $\phi$ during spinodal decomposition at the indicated time instants, with $\overline{\phi^{0}} =0$, $\sigma=5$.}
	\label{maxmin_sigma5initial0}
	\begin{tabular}[c]{ccc}
		\toprule
		Time instants  & the minimum values  & the maximum values \\
		\midrule
		0.25 & -0.8580659 & 0.8510982 \\
		0.50 & -0.8560963 & 0.8500969 \\
		1    & -0.8539558 & 0.8479956 \\
		5    & -0.8470644 & 0.8416601 \\
		10   & -0.8467009 & 0.8418581 \\
		100  & -0.8470313 & 0.8422207 \\
		600  & -0.8441764 & 0.8410790 \\
		\bottomrule
	\end{tabular}
\end{table}

Figure~\ref{mass_energy} presents a log-log plot of energy versus time, where the discrete energy is defined as in Theorem~\ref{dis ener theo}, along with the mass error relative to the initial state. The results clearly confirm that scheme~\eqref{2order} is both energy stable and mass conservative, consistent with our theoretical analysis.
\begin{figure}[h!]
	\centering
	\begin{minipage}{\textwidth}
		\centering
		\begin{subfigure}[b]{0.45\textwidth}
			\includegraphics[width=\linewidth]{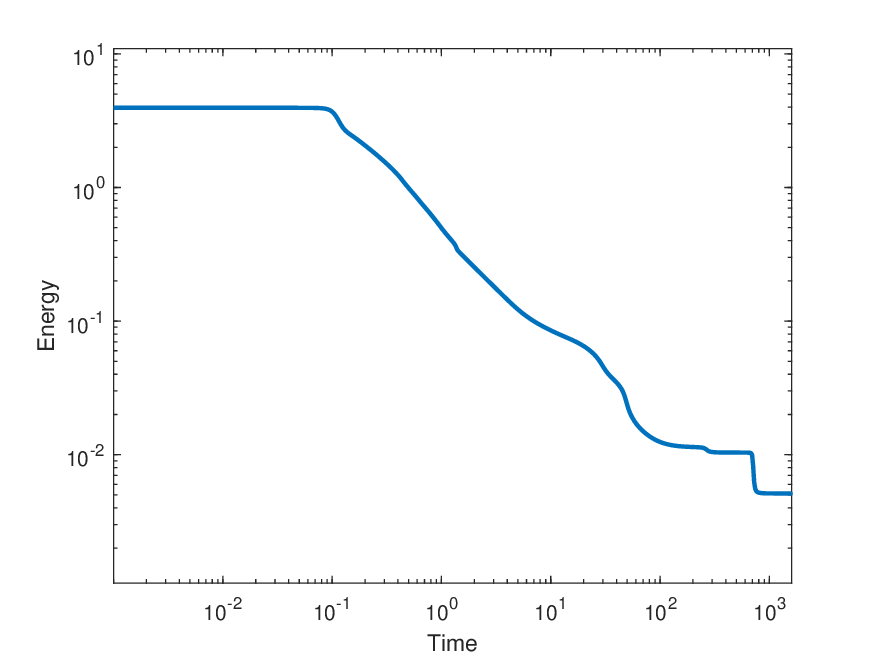}
			\caption{Energy decay}
		\end{subfigure}
		\hfill
		\begin{subfigure}[b]{0.45\textwidth}
			\includegraphics[width=\linewidth]{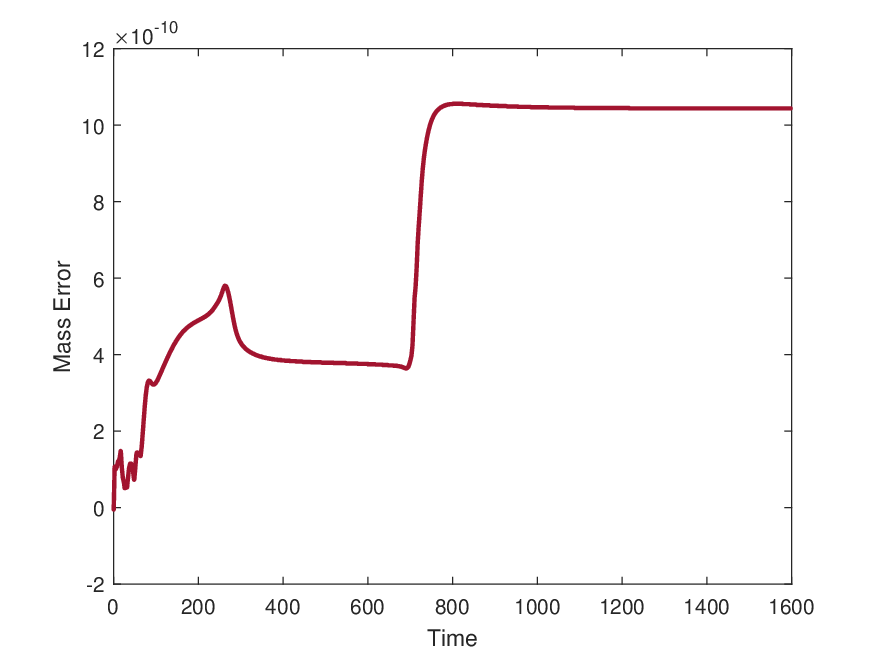}
			\caption{Mass error}
		\end{subfigure}
	\end{minipage}
	\caption{Energy decay and mass error for $\sigma = 5$, $\overline{\phi^{0}} =0$.}
	\label{mass_energy}
\end{figure}

Same experiments are performed for the case $\sigma=10$, $\overline{\phi^{0}} = 0.3$. Due to the different initial conditions, specifically the ratio between the two components, together with the enhanced nonlocal interactions, the coarsening dynamics exhibit notably distinct structures (see Figure~\ref{randomfigure_2varmobCH_1phirandom256sigma10initial03}). Similarly, the minimum and maximum values of the phase variable over the corresponding time sequence are presented in Table~\ref{maxmin_sigma10initial03}. As before, a clear separation is observed between the numerical solution and the singular limit values. Furthermore, Figure~\ref{mass_energy1} demonstrates that the scheme continues to preserve the expected energy decay and mass conservation properties.

\begin{figure}[h!]
	\centering
	\begin{minipage}{\textwidth}
		\centering
		\begin{subfigure}[b]{0.28\textwidth}
			\includegraphics[width=\linewidth]{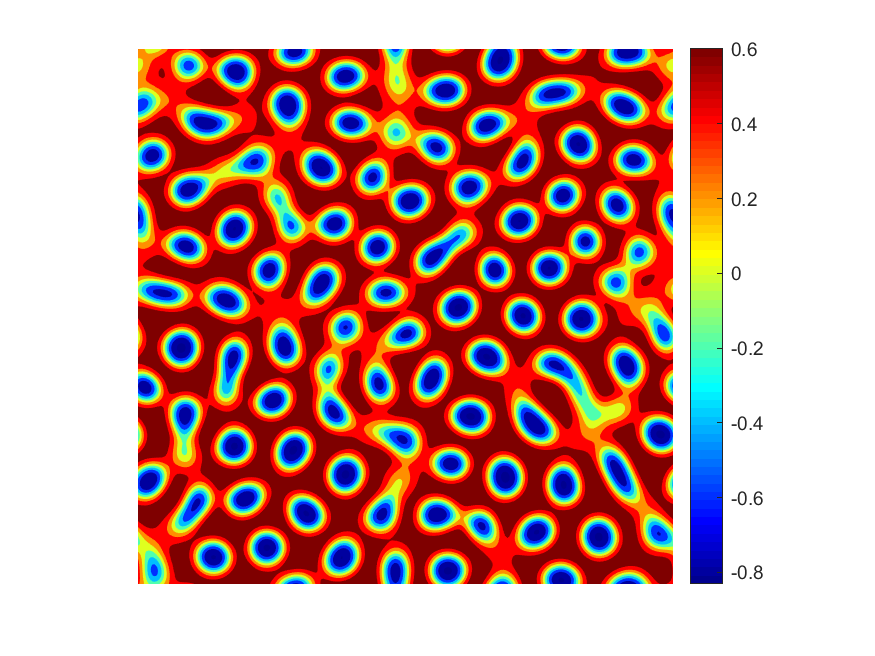}
			\caption{$t=0.25$}
		\end{subfigure}
		\hspace{-1cm}
		\begin{subfigure}[b]{0.28\textwidth}
			\includegraphics[width=\linewidth]{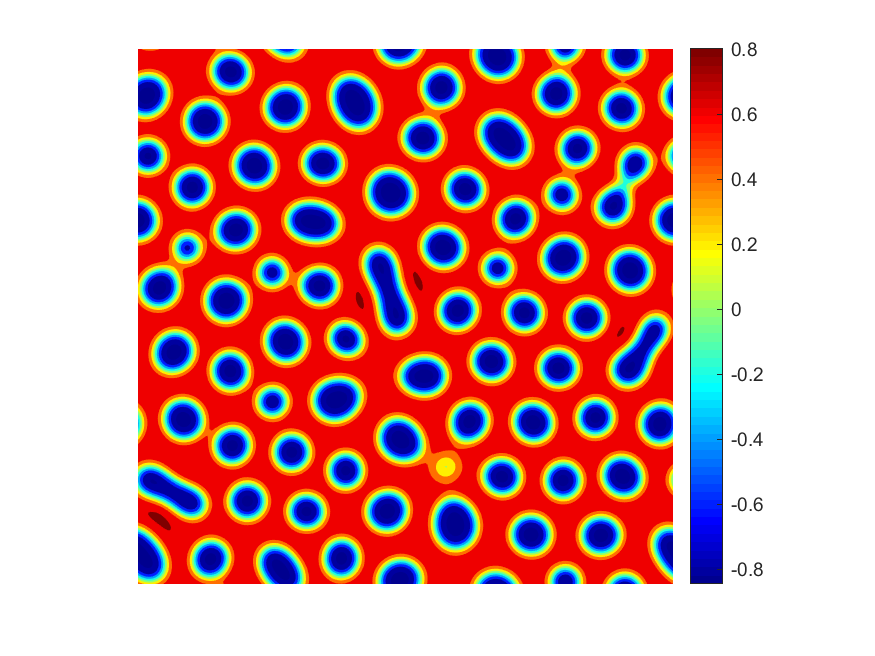}
			\caption{$t=1$}
		\end{subfigure}
		\hspace{-1cm}
		\begin{subfigure}[b]{0.28\textwidth}
			\includegraphics[width=\linewidth]{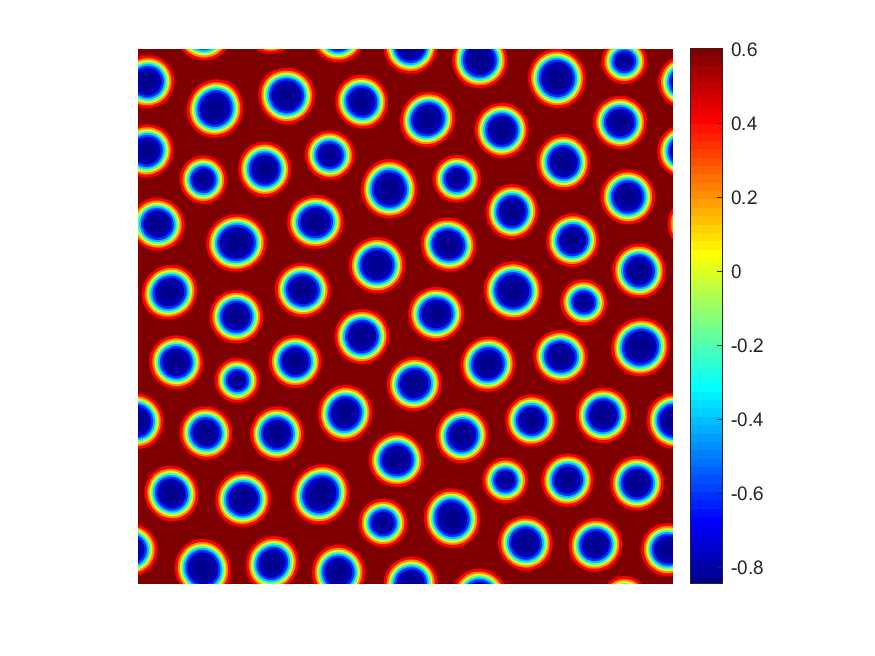}
			\caption{$t=100$}
		\end{subfigure}
		\hspace{-1cm}
		\begin{subfigure}[b]{0.28\textwidth}
			\includegraphics[width=\linewidth]{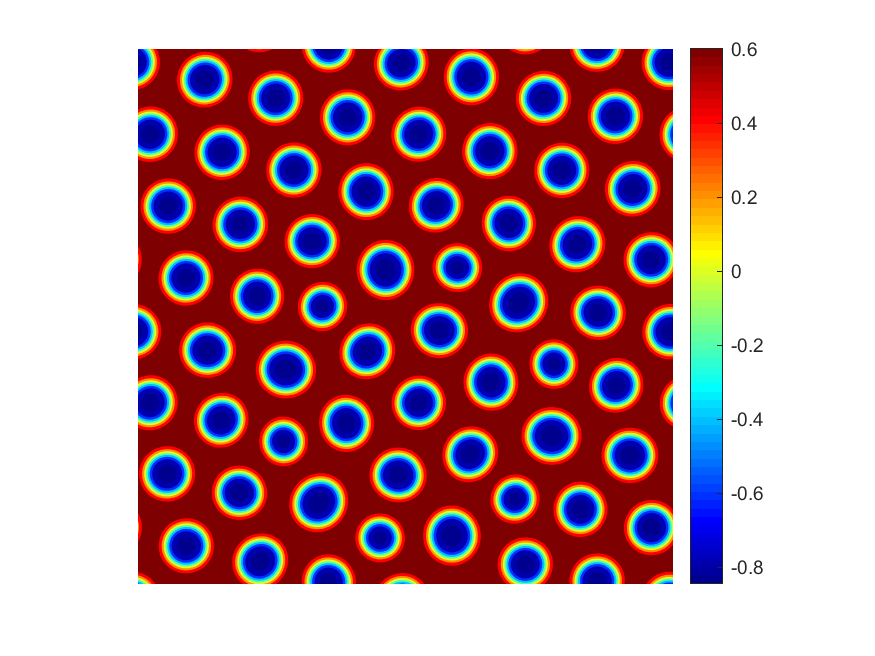}
			\caption{$t=600$}
		\end{subfigure}
	\end{minipage}
	
	\caption{Evolution of the phase variable $\phi$ at selected time instants. The simulation uses the mobility and initial condition specified in~\eqref{test mob} and~\eqref{initial}, with parameters $\sigma = 10$ and $\overline{\phi^{0}} =0.3$.}
	\label{randomfigure_2varmobCH_1phirandom256sigma10initial03}
\end{figure}

\begin{table}[h!]
	\centering
	\caption{The extrema of the phase variable $\phi$ during spinodal decomposition at the indicated time instants, with $\overline{\phi^{0}} =0.3$, $\sigma=10$.}
	\label{maxmin_sigma10initial03}
	\begin{tabular}[c]{ccc}
		\toprule
		Time instants  & the minimum values  & the maximum values \\
		\midrule
		0.25 & -0.8310225 & 0.7858210 \\
		0.50 & -0.8454229 & 0.8016763 \\
		1    & -0.8454691 & 0.8105549 \\
		5    & -0.8453975 & 0.7871832 \\
		10   & -0.8452209 & 0.7865080 \\
		100  & -0.8454444 & 0.7871405 \\
		600  & -0.8451800 & 0.7844330 \\
		\bottomrule
	\end{tabular}
\end{table}

\begin{figure}[h!]
	\centering
	\begin{minipage}{\textwidth}
		\centering
		\begin{subfigure}[b]{0.48\textwidth}
			\includegraphics[width=\linewidth]{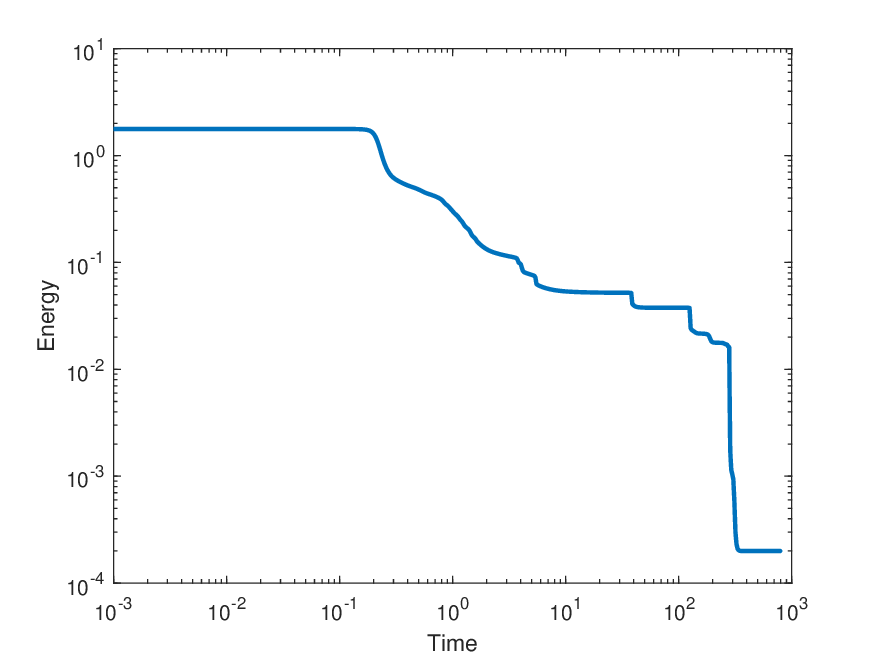}
			\caption{Energy decay}
		\end{subfigure}
		\hfill
		\begin{subfigure}[b]{0.48\textwidth}
			\includegraphics[width=\linewidth]{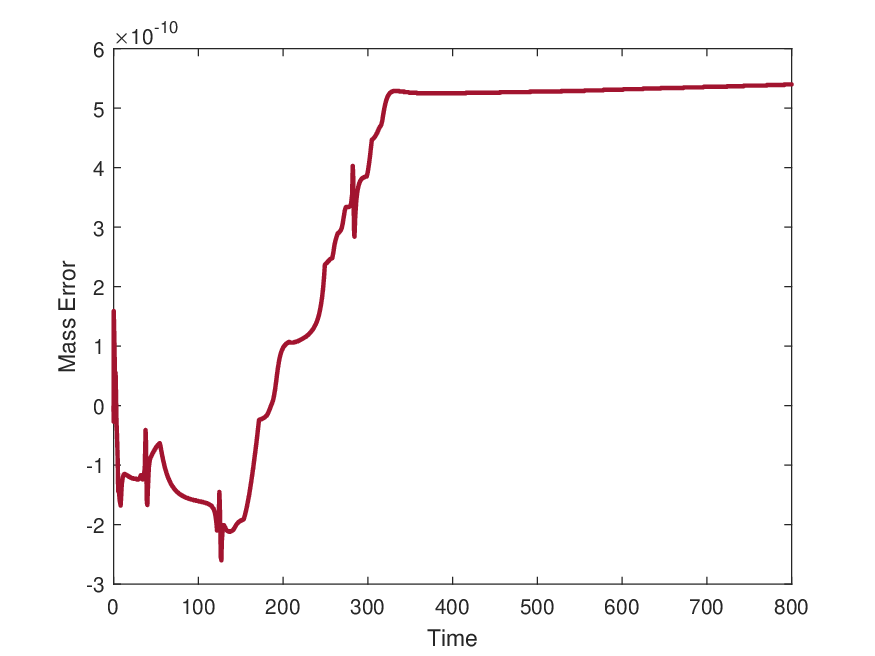}
			\caption{Mass error}
		\end{subfigure}
	\end{minipage}
	\caption{Energy decay and mass error for $\sigma = 10$, $\overline{\phi^{0}} =0.3$.}
	\label{mass_energy1}
\end{figure}

\section{Conclusions}
In this paper, we develop and analyze a second order accurate finite difference scheme in both time and space for the nonlocal Cahn-Hilliard (NCH) system with variable mobility and the singular Flory-Huggins logarithmic potential. The scheme combines a modified Crank-Nicolson discretization with a nonlinear artificial regularization, and treats the mobility explicitly to ensure strict ellipticity and reduce computational cost. Our rigorous analysis shows that the scheme enjoys all essential properties: unique solvability, positivity preservation, and discrete energy stability. The singular feature of the logarithmic function plays an important role in the proof of the positivity-preserving property pointwisely. To overcome the analytical difficulties arising from the variable mobility and the singular potential, we establish a convergence framework based on higher-order consistency analysis combined with rough and refined error estimates. Numerical experiments confirm the theoretical results and demonstrate the efficiency of our scheme.

\section*{Acknowledgements}
The research of Z.R.~Zhang was partially supported by the NSFC No.12471363 and No.12231003.

\bibliographystyle{amsplain}
\bibliography{reference}

\appendix
\renewcommand{\thesection}{Appendix \Alph{section}:}  
\renewcommand{\theequation}{\Alph{section}.\arabic{equation}}

\section{Proof of Proposition \ref{prop higher order}}
We begin with the local truncation error analysis for the temporal discretization by employing a Taylor expansion in time, in conjunction with the projection estimate \eqref{standard estimate}:
\begin{equation}
	\begin{aligned}
		\frac{\Phi_{N}^{n+1}-\Phi_{N}^{n}}{\tau}
		&=\nabla\cdot\left(\mathcal{M}_N^{n+\frac{1}{2}}\nabla\mu_N^{n+\frac{1}{2}}\right)
		+\tau^2(K^{(0)})^{n+\frac{1}{2}}+O(\tau^3)+O(h^{m_0}),
		\label{consis1 2nd}\\
		\mu_N^{n+\frac{1}{2}}
		&=\frac{G(1+\Phi_{N}^{n+1})-G(1+\Phi_{N}^{n})}{\Phi_{N}^{n+1}-\Phi_{N}^{n}}
		+\frac{G(1-\Phi_{N}^{n+1})-G(1-\Phi_{N}^{n})}{\Phi_{N}^{n+1}-\Phi_{N}^{n}}\\
		&\quad-\theta_{0}\check{\Phi}_{N}^{n+\frac{1}{2}}
		-\varepsilon^2\Delta\hat{\Phi}_{N}^{n+\frac{1}{2}}
		+\sigma(-\Delta)^{-1}\left(\Phi_{N}^{n+\frac{1}{2}}
		-\overline{\Phi_{N}^{n+\frac{1}{2}}}\right)\\
		&\quad+\tau \,
		\Big(\ln(1+\Phi_{N}^{n+1})-\ln(1+\Phi_{N}^{n})
		-\ln(1-\Phi_{N}^{n+1})+\ln(1-\Phi_{N}^{n})\Big),
	\end{aligned}
\end{equation}
where
\begin{equation}\label{check Phi N}
	\begin{aligned}
		&\mathcal{M}_N^{n+\frac{1}{2}}=\frac{3}{2}\mathcal{M}(\Phi_{N}^{n})-\frac{1}{2}\mathcal{M}(\Phi_{N}^{n-1}),\quad
		&&\check{\Phi}_{N}^{n+\frac{1}{2}}=\frac{3}{2}\Phi_{N}^{n}-\frac{1}{2}\Phi_{N}^{n-1},\\
		&\hat{\Phi}_{N}^{n+\frac{1}{2}}=\frac{3}{4}\Phi_{N}^{n+1}+\frac{1}{4}\Phi_{N}^{n-1},\;
		&&\Phi_{N}^{n+\frac{1}{2}}=\frac{1}{2}\Phi_{N}^{n+1}+\frac{1}{2}\Phi_{N}^{n}.
	\end{aligned}
\end{equation}
Here, the projection accuracy order satisfies $m_{0}\ge 4$, and the function $K^{(0)}$, which depends only on the higher-order derivatives of $\Phi_{N}$, is sufficiently smooth in the sense that its derivatives are bounded. Moreover, by the mass conservation identity 
\begin{equation*}
\int_{\Omega}(\Phi_{N}^{n+1}-\Phi_{N}^{n})\,\mathrm{d}\mathbf{x}=0,\quad
\int_{\Omega}\nabla\cdot\left(\mathcal{M}_N^{n+\frac{1}{2}}\nabla\mu_N^{n+\frac{1}{2}}\right)\,\mathrm{d}\mathbf{x}=0,
\end{equation*}
we deduce that $K^{(0)}$ has a mean-zero property,
\begin{equation}\label{g0 2nd}
	\int_{\Omega}(K^{(0)})^{n+\frac{1}{2}} \,\mathrm{d}\mathbf{x}=0.
\end{equation}
The time-correction function $\Phi_{\tau}$ is defined as the solution to
\begin{align}
	&\partial_{t}\Phi_{\tau}
	=\nabla\cdot\left(\mathcal{M}(\Phi_{N})\nabla\mathcal{V}_{\tau}
	+\mathcal{M}'(\Phi_{N})\Phi_{\tau}\nabla\mathcal{W}_{N}\right)-K^{(0)},
	\label{time correction function 1 2nd}\\
	&\mathcal{V}_{\tau}
	=\frac{\Phi_{\tau}}{1+\Phi_{N}}+\frac{\Phi_{\tau}}{1-\Phi_{N}}
	-\theta_{0}\Phi_{\tau}-\varepsilon^{2}\Delta\Phi_{\tau}
	+\sigma(-\Delta)^{-1}\Phi_{\tau},\\
	&\mathcal{W}_{N}
	=\ln(1+\Phi_{N})-\ln(1-\Phi_{N})
	-\theta_{0}\Phi_{N}-\varepsilon^{2}\Delta\Phi_{N}
	+\sigma(-\Delta)^{-1}(\Phi_{N}-\overline{\Phi_{N}}),
	\label{time correction function 3 2nd}
\end{align}
with periodic boundary conditions. The existence of a solution to this linear PDE system is straightforward. Since $\Phi_{\tau}$ depends only on the projection solution $\Phi_{N}$, which possesses sufficient smoothness, all derivatives of $\Phi_{\tau}$ are bounded. Trivial initial data $\Phi_{\tau}(\cdot,t=0)\equiv 0$ could be imposed to \eqref{time correction function 1 2nd} - \eqref{time correction function 3 2nd}. In combination with the identity \eqref{g0 2nd}, we get 
\[
\int_{\Omega}\partial_{t}\Phi_{\tau}\,\mathrm{d}\mathbf{x}=\int_{\Omega}\nabla\cdot\left(\mathcal{M}(\Phi_{N})\nabla\mathcal{V}_{\tau}+\mathcal{M}'(\Phi_{N})\Phi_{\tau}\nabla\mathcal{W}_{N}\right)\,\mathrm{d}\mathbf{x}-\int_{\Omega}K^{(0)}\,\mathrm{d}\mathbf{x}=0,
\]
so that the mass conservative property is valid for $\Phi_{\tau}$:
\begin{equation}\label{massconser_phi_deltat 2nd}
	\int_{\Omega}\Phi_{\tau}(\cdot,t)\,\mathrm{d}\mathbf{x}=\int_{\Omega}\Phi_{\tau}(\cdot,t=0)\,\mathrm{d}\mathbf{x}=0,\quad\forall\, t>0.
\end{equation}
An application of the semi-implicit discretization (as given by \eqref{consis1 2nd}) to \eqref{time correction function 1 2nd} - \eqref{time correction function 3 2nd} implies
\begin{align}
	\frac{\Phi_{\tau}^{n+1}-\Phi_{\tau}^{n}}{\tau}
	&=\nabla\cdot\left[\mathcal{M}_{N}^{n+\frac{1}{2}}\nabla\mathcal{V}_{\tau}^{n+\frac{1}{2}}
	+\left(\frac{3}{2}\mathcal{M}'(\Phi_{N}^{n})\Phi_{\tau}^{n}
	-\frac{1}{2}\mathcal{M}'(\Phi_{N}^{n-1})\Phi_{\tau}^{n-1}\right)\nabla\mathcal{W}_{N}^{n+\frac{1}{2}}\right]
	\nonumber\\
	&\quad-(K^{(0)})^{n+\frac{1}{2}}+\tau^2h_1^{n+\frac{1}{2}}+O(\tau^3),
	\label{time corre func 1 2nd}\\
	\mathcal{V}_{\tau}^{n+\frac{1}{2}}
	&=\frac{\Phi_{\tau}^{n+\frac{1}{2}}}{1+\Phi_{N}^{n+\frac{1}{2}}}
	+\frac{\Phi_{\tau}^{n+\frac{1}{2}}}{1-\Phi_{N}^{n+\frac{1}{2}}}
	-\theta_{0}\check{\Phi}_{\tau}^{n+\frac{1}{2}}
	-\varepsilon^{2}\Delta\hat{\Phi}_{\tau}^{n+\frac{1}{2}}
	+\sigma(-\Delta)^{-1}\Phi_{\tau}^{n+\frac{1}{2}},\\
	\mathcal{W}_{N}^{n+\frac{1}{2}}
	&=\frac{G(1+\Phi_{N}^{n+1})-G(1+\Phi_{N}^{n})}{\Phi_{N}^{n+1}-\Phi_{N}^{n}}
	+\frac{G(1-\Phi_{N}^{n+1})-G(1-\Phi_{N}^{n})}{\Phi_{N}^{n+1}-\Phi_{N}^{n}}
	-\theta_{0}\check{\Phi}_{N}^{n+\frac{1}{2}}
	\label{time corre func 3 2nd}\\
	&\quad -\varepsilon^2\Delta\hat{\Phi}_{N}^{n+\frac{1}{2}}
	+\sigma(-\Delta)^{-1}\left(\Phi_{N}^{n+\frac{1}{2}}
	-\overline{\Phi_{N}^{n+\frac{1}{2}}}\right),
	\nonumber
\end{align}
with
\begin{equation*}
	\check{\Phi}_{\tau}^{n+\frac{1}{2}}
	=\frac{3}{2}\Phi_{\tau}^{n}-\frac{1}{2}\Phi_{\tau}^{n-1},\quad
	\hat{\Phi}_{\tau}^{n+\frac{1}{2}}
	=\frac{3}{4}\Phi_{\tau}^{n+1}+\frac{1}{4}\Phi_{\tau}^{n-1},\quad
	\Phi_{\tau}^{n+\frac{1}{2}}
	=\frac{1}{2}\Phi_{\tau}^{n+1}+\frac{1}{2}\Phi_{\tau}^{n}.
\end{equation*}
$\mathcal{M}_{N}^{n+\frac{1}{2}},\,\check{\Phi}_{N}^{n+\frac{1}{2}},\,\hat{\Phi}_{N}^{n+\frac{1}{2}},\,\Phi_{N}^{n+\frac{1}{2}}$ have been defined in \eqref{check Phi N}.
Therefore, a combination of \eqref{consis1 2nd} and the Fourier projection of \eqref{time corre func 1 2nd} - \eqref{time corre func 3 2nd} leads to the desired third-order temporal consistency estimate for $ \hat{\Phi}_{1}=\Phi_{N}+\tau^2\mathcal{P}_{N}\Phi_{\tau}$:
\begin{equation*}
	\begin{aligned}
		\frac{\hat{\Phi}_{1}^{n+1}-\hat{\Phi}_{1}^{n}}{\tau}
		&=\nabla\cdot\left[\left(\frac{3}{2}\mathcal{M}(\hat{\Phi}_{1}^{n})-\frac{1}{2}\mathcal{M}(\hat{\Phi}_{1}^{n-1})\right)\nabla\hat{\mu}_{1}^{n+\frac{1}{2}}\right]
		+\tau^3(K^{(1)})^{n+\frac{1}{2}}+O(\tau^4)+O(h^{m_0}),\\
		\hat{\mu}_{1}^{n+\frac{1}{2}}
		&=\frac{G(1+\hat{\Phi}_{1}^{n+1})-G(1+\hat{\Phi}_{1}^{n})}{\hat{\Phi}_{1}^{n+1}-\hat{\Phi}_{1}^{n}}
		+\frac{G(1-\hat{\Phi}_{1}^{n+1})-G(1-\hat{\Phi}_{1}^{n})}{\hat{\Phi}_{1}^{n+1}-\hat{\Phi}_{1}^{n}}\\
		&\quad-\theta_{0}\left(\frac{3}{2}\hat{\Phi}_{1}^{n}-\frac{1}{2}\hat{\Phi}_{1}^{n-1}\right)
		-\varepsilon^2\Delta\left(\frac{3}{4}\hat{\Phi}_{1}^{n+1}+\frac{1}{4}\hat{\Phi}_{1}^{n-1},\right)
		+\sigma(-\Delta)^{-1}\left(\hat{\Phi}_{1}^{n+\frac{1}{2}}
		-\overline{\hat{\Phi}_{1}^{n+\frac{1}{2}}}\right)\\
		&\quad+\tau \,
		\Big(\ln(1+\hat{\Phi}_{1}^{n+1})-\ln(1+\hat{\Phi}_{1}^{n})
		-\ln(1-\hat{\Phi}_{1}^{n+1})+\ln(1-\hat{\Phi}_{1}^{n})\Big).
	\end{aligned}
\end{equation*}
In the derivation, the following linearized expansions have been applied:
\begin{equation*}
	\begin{aligned}
		&H_{1+\hat{\Phi}_1^n}^1(1+\hat{\Phi}_1^{n+1})
		=H_{1+\Phi_{N}^n}^1(1+\Phi_{N}^{n+1})+\frac{\tau^2}{2}\left(\frac{\Phi_{\tau}^{n+1}}{1+\Phi_{N}^{n+1}}+\frac{\Phi_{\tau}^{n}}{1+\Phi_{N}^{n}}\right)+O(\tau^{3}),\\
		&\frac{1}{2}\left(\frac{\Phi_{\tau}^{n+1}}{1+\Phi_{N}^{n+1}}+\frac{\Phi_{\tau}^{n}}{1+\Phi_{N}^{n}}\right)
		=\frac{\Phi_{\tau}^{n+\frac{1}{2}}}{1+\Phi_{N}^{n+\frac{1}{2}}}+O(\tau^{2}),\\
		&\ln(1\pm\hat{\Phi}_1)=\ln(1\pm\Phi_{N}\pm\tau^2\mathcal{P}_{N}\Phi_{\tau})
		=\ln(1\pm\Phi_{N})\pm\frac{\tau^2\mathcal{P}_{N}\Phi_{\tau}}{1\pm\Phi_{N}}+O(\tau^{4}),
	\end{aligned}
\end{equation*}
where the property of $H_a^1(x)$ stated in Lemma~\ref{G prop} has been recalled. Moreover, owing to the mass-conservation property of $\Phi_{\tau}$ established in \eqref{massconser_phi_deltat 2nd}, the subsequent identities hold naturally:
\begin{equation*}
	\int_{\Omega}\mathcal{P}_{N}\Phi_{\tau}(\cdot,t)\,\mathrm{d}\mathbf{x}=0,\qquad
	\int_{\Omega}\hat{\Phi}_{1}(\cdot,t)\,\mathrm{d}\mathbf{x}=\int_{\Omega}\Phi_{N}(\cdot,t)\,\mathrm{d}\mathbf{x},\qquad\forall\, t>0.
\end{equation*}
As a result, we see that
\begin{equation}\label{mass conser hat Phi1 2nd}
	\begin{split}
		\frac{1}{\abs{\Omega}}\int_{\Omega}\hat{\Phi}_{1}^{n+1}\,\mathrm{d}\mathbf{x} 
		&= \frac{1}{\abs{\Omega}}\int_{\Omega}\hat{\Phi}_{1}^{n}\,\mathrm{d}\mathbf{x} 
		= \cdots 
		= \frac{1}{\abs{\Omega}}\int_{\Omega}\hat{\Phi}_{1}^{0}\,\mathrm{d}\mathbf{x} \\
		&= \frac{1}{\abs{\Omega}}\int_{\Omega}\Phi_{N}(\cdot,t=0)\,\mathrm{d}\mathbf{x} 
		= \overline{\phi^{0}}=\overline{\phi^{n}}=\overline{\phi^{n+1}}, \quad \forall\, n \in \mathbb{N}.
	\end{split}
\end{equation}
In terms of spatial discretization, we introduce the spatial correction term $\Phi_{h}$ to enhance the accuracy order. To be specific, a straightforward Taylor expansion yields
\begin{align}
	\frac{\hat{\Phi}_{1}^{n+1}-\hat{\Phi}_{1}^{n}}{\tau}
	&=\nabla_{h}\cdot\left[\left(\frac{3}{2}\breve{\mathcal{M}}(\hat{\Phi}_{1}^{n})-\frac{1}{2}\breve{\mathcal{M}}(\hat{\Phi}_{1}^{n-1})\right)\nabla_{h}\hat{\mu}_{1}^{n+\frac{1}{2}}\right]
	+h^2(Q^{(0)})^{n+\frac{1}{2}}+O(\tau^3+h^4),
	\label{hat Phi1 1}\\
	\hat{\mu}_{1}^{n+\frac{1}{2}}
	&=\frac{G(1+\hat{\Phi}_{1}^{n+1})-G(1+\hat{\Phi}_{1}^{n})}{\hat{\Phi}_{1}^{n+1}-\hat{\Phi}_{1}^{n}}
	+\frac{G(1-\hat{\Phi}_{1}^{n+1})-G(1-\hat{\Phi}_{1}^{n})}{\hat{\Phi}_{1}^{n+1}-\hat{\Phi}_{1}^{n}}
	\label{hat Phi1 2}\\
	&\quad-\theta_{0}\left(\frac{3}{2}\hat{\Phi}_{1}^{n}-\frac{1}{2}\hat{\Phi}_{1}^{n-1}\right)
	-\varepsilon^2\Delta_{h}\left(\frac{3}{4}\hat{\Phi}_{1}^{n+1}+\frac{1}{4}\hat{\Phi}_{1}^{n-1}\right)
	+\sigma(-\Delta_{h})^{-1}\left(\hat{\Phi}_{1}^{n+\frac{1}{2}}
	-\overline{\hat{\Phi}_{1}^{n+\frac{1}{2}}}\right)
	\nonumber\\
	&\quad+\tau \,
	\Big(\ln(1+\hat{\Phi}_{1}^{n+1})-\ln(1+\hat{\Phi}_{1}^{n})
	-\ln(1-\hat{\Phi}_{1}^{n+1})+\ln(1-\hat{\Phi}_{1}^{n})\Big),
	\nonumber
\end{align}
where $\breve{\mathcal{M}}(x)=A_h \mathcal{M}(x)$.
Similarly, the spatially discrete function $Q^{(0)}$ is smooth enough in the sense that its discrete derivatives are bounded. Because of the symmetry of the centered finite difference approximation, the $O(h^3)$ truncation error term vanishes. $\Phi_{h}$ is obtained by solving the following linear PDE system:
\begin{align}
	&\partial_{t}\Phi_{h}
	=\nabla\cdot\left(\mathcal{M}(\hat{\Phi}_{1})\nabla\mathcal{V}_{h}
	+\mathcal{M}'(\hat{\Phi}_{1})\Phi_{h}\nabla\mathcal{W}_{1}\right)-Q^{(0)},
	\label{Phih 1}\\
	&\mathcal{V}_{h}
	=\frac{\Phi_{h}}{1+\hat{\Phi}_{1}}+\frac{\Phi_{h}}{1-\hat{\Phi}_{1}}
	-\theta_{0}\Phi_{h}-\varepsilon^{2}\Delta\Phi_{h}
	+\sigma(-\Delta)^{-1}\Phi_{h},\\
	&\mathcal{W}_{1}
	=\ln(1+\hat{\Phi}_{1})-\ln(1-\hat{\Phi}_{1})
	-\theta_{0}\hat{\Phi}_{1}-\varepsilon^{2}\Delta\hat{\Phi}_{1}
	+\sigma(-\Delta)^{-1}(\hat{\Phi}_{1}-\overline{\hat{\Phi}}_{1}),
	\label{Phih 3}
\end{align}
Again, trivial initial data $\Phi_{h}(\cdot,t=0)\equiv 0$ could be imposed. 

An application of a full discretization to \eqref{Phih 1} – \eqref{Phih 3} leads to
\begin{align}
	\frac{\Phi_{h}^{n+1}\!-\!\Phi_{h}^{n}}{\tau}
	&=\nabla_{h}\!\cdot\!\left[\left(\frac{3}{2}\breve{\mathcal{M}}(\hat{\Phi}_1^n)
	\!-\!\frac{1}{2}\breve{\mathcal{M}}(\hat{\Phi}_1^{n \!-\!1})\right)\nabla_{h}\mathcal{V}_{h}^{n \!+\!\frac{1}{2}}
	\!+\!\left(\frac{3}{2}\breve{\mathcal{M}}'(\hat{\Phi}_{1}^{n})\Phi_{h}^{n}
	\!-\!\frac{1}{2}\breve{\mathcal{M}}'(\hat{\Phi}_{1}^{n \!-\!1})\Phi_{h}^{n \!-\!1}\right)\nabla_{h}\mathcal{W}_{1}^{n \!+\!\frac{1}{2}}\right]
	\nonumber\\
	&\quad-(Q^{(0)})^{n+\frac{1}{2}}+O(\tau^2+h^2),
	\label{dis Phih 1}\\
	\mathcal{V}_{h}^{n+\frac{1}{2}}
	&=\frac{\Phi_{h}^{n+\frac{1}{2}}}{1+\hat{\Phi}_{1}^{n+\frac{1}{2}}}
	+\frac{\Phi_{h}^{n+\frac{1}{2}}}{1-\hat{\Phi}_{1}^{n+\frac{1}{2}}}
	-\theta_{0}\check{\Phi}_{h}^{n+\frac{1}{2}}
	-\varepsilon^{2}\Delta_{h}\hat{\Phi}_{h}^{n+\frac{1}{2}}
	+\sigma(-\Delta_{h})^{-1}\Phi_{h}^{n+\frac{1}{2}},\\
	\mathcal{W}_{1}^{n+\frac{1}{2}}
	&=\frac{G(1+\hat{\Phi}_{1}^{n+1})-G(1+\hat{\Phi}_{1}^{n})}{\hat{\Phi}_{1}^{n+1}-\hat{\Phi}_{1}^{n}}
	+\frac{G(1-\hat{\Phi}_{1}^{n+1})-G(1-\hat{\Phi}_{1}^{n})}{\hat{\Phi}_{1}^{n+1}-\hat{\Phi}_{1}^{n}}
	-\theta_{0}\left(\frac{3}{2}\hat{\Phi}_{1}^{n}-\frac{1}{2}\hat{\Phi}_{1}^{n-1}\right)	\nonumber\\
	&\quad -\varepsilon^2\Delta_{h}\left(\frac{3}{4}\hat{\Phi}_{1}^{n+1}+\frac{1}{4}\hat{\Phi}_{1}^{n-1}\right)
	+\sigma(-\Delta_{h})^{-1}\left(\hat{\Phi}_{1}^{n+\frac{1}{2}}
	-\overline{\hat{\Phi}_{1}^{n+\frac{1}{2}}}\right),
	\label{dis Phih 3}
\end{align}
where
\begin{equation*}
	\check{\Phi}_{h}^{n+\frac{1}{2}}
	=\frac{3}{2}\Phi_{h}^{n}-\frac{1}{2}\Phi_{h}^{n-1},\;
	\hat{\Phi}_{h}^{n+\frac{1}{2}}
	=\frac{3}{4}\Phi_{h}^{n+1}+\frac{1}{4}\Phi_{h}^{n-1},\;
	\Phi_{h}^{n+\frac{1}{2}}
	=\frac{1}{2}\Phi_{h}^{n+1}+\frac{1}{2}\Phi_{h}^{n},\;
	\hat{\Phi}_{1}^{n+\frac{1}{2}}
	=\frac{1}{2}\hat{\Phi}_{1}^{n+1}+\frac{1}{2}\hat{\Phi}_{1}^{n}.
\end{equation*}
Finally, combining \eqref{hat Phi1 1} - \eqref{hat Phi1 2} with \eqref{dis Phih 1} - \eqref{dis Phih 3} yields the higher-order truncation error for $\hat{\Phi}$, as given by \eqref{hat Phi 1} - \eqref{hat Phi 2}. Of course, linear expansions are extensively utilized. Based on the fact that $\hat{\Phi}^{n}\in\mathfrak{B}^{K}$ and the mass conservation property of $\hat{\Phi}^{n}$ at the continuous level, the identities
\begin{equation}\label{mass conser hat Phi 2nd}
	\begin{split}
		\frac{1}{\abs{\Omega}}\int_{\Omega}\hat{\Phi}^{n+1}\,\mathrm{d}\mathbf{x} 
		= \frac{1}{\abs{\Omega}}\int_{\Omega}\hat{\Phi}^{n}\,\mathrm{d}\mathbf{x} 
		&= \cdots 
		= \frac{1}{\abs{\Omega}}\int_{\Omega}\hat{\Phi}^{0}\,\mathrm{d}\mathbf{x}
		= \frac{1}{\abs{\Omega}}\int_{\Omega}\hat{\Phi}_1(\cdot,t=0)\,\mathrm{d}\mathbf{x} \\
		&
		= \frac{1}{\abs{\Omega}}\int_{\Omega}\hat{\Phi}_{N}(\cdot,t=0)\,\mathrm{d}\mathbf{x} 
		= \overline{\phi^{0}}=\overline{\phi^{n}}=\overline{\phi^{n+1}}, \quad \forall n \in \mathbb{N},
	\end{split}
\end{equation}
follow from arguments analogous to those used in \eqref{mass conser hat Phi1 2nd}. Therefore, the local truncation error satisfies $\overline{\omega^{n+\frac{1}{2}}}=0$, at the discrete level for any $n\in\mathbb{N}$.

\section{Proof of Lemma \ref{rough bound for nolinear}}
	For the term $\mathcal{N}_3^{n}$ associated with the artificial regularization, a Taylor expansion yields
	\begin{align}
		&\mathcal{N}_{3}^{n} 
		=\frac{1}{\zeta^{(n+1)}}\tilde{e}^{n+1},\qquad
		\zeta^{(n+1)}\, \text{between}\, 1+\hat{\Phi}^{n+1}\,\text{and}\, 1+\phi^{n+1},\nonumber\\
		&\langle\tilde{e}^{n+1},\mathcal{N}_{3}^{n}\rangle
		=\langle\tilde{e}^{n+1},\frac{1}{\zeta^{(n+1)}}\tilde{e}^{n+1}\rangle\ge 0,
		\label{N3}
	\end{align}
	where the positivity of $1+\hat{\Phi}$ and $1+\phi^{n+1}$ has been used.
	For the other term, $\mathcal{N}_4^{n}$, the following estimate holds:
		\begin{align}
		&\mathcal{N}_4^{n} 
		=-\frac{1}{\zeta^{(n)}}\tilde{e}^{n},\qquad
		\zeta^{(n)}\, \text{between}\, 1+\hat{\Phi}^{n}\,\text{and}\, 1+\phi^{n},\nonumber\\
		&\abs{\frac{1}{\zeta^{(n)}}}
		\le\max\left\{\frac{1}{1+\hat{\Phi}^{n}},\,\frac{1}{1+\phi^{n}}\right\}
		\le \frac{2}{\epsilon_{0}^{\star}},\qquad
		(\,\text{by~\eqref{Phi sep} and~\eqref{sep for phi n}}\,),\nonumber\\
		& \langle\tilde{e}^{n+1},\mathcal{N}_{4}^{n}\rangle_{\Omega}
		\ge -\frac{2}{\epsilon_{0}^{\star}}\abs{\langle\tilde{e}^{n+1},\tilde{\phi^{n}}\rangle_{\Omega}}
		\ge -\frac{2h^3}{\epsilon_{0}^{\star}}\sum_{i,j,k}\abs{\tilde{e}_{i,j,k}^{n+1}}\cdot\abs{\tilde{e}_{i,j,k}^{n}}.\label{N4}
	\end{align}
	Next, we consider the two components in the decomposition \eqref{N1n split} of $\mathcal{N}_{1}^{n}$ separately. We begin with the term $\mathcal{N}_{12}^{n}$. By the mean value theorem and properties in Lemma~\ref{G prop}, 
	\begin{equation*}
		\begin{split}
			\mathcal{N}_{12}^{n}
			&=(H_{1+\hat{\Phi}^{n+1}}^{1})'(\eta^{(n)})\cdot\tilde{e}^{n},\qquad
			\eta^{(n)}\, \text{between}\, 1+\hat{\Phi}^{n}\,\text{and}\, 1+\phi^{n},\\
			&=\frac{1}{2\xi^{(n)}}\cdot\tilde{e}^{n},\qquad
			\xi^{(n)}\, \text{between}\, 1+\hat{\Phi}^{n+1}\,\text{and}\, \eta^{(n)}.
		\end{split}
	\end{equation*}
	The above relations indicate that $\xi^{(n)}$ is between $1+\hat{\Phi}^{n+1}$, $1+\hat{\Phi}^{n}$, and $1+\phi^{n}$, and satisfies the boundedness estimate
	\begin{equation*}
		\abs{\frac{1}{\xi^{(n)}}}
		\le\max\left\{\frac{1}{1+\hat{\Phi}^{n+1}},\,\frac{1}{1+\hat{\Phi}^{n}},\,\frac{1}{1+\phi^{n}}\right\}
		\le \frac{2}{\epsilon_{0}^{\star}} = \frac{4}{\epsilon_{0}},
	\end{equation*}
	where the separation properties \eqref{Phi sep} and \eqref{sep for phi n} have been recalled. Consequently,
	\begin{equation*}
		\abs{\mathcal{N}_{12}^{n}}
		=\abs{(H_{1+\hat{\Phi}^{n+1}}^{1})'(\eta^{(n)})}\cdot\abs{\tilde{e}^{n}}
		=\abs{\frac{1}{2\xi^{(n)}}}\cdot\abs{\tilde{e}^{n}}
		\le\frac{1}{\epsilon_{0}^{\star}}\abs{\tilde{e}^{n}}.
	\end{equation*}
    Based on the above absolute value estimate, we obtain
    \begin{equation}\label{inner N12}
    	\langle\tilde{e}^{n+1},\mathcal{N}_{12}^{n}\rangle
    	\ge -\frac{1}{\epsilon_{0}^{\star}}\abs{\langle\tilde{e}^{n+1},\tilde{e}^{n}\rangle}
    	\ge-\frac{h^3}{\epsilon_{0}^{\star}}\sum_{i,j,k}\abs{\tilde{e}_{i,j,k}^{n+1}}\cdot\abs{\tilde{e}_{i,j,k}^{n}}.
    \end{equation}
    For the term $\mathcal{N}_{11}^{n}$, a similar nonlinear analysis yields
  \begin{equation}\label{N11 1}
	\begin{aligned}
		\mathcal{N}_{11}^{n} 
		&=(H_{1+\phi^{n}}^{1})'(\eta^{(n+1)})\cdot\tilde{e}^{n+1},\,\,
		\eta^{(n+1)}\, \text{between}\, 1+\hat{\Phi}^{n+1}\,\text{and}\, 1+\phi^{n+1},\\
		&=\frac{1}{2\xi^{(n+1)}}\cdot\tilde{e}^{n+1},\quad
		\xi^{(n+1)}\, \text{between}\, 1+\phi^{n}\,\text{and}\, \eta^{(n+1)}.
	\end{aligned}
\end{equation}
From \eqref{N11 1}, it follows that $\xi^{(n+1)}$ lies between $1+\phi^{n}$, $1+\phi^{n+1}$, and $1+\hat{\Phi}^{n+1}$, so that
\begin{equation*}
	\frac{1}{\xi^{(n+1)}}
	\ge\min\left\{\frac{1}{1+\hat{\Phi}^{n+1}},\,\frac{1}{1+\phi^{n+1}},\,\frac{1}{1+\phi^{n}}\right\}.
\end{equation*}
This in turn leads to
\begin{equation}\label{inner N11}
	\langle\tilde{e}^{n+1},\mathcal{N}_{11}^{n}\rangle
	=\langle\frac{1}{2\xi^{(n+1)}},(\tilde{e}^{n+1})^{2}\rangle
	=\frac{h^3}{2}\sum_{i,j,k}\frac{1}{(\xi^{(n+1)})_{i,j,k}}\cdot
	\abs{\tilde{e}^{n+1}_{i,j,k}}^2.
\end{equation}
Combining \eqref{inner N12} and \eqref{inner N11}, we obtain the following inner product estimate for $\mathcal{N}_{1}^{n}$:
\begin{equation*}
	\langle\tilde{e}^{n+1},\mathcal{N}_{1}^{n}\rangle
	\ge h^3\sum_{i,j,k}\Big(
	\frac{1}{2(\xi^{(n+1)})_{i,j,k}}\cdot\abs{\tilde{e}^{n+1}_{i,j,k}}^2
	-\frac{1}{\epsilon_{0}^{\star}}\abs{\tilde{e}^{n+1}_{i,j,k}}\cdot\abs{\tilde{e}^{n}_{i,j,k}}
	\Big).
\end{equation*}
Together with the estimates \eqref{N3} and \eqref{N4}, we finally arrive at
\begin{equation}\label{nonlinear part 1}
	\begin{split}
		&\quad\,\langle\tilde{e}^{n+1},\mathcal{N}_{1}^{n}\rangle
		+\tau\left(\langle\tilde{e}^{n+1},\mathcal{N}_{3}^{n}\rangle
		+\langle\tilde{e}^{n+1},\mathcal{N}_{4}^{n}\rangle\right)\\
		&\ge\, h^3\sum_{i,j,k}\Big(
		\frac{1}{2(\xi^{(n+1)})_{i,j,k}}\cdot\abs{\tilde{e}^{n+1}_{i,j,k}}^2
		-\frac{3}{\epsilon_{0}^{\star}}\abs{\tilde{e}^{n+1}_{i,j,k}}\cdot\abs{\tilde{e}^{n}_{i,j,k}}\Big).
	\end{split}
\end{equation}

At any fixed grid point $(i,\, j,\, k)$, Theorem~\ref{positive theo} yields
	\begin{equation*}
		0<1+\phi^{n+1}_{i,j,k}<2<2C^{\star}+3,\quad
		\text{i.e.,}\,\frac{1}{1+\phi^{n+1}_{i,j,k}}>\frac{1}{2}>\frac{1}{2C^{\star}+3}.
	\end{equation*}
	On the other hand, \eqref{the regularity assumption} and \eqref{phi infty es 2nd} provide the lower bounds
	\begin{equation*}
		\begin{aligned}
			&\frac{1}{1+\hat{\Phi}^{n+1}_{i,j,k}}
			\ge \frac{1}{C^{\star}+1}>\frac{1}{2C^{\star}+3},\qquad
			\frac{1}{1+\phi^{n}_{i,j,k}}
			\ge \frac{1}{C^{\star}+2}>\frac{1}{2C^{\star}+3},\\
			&\frac{1}{(\xi^{(n+1)})_{i,j,k}}
			\ge\min\left\{\frac{1}{1+\hat{\Phi}^{n+1}_{i,j,k}},\,\frac{1}{1+\phi^{n+1}_{i,j,k}},\,\frac{1}{1+\phi^{n}_{i,j,k}}\right\}
			>\frac{1}{2C^{\star}+3}.
		\end{aligned}
	\end{equation*}
	It follows that
	\begin{equation*}
		\begin{aligned}
			&\quad\frac{1}{2(\xi^{(n+1)})_{i,j,k}}\cdot\abs{\tilde{e}^{n+1}_{i,j,k}}^2
			-\frac{3}{\epsilon_{0}^{\star}}\abs{\tilde{e}^{n+1}_{i,j,k}}\cdot\abs{\tilde{e}^{n}_{i,j,k}}\\
			&\ge \frac{1}{4C^{\star}+6}\abs{\tilde{e}^{n+1}_{i,j,k}}^2
			-\frac{1}{8C^{\star}+12}\abs{\tilde{e}^{n+1}_{i,j,k}}^2
			-\frac{9(2C^{\star}+3)}{(\epsilon_{0}^{\star})^2}\abs{\tilde{e}^{n}_{i,j,k}}^2\\
			&\ge\frac{1}{8C^{\star}+12}\abs{\tilde{e}^{n+1}_{i,j,k}}^2
			-\frac{9(2C^{\star}+3)}{(\epsilon_{0}^{\star})^2}\abs{\tilde{e}^{n}_{i,j,k}}^2.
		\end{aligned}
	\end{equation*}
	Substituting this into \eqref{nonlinear part 1} leads to the estimate \eqref{5.3.1} as stated in Lemma~\ref{rough bound for nolinear} :
	\begin{equation*}
		\begin{split}
			&\quad\langle\tilde{e}^{n+1},\mathcal{N}_{1}^{n}\rangle
			+\tau\langle\tilde{e}^{n+1},\mathcal{N}_{3}^{n}+\mathcal{N}_{4}^{n}\rangle\\
			&\ge h^3\sum_{i,j,k}\Big(
			\frac{1}{8C^{\star}+12}\abs{\tilde{e}^{n+1}_{i,j,k}}^2
			-\frac{9(2C^{\star}+3)}{(\epsilon_{0}^{\star})^2}\abs{\tilde{e}^{n}_{i,j,k}}^2\Big)\\
			&=:\tilde{C_{5}}\|\tilde{e}^{n+1}\|_{2}^{2}
			-\tilde{C_{4}}\|\tilde{e}^{n}\|_{2}^{2},
		\end{split}
	\end{equation*}
	where 
	\begin{equation}
		\tilde{C_{4}}=\frac{9(2C^{\star}+3)}{(\epsilon_{0}^{\star})^2},\qquad
		\tilde{C_{5}}=\frac{1}{8C^{\star}+12}.
	\end{equation}
	
	The second inequality in Lemma~\ref{rough bound for nolinear}, namely \eqref{4.1.2}, can be established by an analogous argument and is therefore omitted.

\section{Proof of Lemma \ref{nonlinear es in refine 2nd}}
Decomposition \eqref{N1n split} for $\mathcal{N}_{1}^{n}$ is still valid. Firstly,
\begin{equation}\label{eta xi 1}
	\begin{split}
		D_{x}(\mathcal{N}_{11}^{n})_{i+\frac{1}{2},j,k}
		&=\frac{1}{h}(H_{1+\phi^n_{i+1,j,k}}^1)'(\eta_{i+1,j,k}^{(n+1)})\,\tilde{e}_{i+1,j,k}^{n+1}
		-\frac{1}{h}(H_{1+\phi^n_{i,j,k}}^1)'(\eta_{i,j,k}^{(n+1)})\,\tilde{e}_{i,j,k}^{n+1}\\
		&=\frac{1}{2\xi_{i+1,j,k}^{(n+1)}}\frac{\tilde{e}_{i+1,j,k}^{n+1}}{h}
		-\frac{1}{2\xi_{i,j,k}^{(n+1)}}\frac{\tilde{e}_{i,j,k}^{n+1}}{h}\\
		&=\frac{1}{2h}(\frac{1}{\xi_{i+1,j,k}^{(n+1)}}-\frac{1}{\xi_{i,j,k}^{(n+1)}})\,\tilde{e}_{i+1,j,k}^{n+1}
		+\frac{1}{2\xi_{i,j,k}^{(n+1)}}D_{x}\tilde{e}_{i+\frac{1}{2},j,k}^{n+1},
	\end{split}
\end{equation}
where $\eta^{(n+1)}_{\vec{\alpha}}$ lies between $1+\hat{\Phi}^{n+1}_{\vec{\alpha}}$ and $1+\phi^{n+1}_{\vec{\alpha}}$, $\xi^{(n+1)}_{\vec{\alpha}}$ lies between $1+\phi^{n}_{\vec{\alpha}}$ and $\eta^{(n+1)}_{\vec{\alpha}}$, and $\vec{\alpha}=(i,j,k),\,(i+1,j,k)$. Therefore, a pointwise estimate can be established,
\begin{equation}\label{DxN11n}
	\abs{D_{x}(\mathcal{N}_{11}^{n})_{i+\frac{1}{2},j,k}}
	\le\abs{\frac{1}{2h}(\frac{1}{\xi_{i+1,j,k}^{(n+1)}}-\frac{1}{\xi_{i,j,k}^{(n+1)}})}
	\cdot\abs{\tilde{e}_{i+1,j,k}^{n+1}}
	+\abs{\frac{1}{2\xi_{i,j,k}^{(n+1)}}}\cdot\abs{D_{x}\tilde{e}_{i+\frac{1}{2},j,k}^{n+1}}.
\end{equation}
For the first term on the right-hand side of equation~\eqref{DxN11n}, combining~\eqref{es for discrete temporal derivative} with~\eqref{eta xi 1} and applying the triangle inequality repeatedly yields
\begin{equation}
	\max\left\{	\abs{\xi_{i+1,j,k}^{(n+1)}-(1+\hat{\Phi}_{i+1,j,k}^{n+1})},\,
	\abs{\xi_{i+1,j,k}^{(n+1)}-(1+\hat{\Phi}_{i+1,j,k}^{n})}\right\}
	\le (C^{\star}+3)\tau.
\end{equation}
From the separation properties of $\hat{\Phi}$ and $\phi$, it further follows that
\begin{equation}\label{xii1}
	\begin{split}
		&\quad\abs{\frac{1}{2\xi_{i+1,j,k}^{(n+1)}}
			-\frac{1}{(1+\hat{\Phi}_{i+1,j,k}^{n+1})+(1+\hat{\Phi}_{i+1,j,k}^{n})}}\\
		&\le \frac{\abs{(1+\hat{\Phi}_{i+1,j,k}^{n+1})-\xi_{i+1,j,k}^{(n+1)}}
			+\abs{(1+\hat{\Phi}_{i+1,j,k}^{n})-\xi_{i+1,j,k}^{(n+1)}}}
		{\abs{2\xi_{i+1,j,k}^{(n+1)}\big((1+\hat{\Phi}_{i+1,j,k}^{n+1})+(1+\hat{\Phi}_{i+1,j,k}^{n})\big)}}\\
		&\le 2(C^{\star}+3)\,\tau \,\abs{\frac{1}{2\xi_{i+1,j,k}^{(n+1)}}}\cdot
		\abs{\frac{1}{(1+\hat{\Phi}_{i+1,j,k}^{n+1})+(1+\hat{\Phi}_{i+1,j,k}^{n})}}\\
		&\le M^{(0)}\tau(\epsilon_{0}^{\star})^{-2},
	\end{split}
\end{equation}
where $M^{(0)}:=2(C^{\star}+3)$. In the derivation, equations
\begin{align*}
	&\abs{\frac{1}{\xi_{i+1,j,k}^{(n+1)}}}
	\le\max\{\frac{1}{1+\phi_{i+1,j,k}^{n}},\,\frac{1}{1+\phi_{i+1,j,k}^{n+1}},\,\frac{1}{1+\hat{\Phi}_{i+1,j,k}^{n+1}}\}
	\le\frac{2}{\epsilon_{0}^{\star}},\\
	&\frac{(1+\hat{\Phi}^{n+1})+(1+\hat{\Phi}^{n})}{2}
	=1+\hat{\Phi}^{n+\frac{1}{2}}+O(\tau^2)
	\ge\epsilon_{0}^{\star}+O(\tau^2)\ge\frac{\epsilon_{0}^{\star}}{2},
\end{align*}
have been used. Similar bound can be derived for $\xi_{i,j,k}^{(n+1)}$; technical details are omitted for the sake of brevity.
\begin{equation}\label{xii}
	\begin{aligned}
		&\max\left\{\abs{\xi_{i,j,k}^{(n+1)}-(1+\hat{\Phi}_{i,j,k}^{n+1})},\,
		\abs{\xi_{i,j,k}^{(n+1)}-(1+\hat{\Phi}_{i,j,k}^{n})}\right\}
		\le (C^{\star}+3)\tau,\\
		&\abs{\frac{1}{(1+\hat{\Phi}_{i,j,k}^{n+1})+(1+\hat{\Phi}_{i,j,k}^{n})}
			-\frac{1}{2\xi_{i,j,k}^{(n+1)}}}
		\le M^{(0)}\tau(\epsilon_{0}^{\star})^{-2}.
	\end{aligned}
\end{equation}

Simultaneously, we observe that
\begin{equation}\label{frac minus frac}
	\begin{split}
		&\quad\abs{\frac{1}{(1+\hat{\Phi}_{i+1,j,k}^{n+1})+(1+\hat{\Phi}_{i+1,j,k}^{n})}
			-\frac{1}{(1+\hat{\Phi}_{i,j,k}^{n+1})+(1+\hat{\Phi}_{i,j,k}^{n})}
		}\cdot\frac{1}{h}\\
		&=\abs{\frac{(\hat{\Phi}_{i+1,j,k}^{n+1}-\hat{\Phi}_{i,j,k}^{n+1})
				+(\hat{\Phi}_{i+1,j,k}^{n}-\hat{\Phi}_{i,j,k}^{n})}
			{2\left(1+\hat{\Phi}_{i+1,j,k}^{n+\frac{1}{2}}+O(\tau^{2})\right)\cdot
				2\left(1+\hat{\Phi}_{i,j,k}^{n+\frac{1}{2}}+O(\tau^{2})\right)}}\cdot\frac{1}{h}\\
		&\le(\epsilon_{0}^{\star})^{-2}\left(\abs{D_{x}\hat{\Phi}_{i+\frac{1}{2},j,k}^{n+1}}
		+\abs{D_{x}\hat{\Phi}_{i+\frac{1}{2},j,k}^{n}}\right).
	\end{split}
\end{equation}
Integrating the estimates~\eqref{xii1} - \eqref{frac minus frac} yields
\begin{equation*}
	\begin{split}
		&\quad\frac{1}{h}\cdot\abs{\frac{1}{2\xi_{i+1,j,k}^{(n+1)}}-\frac{1}{2\xi_{i,j,k}^{(n+1)}}}\\
		&\le \frac{1}{h}\cdot\abs{\frac{1}{2\xi_{i+1,j,k}^{(n+1)}}
			-\frac{1}{(1+\hat{\Phi}_{i+1,j,k}^{n+1})+(1+\hat{\Phi}_{i+1,j,k}^{n})}}\\
		&\quad+\frac{1}{h}\cdot\abs{\frac{1}{(1+\hat{\Phi}_{i+1,j,k}^{n+1})+(1+\hat{\Phi}_{i+1,j,k}^{n})}
			-\frac{1}{(1+\hat{\Phi}_{i,j,k}^{n+1})+(1+\hat{\Phi}_{i,j,k}^{n})}
		}\\
		&\quad+\frac{1}{h}\cdot\abs{\frac{1}{(1+\hat{\Phi}_{i,j,k}^{n+1})+(1+\hat{\Phi}_{i,j,k}^{n})}
			-\frac{1}{2\xi_{i,j,k}^{(n+1)}}}\\
		&\le 2C_2M^{(0)}(\epsilon_{0}^{\star})^{-2}
		+(\epsilon_{0}^{\star})^{-2}\left(\abs{D_{x}\hat{\Phi}_{i+\frac{1}{2},j,k}^{n+1}}
		+\abs{D_{x}\hat{\Phi}_{i+\frac{1}{2},j,k}^{n}}\right).
	\end{split}
\end{equation*}
$C_2$ is the linear constraint constant. Returning to~\eqref{eta xi 1} and, with the aid of the regularity assumption~\eqref{the regularity assumption}, we obtain the estimate for the derivative of $\mathcal{N}_{11}^{n}$ in $x$-direction:
\begin{equation*}
	\|D_{x}\mathcal{N}_{11}^{n}\|_{2}
	\le 2(\epsilon_{0}^{\star})^{-2}\hat{C}\|\tilde{e}^{n+1}\|_{4}
	+(\epsilon_{0}^{\star})^{-1}\|D_{x}\tilde{e}^{n+1}\|_{2},\quad
	\hat{C}:=C_2M^{(0)}+C^{\star}.
\end{equation*}
Analogous conclusions can be derived in the $y$- and $z$-directions,
\begin{align*}
	&\|D_{y}\mathcal{N}_{11}^{n}\|_{2}
	\le2(\epsilon_{0}^{\star})^{-2}\hat{C}\|\tilde{e}^{n+1}\|_{4}
	+(\epsilon_{0}^{\star})^{-1}\|D_{y}\tilde{e}^{n+1}\|_{2},\\
	&\|D_{z}\mathcal{N}_{11}^{n}\|_{2}
	\le2(\epsilon_{0}^{\star})^{-2}\hat{C}\|\tilde{e}^{n+1}\|_{4}
	+(\epsilon_{0}^{\star})^{-1}\|D_{z}\tilde{e}^{n+1}\|_{2}.
\end{align*}
Combining the three directional results gives the upper bound for $\mathcal{N}_{11}^{n}$:
\begin{equation}\label{N11 norm 2}
	\|\nabla_{h}\mathcal{N}_{11}^{n}\|_{2}
	\le 2\sqrt{3}(\epsilon_{0}^{\star})^{-2}\hat{C}\|\tilde{e}^{n+1}\|_{4}
	+(\epsilon_{0}^{\star})^{-1}\|\nabla_{h}\tilde{e}^{n+1}\|_{2}.
\end{equation}
Using the same approach for $\mathcal{N}_{12}^{n}$, we have
\begin{equation}\label{N12 norm 2}
	\|\nabla_{h}\mathcal{N}_{12}^{n}\|_{2}
	\le 2\sqrt{3}(\epsilon_{0}^{\star})^{-2}\hat{C}\|\tilde{e}^{n}\|_{4}
	+(\epsilon_{0}^{\star})^{-1}\|\nabla_{h}\tilde{e}^{n}\|_{2}.
\end{equation}
Estimates~\eqref{N11 norm 2} and \eqref{N12 norm 2} together establish the discrete gradient norm bound for $\mathcal{N}_{1}^{n}$:
\begin{equation*}
	\begin{split}
		\|\nabla_{h}\mathcal{N}_{1}^{n}\|_2^2
		\le&\, 48(\epsilon_{0}^{\star})^{-4}\hat{C}^2\left(\|\tilde{e}^{n+1}\|_{4}^2+\|\tilde{e}^{n}\|_{4}^2\right)\\
		&+4(\epsilon_{0}^{\star})^{-2}\left(\|\nabla_{h}\tilde{e}^{n+1}\|_{2}^2+\|\nabla_{h}\tilde{e}^{n}\|_{2}^2\right).
	\end{split}
\end{equation*}
By applying the Poincar{\'e} inequality and the discrete Sobolev inequalities (Lemma~\ref{norm4}), we deduce the first assertion of Lemma~\ref{nonlinear es in refine 2nd}:
\begin{equation*}
	\|\nabla_{h}\mathcal{N}_{1}^{n}\|_{2}^{2}
	\le C^{(1)}\left(\|\nabla_{h}\tilde{e}^{n+1}\|_{2}^{2}+\|\nabla_{h}\tilde{e}^{n}\|_{2}^{2}\right),
\end{equation*}
where $C^{(1)}:=48(\epsilon_{0}^{\star})^{-4}\hat{C}^{2}C+4(\epsilon_{0}^{\star})^{-2}$. The estimate for $\nabla_{h}\mathcal{N}_{2}^{n}$ follows analogously.

Based on the estimates~\eqref{es for norm4 2nd} - \eqref{sep for phi n1} and the discussion in \cite{guo2024convergence}, the remaining part of Lemma~\ref{nonlinear es in refine 2nd} concerning $\nabla_{h}\mathcal{N}_{k}^{n}$, $k=3,\,4,\,5,\,6$, also holds:
\begin{align*}
	&\|\nabla_{h}\mathcal{N}_{m}^{n}\|_{2}
	\le 2(\epsilon_{0}^{\star})^{-1}\|\nabla_{h}\tilde{e}^{n+1}\|_{2}
	+2\sqrt{3}(\epsilon_{0}^{\star})^{-2}(C^{\star}+\frac{1}{2})\|\tilde{e}^{n+1}\|_{4},\quad m=3,\, 5,\\
	&\|\nabla_{h}\mathcal{N}_{l}^{n}\|_{2}
	\le 2(\epsilon_{0}^{\star})^{-1}\|\nabla_{h}\tilde{e}^{n}\|_{2}
	+2\sqrt{3}(\epsilon_{0}^{\star})^{-2}(C^{\star}+\frac{1}{2})\|\tilde{e}^{n}\|_{4},\quad l=4,\, 6.   
\end{align*}
A repeated application of the basic inequality, together with the discrete Sobolev and Poincar{\'e} inequalities, yields
\begin{align*}
	&\|\nabla_{h}\mathcal{N}_{m}^{n}\|_{2}^{2}
	\le C^{(2)}\|\nabla_{h}\tilde{e}^{n+1}\|_{2}^{2},\quad m=3,\, 5,\\
	&\|\nabla_{h}\mathcal{N}_{l}^{n}\|_{2}^{2}
	\le C^{(2)}\|\nabla_{h}\tilde{e}^{n}\|_{2}^{2},\quad l=4,\, 6,
\end{align*}
where $C^{(2)}=8(\epsilon_{0}^{\star})^{-2}+24(\epsilon_{0}^{\star})^{-4}C(C^{\star}+\frac{1}{2})^{2}$. 

This completes the proof of Lemma~\ref{nonlinear es in refine 2nd}.

\end{document}